\documentclass{article}[12pt]
\usepackage[title]{appendix}
\usepackage{tikz}
\usetikzlibrary {intersections}
\usetikzlibrary {intersections,through}
\usetikzlibrary{positioning} 
\usepackage{xcolor} 
\usepackage{mathrsfs}
\usepackage{mathtools}
\usepackage{graphicx}
\usepackage{epstopdf}
\usepackage{stmaryrd}
 \usepackage{subfigure}
\usepackage{float}
\usepackage{amsmath}
    \usepackage{bm}
\usepackage{amsfonts,amssymb}
   \usepackage{natbib}
  \usepackage{dsfont}
\usepackage{amsfonts}
\usepackage{mathrsfs}
  \usepackage{pifont}
\numberwithin{equation}{section}
 \usepackage{amsthm}
 \newtheorem{theorem}{Theorem}[section]
\newtheorem{lemma}[theorem]{Lemma}

\newtheorem{remark}[theorem]{Remark}
\newtheorem{assumption}[theorem]{Assumption}
 \usepackage{amsmath}
 \usepackage[hidelinks]{hyperref}
 \usepackage{multirow}

\begin{document}
\title{An immersed $CR$-$P_0$ element for Stokes interface problems and the optimal convergence analysis
}
\author{
Haifeng Ji\footnotemark[1]
\qquad
Feng Wang\footnotemark[2] \qquad
Jinru Chen\footnotemark[3] \qquad
Zhilin Li\footnotemark[4] 
}
\footnotetext[1]{School of Science, Nanjing University of Posts and Telecommunications, Nanjing 210023, China  (hfji@njupt.edu.cn)}
\footnotetext[2]{Jiangsu Key Laboratory for NSLSCS, School of Mathematical Sciences, Nanjing Normal University, Nanjing 210023, China  (fwang@njnu.edu.cn)}
\footnotetext[3]{Jiangsu Key Laboratory for NSLSCS, School of Mathematical Sciences, Nanjing Normal University, Nanjing 210023, China,\\
School of Mathematics and Information Technology, Jiangsu Second Normal University, Nanjing 211200, China  (jrchen@njnu.edu.cn)}
\footnotetext[4]{Department of Mathematics, North Carolina State University, Raleigh, NC 27695, USA  (zhilin@math.ncsu.edu)}

\date{}
\maketitle

\begin{abstract}
This paper presents and analyzes an immersed finite element (IFE) method for solving Stokes interface problems with a piecewise constant viscosity coefficient that has a jump across the interface. In the method, the triangulation does not need to fit the interface and the IFE spaces are constructed from the traditional $CR$-$P_0$ element with modifications near the interface according to the interface jump conditions. We prove that the IFE basis functions are unisolvent on arbitrary triangles without any angle conditions  and the IFE spaces have the optimal approximation capabilities, although the proof is challenging due to the coupling of the velocity and the pressure. The stability and the optimal error estimates of the proposed IFE method are also derived rigorously. The constants in the error estimates are shown to be independent of the interface location relative to the triangulation. Numerical examples are provided to verify the theoretical results.

\end{abstract}

\textbf{keyword}: Stokes equations, interface, immersed finite element, unfitted mesh, two-phase flow, error estimates

\textbf{AMS subject classification.} 65N15, 65N30, 65N12, 76D07

\section{Introduction}
In this paper we are interested in  designing and analyzing immersed finite element (IFE) methods for solving Stokes interface problems, also known as two-phase Stokes problems. 
Let $\Omega\subset\mathbb{R}^2$  be a bounded domain with a convex polygonal boundary $\partial\Omega$,  and $\Gamma$  be a $C^2$-smooth interface immersed in $\Omega$.  Without loss of generality, we assume that  $\Gamma$  divides $\Omega$ into two phases $\Omega^+$ and $\Omega^-$ such that $\Gamma=\partial\Omega^-$; see Figure~\ref{interfacepb} for an illustration. 
The Stokes interface problem reads: given a body force $\mathbf{f}\in L^2(\Omega)^2$ and a piecewise constant viscosity $\mu|_{\Omega^\pm}=\mu^\pm>0$,  find a velocity $\mathbf{u}$ and a pressure $p$ such that 
\begin{align}
-\nabla\cdot (2\mu\boldsymbol{\epsilon}(\mathbf{u})) +\nabla p&=\mathbf{f}  \qquad\mbox{in } \Omega^+\cup\Omega^-,\label{originalpb1}\\
\nabla\cdot \mathbf{u}&=0\qquad\mbox{in } \Omega,\label{originalpb2}\\
[\sigma(\mu,\mathbf{u},p)\mathbf{n}]_\Gamma&=\mathbf{0}\qquad\mbox{on } \Gamma,\label{jp_cond1}\\
[\mathbf{u}]_\Gamma&=\mathbf{0}\qquad\mbox{on } \Gamma,\label{jp_cond2}\\
\mathbf{u}&=\mathbf{0} \qquad\mbox{on } \partial\Omega,\label{originalpb5}
\end{align}
where  $\boldsymbol{\epsilon}(\mathbf{u})=\frac{1}{2}(\nabla \mathbf{u}+(\nabla \mathbf{u})^T)$ is the strain tensor, $\sigma(\mu,\mathbf{u},p)=2\mu \boldsymbol{\epsilon}(\mathbf{u})-p\mathbb{I}$ is  the Cauchy stress tensor,  $\mathbb{I}$ is the identity matrix, $\mathbf{n}$ is the unit normal vector of the interface  $\Gamma$ pointing toward $\Omega^+$, and $[\mathbf{v}]_\Gamma$ stands for  the jump of a vector function $\mathbf{v}$ on the interface, i.e., $[\mathbf{v}]_\Gamma:=\mathbf{v}^+|_\Gamma-\mathbf{v}^-|_\Gamma$ with $\mathbf{v}^\pm:=\mathbf{v}|_{\Omega^\pm}$. In this paper, the superscript $\pm$ means $+$ or $-$. For simplicity, the notations of the jump $[\cdot]_\Gamma$ and the superscripts $+,-$  are also used for scalar- or matrix-valued functions. 
If the restriction $(\nabla \cdot \mathbf{u})|_{\Gamma}$ makes sense, the equation (\ref{originalpb2}) provides an additional  interface jump condition
\begin{equation}\label{jp_cond3}
[\nabla \cdot \mathbf{u}]_\Gamma=0\quad\mbox{ on }\Gamma.
\end{equation}

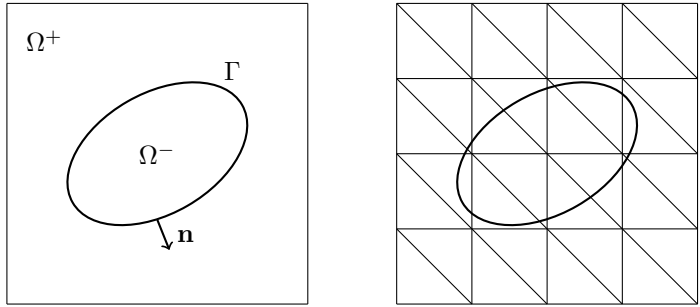
\begin{figure} [htbp]
\centering
\begin{tikzpicture}[scale=2]
\draw  (-1,-1)--(1,-1);
\draw (-1,1)--(1,1);
\draw (-1,-1)--(-1,1);
\draw (1,-1)--(1,1);
\draw[thick][rotate=120] (0,0) ellipse [x radius=0.4, y radius=0.65];
\node at (0,0) {$\Omega^-$};
\node at (-0.75,0.75) {$\Omega^+$};
\node at (0.5,0.55) {$\Gamma$};
\draw [ ->,thick] (0,-0.44)--(0+0.08,-0.44-0.2);
\node [above right] at (0.07,-0.65) {$\textbf{n}$};
\end{tikzpicture}
\qquad\quad
\begin{tikzpicture}[scale=2]
\draw  (-1,-1)--(1,-1);
\draw (-1,-0.5)--(1,-0.5);
\draw (-1,0)--(1,0);
\draw (-1,0.5)--(1,0.5);
\draw (-1,1)--(1,1);
\draw (-1,-1)--(-1,1);
\draw (-0.5,-1)--(-0.5,1);
\draw (0,-1)--(0,1);
\draw (0.5,-1)--(0.5,1);
\draw (1,-1)--(1,1);
\draw (-1,1)--(1,-1);
\draw (-0.5,1)--(1,-0.5);
\draw (0,1)--(1,0);
\draw (0.5,1)--(1,0.5);
\draw (-1,0.5)--(0.5,-1);
\draw (-1,0)--(0,-1);
\draw (-1,-0.5)--(-0.5,-1);
\draw[thick][rotate=120] (0,0) ellipse [x radius=0.4, y radius=0.65];
\end{tikzpicture}
 \caption{Left diagram: geometries of an interface problem; Right diagram: an unfitted mesh.}\label{interfacepb} 
\end{figure}

The study of the Stokes equations is motivated to solve two-phase incompressible flows,  often modeled by the Navier-Stokes equations with a discontinuous density and viscosity  across a sharp interface. The Stokes interface problem is a reasonable approximation if the inertia term is negligible. 
For interface problems, numerical methods using unfitted meshes  have attracted a lot of attention because of the relative ease of  handling  moving interfaces or complex interfaces. Unfitted meshes are generated independently of the interface,  and  can have  elements cut by the interface (called interface elements), 
which makes it challenging to design  numerical methods with  optimal convergence rates  due to the discontinuities in  the pressure and the derivatives of velocity  across the interface.

In the finite element framework, generally there are two kinds of unfitted mesh methods.  One type of the method is to enrich the traditional finite element space  by extra degrees of freedom on interface elements to capture the discontinuities (see, e.g., XFEM \cite{Fries2010}, cutFEM \cite{burman2015cutfem}, Nitsche-XFEM \cite{lehrenfeld2012nitsche},  GFEM\cite{ZHANG2020112889}).
%
%
For the Stokes interface problems, this type of methods have been developed and analyzed in \cite{hansbo2014cut, cattaneo2015stabilized,wang2015new,kirchhart2016analysis, guzman2018inf, caceres2020new,wang2019non,xiaoxiao2019stabilized}. Immersed finite element (IFE) methods \cite{li1998immersed,Li2003new} are the other type of unfitted mesh methods which modify the traditional finite element on interface elements according to interface conditions to achieve the optimal approximation capability, while keeping the degrees of freedom  unchanged. For second-order elliptic interface problems, IFE methods have been studied extensively in \cite{li2004immersed,he2012convergence,taolin2015siam,Guojcp2020,2021ji_IFE}. However, for the Stokes interface problems, 
there are much fewer works on IFE method in the literature.
 One difficulty is that the jumps of velocity and pressure are coupled  together and it is difficult to modify the velocity and the pressure finite element spaces separately. 

Although the idea of IFE methods was proposed in 1998 \cite{li1998immersed}, the first IFE method for Stokes interface problems  was developed  in 2015 by Adjerid, Chaabane, and Lin in \cite{adjerid2015immersed}, in which
 the coupling of the velocity and pressure was taken into account  in constructing  the  IFE spaces and an immersed $Q_1$-$Q_0$ discontinuous Galerkin method was proposed.  The method then was applied  to  the Stokes interface problems with moving interfaces in \cite{adjerid2019immersed},  and the idea was further developed with  immersed $CR$-$P_0$ and rotated $Q_1$-$Q_0$ elements in \cite{jones2021class}.  We also note that recently, a Taylor-Hood IFE  was constructed by a least-squares approximation in \cite{chen2021p2}. However, to the best of our knowledge, there is no theoretical analysis even for the optimal approximation capabilities of the existing IFE spaces, not mentioning the stability and the convergence of the corresponding IFE methods for Stokes interface problems. One of the major obstacles hindering the analysis is that the velocity and the pressure  are also coupled in IFE spaces.

The purposes of this paper is to provide a complete theoretical analysis of an IFE method for Stokes interface problems. We develop and analyze an IFE method based on the immersed $CR$-$P_0$ element originally proposed in \cite{jones2021class}. 
Different  from \cite{jones2021class}, we propose a new bilinear form by including additional integral terms defined on the edges cut by the interface (called interface edges) to ensure  the inf-sup stability and the optimal convergence. We show that these terms are important to prove the optimal convergence of the IFE method. In some sense, one need these terms to get an optimal error estimate on interface edges, otherwise the order of convergence is suboptimal; see the counter example in \cite{2021ji_nonconform}  for the second-order elliptic interface problems. 

Apart from the different scheme considered in this paper (compared with  \cite{jones2021class}) we mention the following other three new contributions of this paper. The first one is  about the unisolvence  (i.e., the existence and uniqueness) of the IFE basis functions. We prove the unisolvence on arbitrary triangles via a new augmented approach inspired by \cite{li2007augmented}. Note that in \cite{jones2021class} the unisolvence is only shown on isosceles right triangles by proving the invertibility of the corresponding $14\times14$ coefficient matrices. It seems that the proof is tedious and cannot be extended to arbitrary triangles. Furthermore, we also provide an explicit formula for the IFE basis functions, which is  convenient in the implementation. 
The second contribution is that we prove the optimal approximation capabilities of the IFE spaces on shape-regular triangulations, although it is challenging due to the coupling of the velocity and pressure. The proof is based on some novel auxiliary functions constructed on interface elements and a $\delta$-strip argument developed by Li et al. \cite{Li2010Optimal} for estimating errors in the region near the interface. The third contribution is the inf-sup stability result and the finite element error estimates. 
By establishing a new trace inequality for IFE functions and investigating the relations of the coupled velocity and pressure in IFE spaces, we prove that the coupled velocity and pressure IFE spaces satisfy the inf-sup condition with a  constant  independent of the meshsize and the interface location relative to the mesh.  The optimal error estimates of the proposed IFE method are also derived where the errors resulting from approximating curved interfaces by line segments are taken into consideration rigorously.  


The rest of this paper is organized as follows. In Section~\ref{preliminary}, we introduce some notations and assumptions. The IFE and corresponding IFE  method are presented in Section~\ref{sec_IFE}. Section~\ref{sec_pro_IFE} is devoted to the study of the properties of the  IFE including the unisolvence of the IFE basis functions and the optimal approximation capabilities of the IFE space. In Section~\ref{sec_error}, the stability and the optimal error estimates are proved. Section~\ref{sec_num} provides some numerical experiments. 

\section{Preliminaries and notations}\label{preliminary}
Throughout the paper we adopt the standard notation $W^k_p(\Lambda)$ for Sobolev spaces on a domain $\Lambda$ with the norm $\|\cdot\|_{W^k_p(\Lambda)}$ and the seminorm $|\cdot|_{W^k_p(\Lambda)}$. Specially, $W^k_2(\Lambda)$ is denoted by $H^{k}(\Lambda)$ with the norm  $\|\cdot\|_{H^{k}(\Lambda)}$ and the seminorm $|\cdot|_{H^{k}(\Lambda)}$. 
As usual $H_0^1(\Lambda)=\{v\in H^1(\Lambda) : v=0 \mbox{ on }\partial \Lambda\}$.
Given a domain $\Lambda$,  we define subregions
$\Lambda^\pm:=\Lambda\cap \Omega^\pm$ and  a broken space
\begin{equation*}
H^k(\Lambda^+\cup\Lambda^-):=\{v\in L^2(\Lambda) : v|_{\Lambda^\pm} \in H^k(\Lambda^\pm)\}
\end{equation*}
 equipped with the norm $\|\cdot\|_{H^k(\Lambda^+\cup\Lambda^-)}$ and the  semi-norm $|\cdot|_{H^k(\Lambda^+\cup\Lambda^-)}$ satisfying
$$
\|\cdot\|^2_{H^k(\Lambda^+\cup\Lambda^-)}=\|\cdot\|^2_{H^k(\Lambda^+)}+\|\cdot\|^2_{H^k(\Lambda^-)}, \quad|\cdot|^2_{H^k(\Lambda^+\cup\Lambda^-)}=|\cdot|^2_{H^k(\Lambda^+)}+|\cdot|^2_{H^k(\Lambda^-)}.
$$
With the usual  spaces $\mathbf{V}:=(H_0^1(\Omega))^2$  and $M:=\{q\in L^2(\Omega) : \int_\Omega q=0\}$, the weak form of the Stokes interface problem (\ref{originalpb1})-(\ref{originalpb5}) reads: find $(\mathbf{u}, p)$ in $(\mathbf{V}, M)$  such that
\begin{equation}\label{weakform}
\begin{aligned}
a(\mathbf{u}, \mathbf{v})+b(\mathbf{v}, p) &= \int_\Omega \mathbf{f}\cdot\mathbf{v}\quad &&\forall \mathbf{v}\in \mathbf{V},\\
b(\mathbf{u}, q)&=0\quad &&\forall q\in M,
\end{aligned}
\end{equation}
where
$$
a(\mathbf{u}, \mathbf{v}):=\int_\Omega2\mu\boldsymbol{\epsilon}(\mathbf{u}):\boldsymbol{\epsilon}(\mathbf{v}),\quad b(\mathbf{v}, q):=-\int_\Omega q\nabla\cdot \mathbf{v}.
$$
It is well-known that the problem (\ref{weakform}) is well-posed, that is,   there exists a unique solution $(\mathbf{u}, p)\in (\mathbf{V}, M)$ to the weak form (\ref{weakform}). For the convergence analysis we assume that the solution has a higher regularity in each sub-domain, i.e., $(\mathbf{u}, p)\in \widetilde{\boldsymbol{H}_2H_1}\cap (\mathbf{V}, M)$, where 
\begin{equation}\label{def_h2h1}
\begin{aligned}
\widetilde{\boldsymbol{H}_2H_1}:=\{(\mathbf{v}, q) ~:~  &\mathbf{v}\in (H^2(\Omega^+\cup\Omega^-))^2, ~q\in H^1(\Omega^+\cup\Omega^-), \\
&[\sigma(\mu,\mathbf{v},q)\mathbf{n}]_\Gamma=\mathbf{0}, ~[\mathbf{v}]_\Gamma=\mathbf{0},~ [\nabla\cdot \mathbf{v}]_\Gamma=0 \}.
\end{aligned}
\end{equation}

In order to solve the problem (\ref{weakform}), we consider a family of triangulations  $ \{\mathcal{T}_h\}_{h>0}$ of $\Omega$, generated independently of the interface $\Gamma$. For each element $T\in\mathcal{T}_h$, let $h_T$ denote its diameter, and define the meshsize of the triangulation $\mathcal{T}_h$ by $h=\max_{T\in\mathcal{T}_h}h_T$. We assume that $\mathcal{T}_h$ is shape-regular, i.e., for every $T$, there exists  $\varrho>0$ such that  $ h_T\leq \varrho r_T$ where $r_T$ is the diameter of the largest circle inscribed in $T$.  
Denote  $\mathcal{E}^\circ_h$ and $\mathcal{E}^b_h$ as the sets of interior and boundary edges, respectively. 
The set of all edges of the triangulation then is $\mathcal{E}_h=\mathcal{E}^\circ_h\cup \mathcal{E}^b_h$.
Since the interface $\Gamma$ is $C^2$-smooth, we can always refine the mesh near the interface to  satisfy the following assumption.
\begin{assumption}\label{assum_2}
The interface $\Gamma$ does not intersect the boundary of any element $T\in\mathcal{T}_h$ at more than two points. The interface $\Gamma$ does not intersect  the closure $\overline{e}$ for any $e\in\mathcal{E}_h$  at more than one point.
\end{assumption}

We adopt the convention that the elements $T\in\mathcal{T}_h$ and edges $e\in\mathcal{E}_h$ are open sets. The sets of interface elements and interface edges are  then defined by
\begin{equation*}
\mathcal{T}_h^\Gamma :=\{T\in\mathcal{T}_h :  T\cap \Gamma\not = \emptyset\} \quad\mbox{ and }\quad \mathcal{E}_h^\Gamma:=\{e\in \mathcal{E}_h : e \cap \Gamma\not = \emptyset\}.
\end{equation*}
The sets of non-interface elements and non-interface edges are $\mathcal{T}^{non}_h=\mathcal{T}_h\backslash\mathcal{T}_h^{\Gamma}$ and $\mathcal{E}^{non}_h=\mathcal{E}_h\backslash\mathcal{E}_h^{\Gamma}$.

On an edge $e={\rm int}(\partial T_1\cap \partial T_2)$ with $T_1, T_2\in\mathcal{T}_h$,  let $\mathbf{n}_e$ be the unit normal vector of $e$ pointing  toward $T_2$. For a piecewise smooth function $v$, we define the jump across the edge $e$ by 
$[v]_e:=v|_{T_1}-v|_{T_2}$ and the average by $\{v\}_e:=\frac{1}{2}(v|_{T_1}+v|_{T_2})$.  If $e\in\mathcal{E}^b_h$, then $\mathbf{n}_e$ is the unit normal vector of $e$ pointing  toward the outside of $\Omega$, and we define $[v]_e:=v$ and $\{v\}_e:=v$.
On a region $\Lambda$, for any $v^+\in L^1(\Lambda)$ and $v^-\in L^1(\Lambda)$, we also need the following notation 
\begin{equation*}
[\![v^\pm]\!](\mathbf{x}):=v^+(\mathbf{x})-v^-(\mathbf{x}) \quad\forall \mathbf{x}\in \Lambda.
\end{equation*}
For vector or matrix-valued functions, the notations $[\cdot]_e$, $\{\cdot\}_e$ and $[\![\cdot]\!]$ are defined analogously. Note that the difference between $[\![\cdot]\!](\mathbf{x})$ and $[\cdot]_\Gamma(\mathbf{x})$ is the range of $\mathbf{x}$.

We approximate the interface $\Gamma$ by $\Gamma_h$, which is composed of all the line segments connecting the intersection points of boundaries of interface elements and the interface.  The approximate interface $\Gamma_h$ divides $\Omega$ into two disjoint sub-domains $\Omega^+_h$ and $\Omega^-_h$ such that $\Gamma_h=\partial \Omega^-_h$.   
On each interface element $T\in\mathcal{T}_h^\Gamma$,  the discrete interface $\Gamma_h$ divides $T$ into two sub-elements:
$T^+_h:=T\cap \Omega^+_h \mbox{ and } T^-_h:=T\cap \Omega^-_h.$ For simplicity of notation, we denote 
$\Gamma_T:=\Gamma\cap T \mbox{ and } \Gamma_{h,T}:=\Gamma_h\cap T.$
Let $\textbf{n}_h(\mathbf{x})$  be the unit normal vector of $\Gamma_h$ pointing toward $\Omega^+_h$; see Figure~\ref{interface_ele} for an illustration.
The unit tangent vectors of $\Gamma_h$ and $\Gamma$ are obtained by a $90^\circ$ clockwise rotation of $\mathbf{n}_h$ and $\mathbf{n}$, i.e.,  $\mathbf{t}_h(\mathbf{x})=R_{-\pi/2}\mathbf{n}_h(\mathbf{x})$ and $\mathbf{t}(\mathbf{x})=R_{-\pi/2}\mathbf{n}(\mathbf{x})$ with  a  rotation matrix 
\begin{equation*}
R_\alpha=
\left(
\begin{aligned}
&\cos\alpha &-\sin\alpha\\
&\sin\alpha  &\cos\alpha
\end{aligned}
\right).
\end{equation*}

At the end of this section, we recall the notation $v^\pm:=v|_{\Omega^\pm}$ for a function $v$ defined on the whole domain $\Omega$.  Again the notation of the superscripts $s=+$ and $-$ may be different in the continuous and discrete cases due to some mismatched regions from the line segment approximation.  We also use  $q^\pm$ to represent $q|_{T_h^\pm}$  if no confusion can arise. Furthermore, if the function $q^s$, $s=+$ or $-$, is  a polynomial,  then the  polynomial $q^s$ is viewed as defined on the whole element $T$, unless otherwise specified. The superscripts are used for vector or matrix-valued functions similarly.

\section{The immersed $CR$-$P_0$ finite element method}\label{sec_IFE}
\subsection{The IFE space}

Let $P_k(T)$ be the set of all polynomials of degree less than or equal to $k$ on each $T\in\mathcal{T}_h$. On a non-interface element $T\in\mathcal{T}_h^{non}$, we use the standard  $CR$-$P_0$ shape function spaces \cite{crouzeix1973conforming}, i.e.,
\begin{equation*}
(\mathbf{V}_h(T), M_h(T))=(P_1(T)^2, P_0(T)). 
\end{equation*}
For every $T\in\mathcal{T}_h$, 
the local degrees of freedom are chosen as
\begin{equation}\label{loc_dof}
N_{i,T}(\mathbf{v},q):=\frac{1}{|e_i|}\int_{e_i} v_1,~N_{3+i,T}(\mathbf{v},q):=\frac{1}{|e_i|}\int_{e_i} v_2, ~i=1,2,3, 
~~~~N_7(\mathbf{v}, q):=\frac{1}{|T|}\int_Tq,
\end{equation}
where $e_i\in\mathcal{E}_h$, $i,=1,2,3$ are edges of $T$, and $v_1$ and $v_2$ are two components of $\mathbf{v}$, i.e., $\mathbf{v}=(v_1,v_2)^T$.

On an interface element $T\in\mathcal{T}_h^\Gamma$,  the shape function spaces $(\mathbf{V}_h(T), M_h(T))$ do not have the optimal approximation capabilities due to the interface jumps  (\ref{jp_cond1}), (\ref{jp_cond2}) and (\ref{jp_cond3}). The shape function spaces need to be modified according to these interface jump conditions.  
Given $\mathbf{v}^\pm\in P_1(T)^2 $ and $q^\pm\in P_0(T)$, we define the following discrete interface jump conditions 
\begin{align}
&[\![\sigma(\mu^\pm,\mathbf{v}^\pm,q^\pm)\mathbf{n}_h]\!]=\mathbf{0},\label{jp_cond_IFE1}\\
&[\![\mathbf{v}^\pm]\!]|_{\Gamma_{h,T}}=\mathbf{0}~  (\mbox{or, equivalently,}~  [\![\mathbf{v}^\pm]\!](\mathbf{x}_T)=\mathbf{0}, [\![\nabla\mathbf{v}^\pm\mathbf{t}_h]\!]=\mathbf{0}),\label{jp_cond_IFE2}\\
&[\![\nabla\cdot\mathbf{v}^\pm]\!]=0,\label{jp_cond_IFE3}
\end{align}
where $\mathbf{x}_T$ is a point on $\Gamma_{h,T}\cap \Gamma_{T}$.
The immersed $CR$-$P_0$ shape function space is then defined by (see \cite{jones2021class})
\begin{equation}\label{def_local_IFE}
\begin{aligned}
\mathbf{V}M_h^{IFE}(T)=\{(\mathbf{v},q) ~:&~ \mathbf{v}|_{T_h^\pm}=\mathbf{v}^\pm|_{T_h^\pm},~ \mathbf{v}^\pm\in P_1(T)^2,~ q|_{T_h^\pm}=q^\pm|_{T_h^\pm},~ q^\pm\in P_0(T),  \\
&~(\mathbf{v}^\pm, q^\pm)  \mbox{ satisfies (\ref{jp_cond_IFE1})-(\ref{jp_cond_IFE3})}  \}.
\end{aligned}
\end{equation}

\begin{remark}\label{remark_14by14}
Note that $\mathbf{v}^\pm$ and $q^\pm$ have fourteen parameters. It is easy to check that (\ref{jp_cond_IFE1}) provides two constraints, (\ref{jp_cond_IFE2}) provides four constraints, and (\ref{jp_cond_IFE3}) provides one constraint. Intuitively, we can expect that the functions $\mathbf{v}^\pm$ and $q^\pm$ satisfying conditions (\ref{jp_cond_IFE1})-(\ref{jp_cond_IFE3}) are  uniquely determined by the degrees of freedom  $N_{i,T}$, $i=1,...,7$ defined in (\ref{loc_dof}).   The authors in \cite{jones2021class} proved the unisolvence of IFE basis functions on isosceles right triangles. In Subsection~\ref{sec_uni}, we will prove that the unisolvence holds on arbitrary triangles without any angle conditions.
\end{remark}


The   global IFE space is  defined by
\begin{equation*}
\begin{aligned}
\mathbf{V}M_h^{IFE}=\left \{ \frac{\null}{\null} (\mathbf{v},q) ~:  \right. &~\mathbf{v}|_T\in \mathbf{V}_h(T),~ q|_T\in M_h(T)~ ~\forall T\in\mathcal{T}_h^{non},  \\
 & \left.~(\mathbf{v}|_T, q|_T)\in \mathbf{V}M_h^{IFE}(T)~~\forall T\in\mathcal{T}_h^\Gamma,~~\int_{e}[\mathbf{v}]_e=\mathbf{0} ~~\forall e\in\mathcal{E}^\circ_h \right \},
\end{aligned}
\end{equation*}
in which the velocity and pressure are  coupled. 
We also define a subspace of $\mathbf{V}M_h^{IFE}$ to take into account the boundary condition of velocity and the constraint of  pressure
\begin{equation*}
\mathbf{V}M_{h,0}^{IFE}= \left \{ (\mathbf{v},q) ~:~ (\mathbf{v},q)\in \mathbf{V}M_h^{IFE},~ \int_e\mathbf{v}=\mathbf{0} ~~\forall e\in\mathcal{E}_h^b, ~\int_\Omega q=0 \right \}.
\end{equation*}

\subsection{The IFE method}
To make the method easy for implementation,  we define a discrete viscosity by $\mu_h|_{\Omega_h^\pm}=\mu^\pm$.
In other words, the viscosity is adjusted in the mismatched small area due to the line segment approximation. 
The immersed $CR$-$P_0$ finite element method for the Stokes interface problem (\ref{originalpb1})-(\ref{originalpb5}) reads: find $(\mathbf{u}_h, p_h)\in  \mathbf{V}M_{h,0}^{IFE}$ such that 
\begin{equation}\label{IFE_method}
A_h(\mathbf{u}_h,p_h; \mathbf{v}_h,q_h )=\int_\Omega \mathbf{f}\cdot \mathbf{v}_h \qquad \forall (\mathbf{v}_h, q_h)\in  \mathbf{V}M_{h,0}^{IFE}.
\end{equation}
Here the bilinear form is defined as follows,
\begin{equation}\label{def_Ah}
\begin{aligned}
&A_h(\mathbf{u}_h, p_h; \mathbf{v}_h, q_h ):=a_h(\mathbf{u}_h,\mathbf{v}_h)+b_h(\mathbf{v}_h,p_h)-b_h(\mathbf{u}_h,q_h)+  J_h(p_h,q_h),\\
&a_h(\mathbf{u}_h,\mathbf{v}_h):=\sum_{T\in\mathcal{T}_h}\int_T2\mu_h \boldsymbol{\epsilon}(\mathbf{u}_h):\boldsymbol{\epsilon}(\mathbf{v}_h)+\sum_{e\in\mathcal{E}_h}\frac{1}{|e|}\int_e[\mathbf{u}_h]_e\cdot[\mathbf{v}_h]_e\\
&\quad-\sum_{e\in\mathcal{E}_h^\Gamma}\int_e \left(\{ 2\mu_h\boldsymbol{\epsilon}(\mathbf{u}_h)\mathbf{n}_e\}_e\cdot[\mathbf{v}_h]_e+\delta\{ 2\mu_h\boldsymbol{\epsilon}(\mathbf{v}_h)\mathbf{n}_e\}_e\cdot[\mathbf{u}_h]_e\right) +\sum_{e\in\mathcal{E}_h^\Gamma}\frac{\eta}{|e|}\int_e[\mathbf{u}_h]_e\cdot[\mathbf{v}_h]_e,\\
&b_h(\mathbf{v}_h,q_h):=-\sum_{T\in\mathcal{T}_h}\int_T q_h\nabla\cdot\mathbf{v}_h +\sum_{e\in\mathcal{E}_h^\Gamma}\int_e\{q_h\}_e[\mathbf{v}_h\cdot\mathbf{n}_e]_e,\\
&J_h(p_h,q_h):=  \sum_{e\in  \mathcal{E}_{h}^{\Gamma} }|e|\int_{e}[p_h]_e[q_h]_e,
\end{aligned}
\end{equation}
where $\delta=\pm 1$ and $\eta\geq 0$.
When the parameter $\delta=1$, the bilinear form $a_h(\cdot,\cdot)$ is symmetric and the penalty $\eta$ should be larger enough to ensure the coercivity.  When $\delta=-1$, the bilinear form $a_h(\cdot,\cdot)$ is non-symmetric. In general, we can choose an arbitrary $\eta\geq0$ to ensure the coercivity; see Lemma~\ref{lem_cor} in Section~\ref{sec_error}.

Different from the method proposed in \cite{jones2021class}, our IFE method includes additional terms on edges. We briefly discuss the roles of these terms in the method.
The second term of $a_h(\cdot,\cdot)$ is added to control the rigid body rotations so that the Korn inequality holds for the Crouzeix-Raviart finite element space. 
The integral $\int_e \{ 2\mu_h\boldsymbol{\epsilon}(\mathbf{u}_h)\mathbf{n}_e\}_e\cdot[\mathbf{v}_h]_e$ in the third term of $a_h(\cdot,\cdot)$ appears to make the method consistent on interface edges; and correspondingly the integral $\int_e \{ 2\mu_h\boldsymbol{\epsilon}(\mathbf{v}_h)\mathbf{n}_e\}_e\cdot[\mathbf{u}_h]_e$ and the fourth term are added to make the bilinear form $a_h(\cdot,\cdot)$ coercive.
We emphasize that,  different from the traditional $CR$-$P_0$ finite element method, these integral terms on interface edges cannot be neglected and are important to ensure the optimal convergence of the IFE method. The reason is similar to that of the nonconforming IFE methods for second-order elliptic interface problems \cite{2021ji_nonconform}. The second term in $b_h(\cdot,\cdot)$ is needed also for the consistency on interface edges and the penalty term $J_h(\cdot,\cdot)$ controlling the jumps of the pressure is added to make the inf-sup condition satisfied.

\section{Properties of the IFE}\label{sec_pro_IFE}
In this section, we discuss some properties of the proposed IFE. To begin with, we make some preparations. Denote $\mbox{dist}(\mathbf{x},\Gamma)$ as  the distance between a point $\mathbf{x}$ and the interface $\Gamma$, and  $U(\Gamma,\delta)=\{\mathbf{x}\in\mathbb{R}^2: \mbox{dist}(\mathbf{x},\Gamma)< \delta\}$  as the neighborhood of $\Gamma$ of thickness $\delta$. 
Define the meshsize of $\mathcal{T}_h^\Gamma$ by
$h_\Gamma:=\max_{T\in\mathcal{T}_h^\Gamma}h_T.$
It is obvious that $h_\Gamma\leq h$ and $\bigcup_{T\in\mathcal{T}_h^\Gamma} T\subset U(\Gamma,h_\Gamma)$. We need the  signed distance function  $\rho(\mathbf{x})$ defined by $\rho(\mathbf{x})|_{\Omega^\pm}=\pm \mbox{dist}(\mathbf{x},\Gamma)$.
As we assume that $\Gamma\in C^2$, there exists a constant $\delta_0>0$ such that $\rho(\mathbf{x})\in C^2(U(\Gamma,\delta_0))$ (see \cite{foote1984regularity}).  In the following analysis, we make the following assumption.
\begin{assumption}\label{assumption_delta}
We assume that $h_\Gamma<\delta_0$ so that  $T\subset U(\Gamma,\delta_0)$ for all interface elements $T\in\mathcal{T}_h^\Gamma$. 
\end{assumption}


Using the signed distance function $\rho(\mathbf{x})$, we can evaluate the unit normal and tangent vectors of the interface as $\textbf{n}(\mathbf{x})=\nabla \rho$ and $\textbf{t}(\mathbf{x}) =R_{-\pi/2}\nabla \rho$ (see \cite{foote1984regularity}). The functions $\textbf{n}(\mathbf{x})$ and $\textbf{t}(\mathbf{x})$ are well-defined in the region $U(\Gamma,\delta_0)$ and now are considered in the extended sense.
We note that the functions $\textbf{n}_h(\mathbf{x})$ and $\textbf{t}_h(\mathbf{x})$ can also be  viewed as piecewise constant vectors defined on interface elements.
Since $\Gamma$ is $C^2$-smooth, by Rolle's Theorem, there exists at least one point $\mathbf{x}^*\in\Gamma\cap T$ such that
$\textbf{n}(\mathbf{x}^*)=\textbf{n}_h(\mathbf{x}^*)$.
Since $\rho(\mathbf{x})\in C^2(U(\Gamma,\delta_0))$, we have $\textbf{n}(\mathbf{x})\in \left(C^1(\overline{T})\right)^2$.
Using  the Taylor expansion  at $\mathbf{x}^*$, we further have
\begin{equation}\label{error_nt}
\|\textbf{n}-\textbf{n}_h\|_{L^\infty(T)}\leq Ch_T,\quad \|\textbf{t}-\textbf{t}_h\|_{L^\infty(T)}=\|R_{-\pi/2}(\textbf{n}-\textbf{n}_h)\|_{L^\infty(T)}\leq Ch_T\quad \forall T\in\mathcal{T}_h^\Gamma.
\end{equation}
Define the region between the mismatched interfaces $\Gamma$ and $\Gamma_h$ as 
\begin{equation}\label{def_triangleT}
T^\triangle:=(T^-\cap T_h^+)\cup(T^+\cap  T_h^-)\qquad \forall T\in\mathcal{T}_h^\Gamma.
\end{equation}
Since $\Gamma$ is $C^2$-smooth, there exists a constant $C$ depending only on the curvature of $\Gamma$ such that 
\begin{equation}\label{triT_rela}
T^\triangle\subset U(\Gamma, Ch^2_\Gamma)\qquad \forall T\in\mathcal{T}_h^\Gamma.
\end{equation}

The following lemma presents a $\delta$-strip argument that will be used for the error estimate in the region near the interface; see Lemma 2.1 in \cite{Li2010Optimal}.
\begin{lemma}\label{strip}
Let $\delta>0$ be a sufficiently small number.  Then it holds  for any $v\in H^1(\Omega)$ that 
\begin{equation*}
\|v\|_{L^2(U(\Gamma,\delta))}\leq C\sqrt{\delta}\,  \| v\|_{H^1(\Omega)}.
\end{equation*}
Furthermore,  if $v|_{\Gamma}=0$, then there holds
\begin{equation*}
\|v\|_{L^2(U(\Gamma,\delta))}\leq C\delta \, \|\nabla v\|_{L^2(U(\Gamma,\delta))}.
\end{equation*}
\end{lemma}

We also need the following well-known extension result \cite{Gilbargbook}.
\begin{lemma}\label{lem_ext}
Suppose that $v^\pm\in H^m(\Omega^\pm)$ with $m>0$. Then there exist extensions $ v_E^\pm \in H^m(\Omega)$ such that
\begin{equation*}
v_E^\pm|_{\Omega^\pm}=v^\pm~\mbox{ and }~\|v_E^\pm\|_{H^m(\Omega)}\leq C\|v^\pm\|_{H^m(\Omega^\pm)}
\end{equation*}
for a constant $C>0$ depending only on $\Omega^\pm$.
\end{lemma}


Let $W(T):=\{v\in L^2(T): \int_{e_i}v, i=1,2,3 \mbox{ are well defined}\}$, where
$e_i$, $i=1,2,3$ are edges of  $T\in\mathcal{T}_h$.
 We define local interpolation operators $\pi_{h,T}^{CR}$, $\pi_{h,T}^0$ and $\Pi_{h,T}$ such that, for all $ v\in W(T)$ and for all $(\mathbf{v},q)\in (W(T)^2, L^2(T))$,
\begin{equation}\label{def_opera}
\begin{aligned}
&\pi_{h,T}^{CR}v\in P_1(T),~~ \int_{e_i}\pi_{h,T}^{CR}v=\int_{e_i}v,~~i=1,2,3,\\
&\pi_{h,T}^{0}q\in P_0(T),~~ \int_{T}\pi_{h,T}^{0}q=\int_{T}q,\\
&\Pi_{h,T}(\mathbf{v},q)\in (\mathbf{V}_h(T), M_h(T)), ~~N_{i,T}\left(\Pi_{h,T}(\mathbf{v},q)\right)=N_{i,T}(\mathbf{v},q),~~i=1,...,7.
\end{aligned}
\end{equation}
Let $\mathbf{v}=(v_1,v_2)^T$. Then we have
\begin{equation}\label{operaPI_1}
\Pi_{h,T}(\mathbf{v}, q)=(\boldsymbol{\pi}^{CR}_{h,T}\mathbf{v}, \pi_{h,T}^{0}q)~ \mbox{ with } ~\boldsymbol{\pi}^{CR}_{h,T}\mathbf{v}:=(\pi_{h,T}^{CR}v_1,\pi_{h,T}^{CR}v_2)^T.
\end{equation}

For an interface element $T\in\mathcal{T}_h^\Gamma$, define a local IFE interpolation operator $\Pi_{h,T}^{IFE}: (W(T)^2, L^2(T))\rightarrow \mathbf{V}M_h^{IFE}(T)$ such that
\begin{equation}\label{local_PIife}
N_{i,T}\left(\Pi_{h,T}^{IFE}(\mathbf{v},q)\right)=N_{i,T}(\mathbf{v},q),~~i=1,...,7,~~~\forall (\mathbf{v},q)\in (W(T)^2, L^2(T)).
\end{equation}
Now the global IFE interpolation operator $\Pi_{h}^{IFE}: (H^1(\Omega)^2,L^2(\Omega))\rightarrow \mathbf{V}M_h^{IFE}$  is defined by  
\begin{equation}\label{Pi_IFE_1}
\forall (\mathbf{v},q)\in (H^1(\Omega)^2, L^2(\Omega)),\qquad
 \left(\Pi_h^{IFE}(\mathbf{v},q)\right)|_{T}=\left\{
\begin{aligned}
&\Pi_{h,T}^{IFE}(\mathbf{v},q)\quad&&\mbox{ if } ~T\in\mathcal{T}_h^\Gamma,\\
&\Pi_{h,T} (\mathbf{v},q)&&\mbox{ if }~ T\in\mathcal{T}_h^{non}.
\end{aligned}\right.
\end{equation}
We use $\Pi_{\mathbf{v},q}^{IFE}\mathbf{v}$ and $\Pi_{\mathbf{v},q}^{IFE} q$ to represent the velocity and pressure of  $\Pi_h^{IFE}(\mathbf{v},q)$, i.e.,
\begin{equation}\label{Pi_IFE_2}
\Pi_h^{IFE}(\mathbf{v},q)=\left (\Pi_{\mathbf{v},q}^{IFE}\mathbf{v}, \Pi_{\mathbf{v},q}^{IFE} q \right ). 
\end{equation}
Note that the subscript of $\Pi_{\mathbf{v},q}^{IFE}$ means that the interpolation operator may depend not only on $\mathbf{v}$ but also on $q$ since the velocity and pressure are coupled in the IFE space; see Remark~\ref{basiseq0_affect} for details. 

We can introduce the standard CR basis functions  by
\begin{equation}\label{cr_basis}
\lambda_{i,T}\in P_1(T),~~\frac{1}{|e_j|}\int_{e_j}\lambda_{i,T}=\delta_{ij} ~(\mbox{the Kronecker function}),~~ i,j=1,2,3,
\end{equation}
and the standard  $CR$-$P_0$ finite element basis functions  by
\begin{equation}\label{crp0_basis}
(\boldsymbol{\phi}_{i,T}, \varphi_{i,T})\in \left(\mathbf{V}_h,M_h(T)\right), \quad N_{j,T}(\boldsymbol{\phi}_{i,T}, \varphi_{i,T})=\delta_{ij},\quad  \forall i,j=1,...,7,
\end{equation}
Obviously, we have 
\begin{equation}\label{crvec_basis}
\begin{aligned}
&\boldsymbol{\phi}_{i,T}=(\lambda_{i,T},0)^T,~ \boldsymbol{\phi}_{i+3,T}=(0, \lambda_{i,T})^T,~ i=1,2,3,~~~\boldsymbol{\phi}_{7,T}=\mathbf{0},\\
&\varphi_{i,T}=0, ~i=1,...,6,  ~~~\varphi_{7,T}=1.
\end{aligned}
\end{equation}

It is well-known that  the local interpolation operators $\pi_{h,T}^{CR}$, $\pi_{h,T}^0$ and $\Pi_{h,T}$ are well-defined.  However, the well-definedness of  the IFE interpolation operator $\Pi_{h,T}^{IFE}$ is not obvious. 
We need a result that the IFE shape functions in $\mathbf{V}M_h^{IFE}(T)$ can be uniquely determined by $N_{i,T}(\mathbf{v},q),~i=1,...,7$, which will be proved in the following subsection.

\subsection{The unisolvence of IFE shape functions}\label{sec_uni}
Note that for many existing IFEs developed for other interface problems,  the unisolvence of  IFE shape functions with respect to the degrees of freedom relies on the mesh assumption, i.e., the no-obtuse angle condition \cite{GuoIMA2019,2021ji_IFE,guo2020solving,ji2020IRT}.  Recently, we showed that for second-order elliptic interface problems, if integral-values on edges are used as the degrees of freedom,  then the unisolvence holds on arbitrary triangles \cite{2021ji_nonconform}.  In this paper, we are able to prove  that the unisolvence also holds on arbitrary  elements for the  immersed $CR$-$P_0$ element for Stokes interface problems as well.

Now we use a new augmented approach inspired by \cite{li2007augmented} to prove the unisolvence. Without loss of generality, we consider an interface element $T\in\mathcal{T}_h^\Gamma$ for the proof. By the definition (\ref{jp_cond_IFE1})-(\ref{def_local_IFE}), it is obvious that the space $\mathbf{V}M_h^{IFE}(T)$ is not an empty set since $(\mathbf{0},0)\in \mathbf{V}M_h^{IFE}(T)$. Given a pair of IFE functions $(\mathbf{v},q)\in \mathbf{V}M_h^{IFE}(T)$, we define $(\mathbf{v}^{J_0}, q^{J_0})$ such that 
\begin{equation}\label{def_vj0}
(\mathbf{v}^{J_0}, q^{J_0})\in (\mathbf{V}_h(T), M_h(T)),\qquad N_{i,T}(\mathbf{v}^{J_0}, q^{J_0})=N_{i,T}(\mathbf{v},q),~i=1,...,7.
\end{equation}
From (\ref{def_opera})-(\ref{operaPI_1}), we know $(\mathbf{v}^{J_0}, q^{J_0})=(\boldsymbol{\pi}^{CR}_{h,T}\mathbf{v}, \pi_{h,T}^{0}q)$.
Recalling the notation of superscripts $\pm$ described at the end of Section~\ref{preliminary},  we set $\mathbf{v}^{J_0,\pm}:=(\mathbf{v}^{J_0})^\pm$ and $q^{J_0,\pm}:=(q^{J_0})^\pm$.  
It is easy to check that 
\begin{equation}\label{vj0_jp}
[\![\sigma(1,\mathbf{v}^{J_0,\pm},q^{J_0,\pm})\mathbf{n}_h]\!]=\mathbf{0},~[\![\mathbf{v}^{J_0,\pm}]\!]|_{\Gamma_{h,T}}=\mathbf{0},~[\![\nabla\cdot\mathbf{v}^{J_0,\pm}]\!]=0.
\end{equation}
We  define  $(\mathbf{v}^{J_1}, q^{J_1})$  such that 
\begin{equation}\label{def_vj1}
\begin{aligned}
&\mathbf{v}^{J_1,\pm}:=(\mathbf{v}^{J_1})^\pm\in \mathbf{V}_h(T),~q^{J_1,\pm}:=(q^{J_1})^\pm\in M_h(T), ~ N_{i,T}(\mathbf{v}^{J_1}, q^{J_1})=0,~i=1,...,7,\\
&[\![\sigma(1,\mathbf{v}^{J_1,\pm},q^{J_1,\pm})\mathbf{n}_h]\!]=\mathbf{n}_h,~[\![\mathbf{v}^{J_1,\pm}]\!]|_{\Gamma_{h,T}}=\mathbf{0},~[\![\nabla\cdot\mathbf{v}^{J_1, \pm}]\!]=0,
\end{aligned}
\end{equation}
and $(\mathbf{v}^{J_2}, q^{J_2})$ such that 
\begin{equation}\label{def_vj2}
\begin{aligned}
&\mathbf{v}^{J_2,\pm}:=(\mathbf{v}^{J_2})^\pm\in \mathbf{V}_h(T),~q^{J_2,\pm}:=(q^{J_2})^\pm\in M_h(T), ~ N_{i,T}(\mathbf{v}^{J_2}, q^{J_2})=0,~i=1,...,7,\\
&[\![\sigma(1,\mathbf{v}^{J_2,\pm},q^{J_2,\pm})\mathbf{n}_h]\!]=\mathbf{t}_h,~[\![\mathbf{v}^{J_2,\pm}]\!]|_{\Gamma_{h,T}}=\mathbf{0},~[\![\nabla\cdot\mathbf{v}^{J_2, \pm}]\!]=0.
\end{aligned}
\end{equation}
The existence and uniqueness of $\mathbf{v}^{J_1}$ and $\mathbf{v}^{J_2}$ will be proved in Lemma~\ref{lem_vj}.
Combining (\ref{def_vj0})-(\ref{def_vj2}), we immediately have the following lemma.
\begin{lemma}
Given $(\mathbf{v},q)\in \mathbf{V}M_h^{IFE}(T)$, if we know the augmented variable 
\begin{equation*}
[\![\sigma(1,\mathbf{v}^{\pm},q^{\pm})\mathbf{n}_h]\!]=c_1\mathbf{n}_h+c_2\mathbf{t}_h,
\end{equation*}
then the pair of functions $(\mathbf{v},q)$ can be written as
\begin{equation}\label{base_decouple}
(\mathbf{v},q)=(\mathbf{v}^{J_0}+c_1\mathbf{v}^{J_1}+c_2\mathbf{v}^{J_2},q^{J_0}+c_1q^{J_1}+c_2q^{J_2}).
\end{equation}
\end{lemma}

 We want to find the augmented variable $(c_1,c_2)^T$ so that the original interface jump condition (\ref{jp_cond_IFE1}) is satisfied.
Substituting (\ref{base_decouple}) into (\ref{jp_cond_IFE1}), we have
\begin{equation}\label{c1c2_1}
\begin{aligned}
[\![\sigma(\mu^\pm,c_1\mathbf{v}^{J_1,\pm}+c_2\mathbf{v}^{J_2,\pm},c_1q^{J_1,\pm}+c_2q^{J_2,\pm})\mathbf{n}_h]\!]&=-[\![\sigma(\mu^\pm,\mathbf{v}^{J_0,\pm},q^{J_0,\pm})\mathbf{n}_h]\!]\\
&=-\sigma([\![\mu^\pm]\!],\mathbf{v}^{J_0},0)\mathbf{n}_h.
\end{aligned}
\end{equation}
To derive an equation for the augmented variable $(c_1,c_2)^T$ according to (\ref{c1c2_1}), we need the following lemma about the functions $(\mathbf{v}^{J_1}, q^{J_1})$ and $(\mathbf{v}^{J_2}, q^{J_2})$.
\begin{lemma}\label{lem_vj}
The functions $(\mathbf{v}^{J_1}, q^{J_1})$ and $(\mathbf{v}^{J_2}, q^{J_2})$ defined in (\ref{def_vj1}) and (\ref{def_vj2}) are unique and  can be constructed explicitly as
\begin{equation}\label{cons_j1j2}
\begin{aligned}
&\mathbf{v}^{J_1}=\mathbf{0},\quad q^{J_1}=z-\pi_{h,T}^0z, \quad\mathbf{v}^{J_2}=(w-\pi_{h,T}^{CR}w)\mathbf{t}_h,\quad q^{J_2}=0,
\end{aligned}
\end{equation}
with 
\begin{equation}\label{def_zw}
z(\mathbf{x})=\left\{
\begin{aligned}
&z^+=-1\quad &&\mbox{ if }\mathbf{x}\in T_h^+,\\
&z^-=0\quad &&\mbox{ if }\mathbf{x}\in T_h^-,
\end{aligned}\right.
\qquad
w(\mathbf{x})=\left\{
\begin{aligned}
&w^+=\mbox{dist}(\mathbf{x},\Gamma_{h,T}^{ext})\quad &&\mbox{ if }\mathbf{x}\in T_h^+,\\
&w^-=0\quad &&\mbox{ if }\mathbf{x}\in T_h^-,
\end{aligned}\right.
\end{equation}
where $\Gamma_{h,T}^{ext}$ is the straight line containing the line segment $\Gamma_{h,T}$.
\end{lemma}
\begin{proof}
First we introduce the following  identities about the interface jump conditions. If $\mathbf{v}^{J,\pm}\in \mathbf{V}_h(T)$ and $q^{J,\pm} \in M_h(T)$ satisfy 
\begin{equation}\label{new_add}
[\![\sigma(1,\mathbf{v}^{J,\pm},q^{J,\pm})\mathbf{n}_h]\!]=\mathbf{g},~[\![\mathbf{v}^{J,\pm}]\!]|_{\Gamma_{h,T}}=\mathbf{0},~[\![\nabla\cdot\mathbf{v}^{J,\pm}]\!]=0,
\end{equation}
then the following identities hold
\begin{equation}\label{new_add2}
\begin{aligned}
&[\![\nabla (\mathbf{v}^{J,\pm} \cdot\mathbf{n}_h)\cdot\mathbf{n}_h]\!]=0,~[\![\nabla (\mathbf{v}^{J,\pm} \cdot\mathbf{n}_h)\cdot\mathbf{t}_h]\!]=0, \\
&[\![\nabla (\mathbf{v}^{J,\pm} \cdot\mathbf{t}_h)\cdot\mathbf{n}_h]\!]=\mathbf{g}\cdot\mathbf{t}_h,~[\![\nabla (\mathbf{v}^{J,\pm} \cdot\mathbf{t}_h)\cdot\mathbf{t}_h]\!]=0, [\![q^{J,\pm}]\!]=-\mathbf{g}\cdot\mathbf{n}_h.
\end{aligned}
\end{equation}
The second and fourth identities are direct  consequences of  $[\![\mathbf{v}^{J,\pm}]\!]|_{\Gamma_{h,T}}=\mathbf{0}$. 
The other identities can be proved easily by decomposing $\mathbf{v}^{J,\pm}$ into the normal direction $\mathbf{n}_h$ and the tangential direction $\mathbf{t}_h$, i.e., 
\begin{equation*}
\begin{aligned}
&\sigma(1,\mathbf{v}^{J,\pm},q^{J,\pm})\mathbf{n}_h=\left(2\frac{\partial (\mathbf{v}^{J,\pm}\cdot \mathbf{n}_h)}{\partial \mathbf{n}_h}-q^{J,\pm}\right)\mathbf{n}_h+ \left( \frac{\partial (\mathbf{v}^{J,\pm}\cdot \mathbf{n}_h)}{\partial \mathbf{t}_h}+\frac{\partial (\mathbf{v}^{J,\pm}\cdot \mathbf{t}_h)}{\partial \mathbf{n}_h}\right)\mathbf{t}_h,\\
&\nabla \cdot \mathbf{v}^{J,\pm}=\frac{\partial (\mathbf{v}^{J,\pm}\cdot \mathbf{n}_h)}{\partial \mathbf{n}_h}+\frac{\partial (\mathbf{v}^{J,\pm}\cdot \mathbf{t}_h)}{\partial \mathbf{t}_h},
\end{aligned}
\end{equation*}
which can also be derived  easily in a new $\mathbf{n}_h$-$\mathbf{t}_h$ coordinate system.
The detailed  proof can be found in the literature;  see, e.g., \cite{ito2006interface,tan2009immersed}.

For the function $\mathbf{v}^{J_1}$ defined in (\ref{def_vj1}), we set $\mathbf{g}=\mathbf{n}_h$ in (\ref{new_add}), then (\ref{new_add2}) becomes
\begin{equation*}
\begin{aligned}
&[\![\nabla (\mathbf{v}^{J_1,\pm} \cdot\mathbf{n}_h)\cdot\mathbf{n}_h]\!]=0,~[\![\nabla (\mathbf{v}^{J_1,\pm} \cdot\mathbf{n}_h)\cdot\mathbf{t}_h]\!]=0, \\
&[\![\nabla (\mathbf{v}^{J_1,\pm} \cdot\mathbf{t}_h)\cdot\mathbf{n}_h]\!]=0,~[\![\nabla (\mathbf{v}^{J_1,\pm} \cdot\mathbf{t}_h)\cdot\mathbf{t}_h]\!]=0, ~[\![q^{J_1,\pm}]\!]=-1,
\end{aligned}
\end{equation*}
which together with $[\![\mathbf{v}^{J_1,\pm}]\!]|_{\Gamma_{h,T}}=\mathbf{0}$, $\mathbf{v}^{J_1,\pm}\in \mathbf{V}_h(T),~q^{J_1,\pm}\in M_h(T)$ and $N_{i,T}(\mathbf{v}^{J_1}, q^{J_1})=0,~i=1,...,7$ implies that $\mathbf{v}^{J_1}$ and  $q^{J_1}$ exist uniquely and can be constructed from (\ref{cons_j1j2})-(\ref{def_zw}).
Similarly, for the function $\mathbf{v}^{J_2}$ defined in (\ref{def_vj2}),  with $\mathbf{g}=\mathbf{t}_h$, we obtain
\begin{equation*}
\begin{aligned}
&[\![\nabla (\mathbf{v}^{J_2,\pm} \cdot\mathbf{n}_h)\cdot\mathbf{n}_h]\!]=0,~[\![\nabla (\mathbf{v}^{J_2,\pm} \cdot\mathbf{n}_h)\cdot\mathbf{t}_h]\!]=0, \\
&[\![\nabla (\mathbf{v}^{J_2,\pm} \cdot\mathbf{t}_h)\cdot\mathbf{n}_h]\!]=1,~[\![\nabla (\mathbf{v}^{J_2,\pm} \cdot\mathbf{t}_h)\cdot\mathbf{t}_h]\!]=0, ~[\![q^{J_2,\pm}]\!]=0.
\end{aligned}
\end{equation*}
Using the fact $[\![\mathbf{v}^{J_2,\pm}]\!]|_{\Gamma_{h,T}}=\mathbf{0}$, $\mathbf{v}^{J_2,\pm}\in \mathbf{V}_h(T),~q^{J_2,\pm}\in M_h(T)$ and $N_{i,T}(\mathbf{v}^{J_2}, q^{J_2})=0,~i=1,...,7$, we have
\begin{equation*}
\mathbf{v}^{J_2} \cdot\mathbf{n}_h=0,~~ \mathbf{v}^{J_2} \cdot\mathbf{t}_h=w-\pi_{h,T}^{CR}w,~~q^{J_2}=0,
\end{equation*}
which completes the proof.
\end{proof}

Since $\mathbf{v}^{J_1}=\mathbf{0}$ and $q^{J_2}=0$ from (\ref{cons_j1j2}), the equation (\ref{c1c2_1}) can be simplified as   
\begin{equation}\label{c1c2_2}
[\![\sigma(\mu^\pm,c_2\mathbf{v}^{J_2,\pm},c_1q^{J_1,\pm})\mathbf{n}_h]\!]=-\sigma([\![\mu^\pm]\!],\mathbf{v}^{J_0},0)\mathbf{n}_h.
\end{equation}
By the fact $\mathbf{v}^{J_2}\cdot\mathbf{n}_h=0$ from (\ref{cons_j1j2}), the above equation (\ref{c1c2_2}) becomes
\begin{equation}\label{c1c2_3}
\left(
\begin{aligned}
&-[\![q^{J_1,\pm}]\!] & 0 \\
&0  & [\![\mu^\pm \nabla(\mathbf{v}^{J_2,\pm}\cdot\mathbf{t}_h)\cdot \mathbf{n}_h]\!] 
\end{aligned}
\right)
\left(
\begin{aligned}
c_1\\
c_2
\end{aligned}
\right)=
-\left(
\begin{aligned}
&\sigma([\![\mu^\pm]\!],\mathbf{v}^{J_0},0)\mathbf{n}_h\cdot\mathbf{n}_h\\
&\sigma([\![\mu^\pm]\!],\mathbf{v}^{J_0},0)\mathbf{n}_h\cdot\mathbf{t}_h
\end{aligned}
\right).
\end{equation}
Using (\ref{cons_j1j2}) again, we have
$-[\![q^{J_1,\pm}]\!]=1$ and 
\begin{equation*}
\begin{aligned}
[\![\mu^\pm \nabla(\mathbf{v}^{J_2,\pm}\cdot\mathbf{t}_h)\cdot \mathbf{n}_h]\!]&=[\![\mu^\pm \nabla(w-\pi_{h,T}^{CR}w)^\pm\cdot \mathbf{n}_h]\!]=[\![\mu^\pm \nabla w^\pm\cdot \mathbf{n}_h]\!]-[\![\mu^\pm]\!] \nabla\pi_{h,T}^{CR}w\cdot \mathbf{n}_h\\
&=\mu^+-(\mu^+-\mu^-)\nabla\pi_{h,T}^{CR}w\cdot \mathbf{n}_h.
\end{aligned}
\end{equation*}
Thus,  the system of linear equations (\ref{c1c2_3}) for the augmented variable $(c_1,c_2)^T$ becomes 
\begin{equation}\label{eq_c1c2}
\left(
\begin{array}{cc}
1 & 0 \\
0  & 1+(\mu^-/\mu^+-1)\nabla\pi_{h,T}^{CR}w\cdot \mathbf{n}_h
\end{array}
\right)
\left(
\begin{aligned}
c_1\\
c_2
\end{aligned}
\right)=
\left(
\begin{aligned}
\sigma(\mu^--\mu^+,\mathbf{v}^{J_0},0)\mathbf{n}_h\cdot\mathbf{n}_h\\
\sigma(\mu^-/\mu^+-1,\mathbf{v}^{J_0},0)\mathbf{n}_h\cdot\mathbf{t}_h
\end{aligned}
\right).
\end{equation}

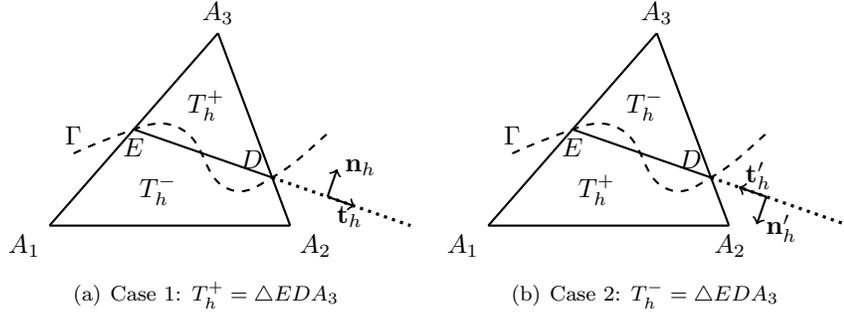
\begin{figure}[htbp]
\centering
\subfigure[Case 1: $T_h^+=\triangle EDA_3$]{ \label{interface_ele1} 
\begin{tikzpicture}[scale=1.6]
\draw [thick] (-1,-1)--(1,-1);
\draw [thick, name path=e2] (-1,-1)--(0.4,0.6);
\draw [thick, name path=e1] (0.4,0.6)--(1,-1);
\draw [thick](-1+1.4/2,-1+1.6/2)--(1-0.6/4,-1+1.6/4);
\node [below left] at (-1,-1) {$A_1$};
\node [below right ]at (1,-1) {$A_2$};
\node [above]at (0.4,0.6) {$A_3$};
\node [above left] at (1-0.6/4,-1+1.6/4) {$D$};
\node [below ] at (-1+1.4/2,-1+1.6/2) {$E$};
\draw [dashed, thick](-1+1.4/2-0.5,-1+1.6/2-0.2) to [out=25,in=180+20] (-1+1.4/2,-1+1.6/2)
to [out=20,in=150] (-1+1.4/2+0.4,-1+1.6/2)
to [out=-30,in=160] (-1+1.4/2+0.8,-1+1.6/2-0.5)
to [out=-20,in=180+30] (1-0.6/4,-1+1.6/4)
to [out=30,in=180+40] (1-0.6/4+0.5,-1+1.6/4+0.4);
\node [above] at (-1+1.4/2-0.5,-1+1.6/2-0.2) {$\Gamma$};
\node at (-1+1.4/2+0.6,-1+1.6/2+0.2) {$T_h^+$};
\node at (-1+1.4/2+0.2,-1+1.6/2-0.5) {$T_h^-$};
\draw [dotted, very thick] (0.85,-0.6)--(0.85+1.15,-0.6+-0.4);
\draw [ ->,thick] (0.85+1.15*0.4,-0.6-0.4*0.4)--(0.85+1.15*0.6,-0.6-0.4*0.6);
\draw [ ->,thick] (0.85+1.15*0.4,-0.6-0.4*0.4)--(0.85+1.15*0.4+0.4*0.2 ,-0.6-0.4*0.4 +1.15*0.2);
\node [above right] at (0.85+1.15*0.4+0.07,-0.6-0.4*0.4+0.1) {$\textbf{n}_h$};
\node [below right] at (0.85+1.15*0.4,-0.6-0.4*0.4) {$\textbf{t}_h$};
\end{tikzpicture}
}
\subfigure[Case 2: $T_h^-=\triangle EDA_3$]{ \label{interface_ele2} 
\begin{tikzpicture}[scale=1.6]
\draw [thick] (-1,-1)--(1,-1);
\draw [thick, name path=e2] (-1,-1)--(0.4,0.6);
\draw [thick, name path=e1] (0.4,0.6)--(1,-1);
\draw [thick](-1+1.4/2,-1+1.6/2)--(1-0.6/4,-1+1.6/4);
\node [below left] at (-1,-1) {$A_1$};
\node [below right ]at (0.8,-1) {$A_2$};
\node [above]at (0.4,0.6) {$A_3$};
\node [above left] at (1-0.6/4,-1+1.6/4) {$D$};
\node [below ] at (-1+1.4/2,-1+1.6/2) {$E$};
\draw [dashed, thick](-1+1.4/2-0.5,-1+1.6/2-0.2) to [out=25,in=180+20] (-1+1.4/2,-1+1.6/2)
to [out=20,in=150] (-1+1.4/2+0.4,-1+1.6/2)
to [out=-30,in=160] (-1+1.4/2+0.8,-1+1.6/2-0.5)
to [out=-20,in=180+30] (1-0.6/4,-1+1.6/4)
to [out=30,in=180+40] (1-0.6/4+0.5,-1+1.6/4+0.4);
\node [above] at (-1+1.4/2-0.5,-1+1.6/2-0.2) {$\Gamma$};
\node at (-1+1.4/2+0.6,-1+1.6/2+0.2) {$T_h^-$};
\node at (-1+1.4/2+0.2,-1+1.6/2-0.5) {$T_h^+$};
\draw [dotted, very thick] (0.85,-0.6)--(0.85+1.15,-0.6+-0.4);
\draw [ ->,thick] (0.85+1.15*0.4,-0.6-0.4*0.4)--(0.85+1.15*0.2,-0.6-0.4*0.2);
\draw [ ->,thick] (0.85+1.15*0.4,-0.6-0.4*0.4)--(0.85+1.15*0.4-0.4*0.2 ,-0.6-0.4*0.4 -1.15*0.2);
\node [above right] at (0.85+1.15*0.4-0.08,-0.6-0.4*0.4-0.45) {$\textbf{n}_h^\prime$};
\node [below right] at (0.85+1.15*0.4-0.25,-0.6+0.4*0.4+0.05) {$\textbf{t}_h^\prime$};
\end{tikzpicture}
}
\caption{Diagrams of typical interface elements. }\label{interface_ele}
\end{figure}

To prove the above system of linear equations having a unique solution, we need an estimate of $\nabla\pi_{h,T}^{CR}w\cdot \mathbf{n}_h$, which is shown in the following lemma.
\begin{lemma}\label{lem_01}
Let $T$ be an arbitrary interface triangle with an arbitrary $\Gamma_{h,T}$,  and $w$ be a piecewise linear function defined in (\ref{def_zw}). Then it holds that
\begin{equation}\label{est01}
\nabla\pi_{h,T}^{CR}w\cdot \mathbf{n}_h=\frac{|T_h^+|}{|T|}\in[0,1].
\end{equation}
\end{lemma}
\begin{proof}
Consider $T=\triangle A_1A_2A_3$ with edges $e_1=\overline{A_2A_3}$, $e_2=\overline{A_1A_3}$ and $e_3=\overline{A_1A_2}$.  Without loss of generality, we  assume that the interface $\Gamma$ cuts $e_1$ and $e_2$ at points $D$ and $E$.  
There are two cases: Case 1: $T_h^+=\triangle EDA_3$ (see Figure~\ref{interface_ele1}); Case 2: $T_h^-=\triangle EDA_3$ (see Figure~\ref{interface_ele2}). In Case 1, we have from (\ref{def_zw})  that 
\begin{equation}\label{case1}
w(\mathbf{x})=\left\{
\begin{aligned}
&\mathbf{n}_h\cdot\overrightarrow{D\mathbf{x}} \qquad&&\mbox{ if }\mathbf{x} \in \triangle EDA_3,\\
&0&&\mbox{ if } \mathbf{x} \in T\backslash\triangle EDA_3.
\end{aligned}\right.
\end{equation}
In order to distinguish between these two cases, we replace the notations $\textbf{n}_h$ and $w$ by $\textbf{n}^\prime_h$ and $w^\prime$ in Case 2. Using the fact $\textbf{n}_h^\prime=-\textbf{n}_h$, we have the following result according to (\ref{def_zw}) 
\begin{equation}\label{case2}
w^\prime(\mathbf{x})=\left\{
\begin{aligned}
&0\qquad&&\mbox{ if } \mathbf{x}\in \triangle EDA_3,\\
&-\textbf{n}_h\cdot\overrightarrow{D\mathbf{x}}   \qquad&&\mbox{ if  } \mathbf{x} \in T\backslash\triangle EDA_3.
\end{aligned}\right.
\end{equation}
Comparing  (\ref{case1}) with (\ref{case2}), we find $w^\prime=w-\mathbf{n}_h\cdot\overrightarrow{D\mathbf{x}}$,
which implies 
\begin{equation}\label{case12c}
\nabla\pi_{h,T}^{CR}w^\prime\cdot \mathbf{n}_h^\prime =\nabla\pi_{h,T}^{CR}(w-\mathbf{n}_h\cdot\overrightarrow{D\mathbf{x}}) \cdot (-\mathbf{n}_h)=1-\nabla\pi_{h,T}^{CR}w\cdot \mathbf{n}_h.
\end{equation}
If the identity~(\ref{est01}) holds for Case 1, then we can conclude from (\ref{case12c})  that the identity~(\ref{est01}) also holds for Case 2. Therefore, we just need to consider Case 1 whose geometric configuration is given in Figure~\ref{interface_ele1}. 

The proof for Case 1  is similar to that of  Lemma~3.3 in \cite{2021ji_nonconform}. By the definitions of the interpolation operator  $\pi_{h,T}^{CR}$ in (\ref{def_opera}) and the basis functions $\lambda_{i,T}$ in (\ref{cr_basis}), we have
\begin{equation}\label{pro_010}
\nabla\pi_{h,T}^{CR}w\cdot \mathbf{n}_h=\nabla\lambda_{1,T}\cdot\mathbf{n}_h\frac{1}{|e_1|}\int_{\overline{A_3D}} \mathbf{n}_h\cdot\overrightarrow{D\mathbf{x}} +\nabla\lambda_{2,T}\cdot\mathbf{n}_h\frac{1}{|e_2|}\int_{\overline{A_3E}} \mathbf{n}_h\cdot\overrightarrow{D\mathbf{x}}.
\end{equation}
Let $M_2$ be the midpoint of the edge $e_2$ and $Q$ be the orthogonal projection of $M_2$ onto the line $A_2A_3$. Then,  it holds
\begin{equation*}
\begin{aligned}
\nabla\lambda_1\cdot \textbf{n}_h&= |M_2Q|^{-1}\overrightarrow{M_2Q}|M_2Q|^{-1}\cdot \textbf{n}_h=|M_2Q|^{-1}R_{-\pi/2}\left(\overrightarrow{M_2Q}|M_2Q|^{-1}\right)\cdot R_{-\pi/2}\textbf{n}_h\\
&=|M_2Q|^{-1}|A_3D|^{-1}\overrightarrow{A_3D}\cdot \textbf{t}_h.
\end{aligned}
\end{equation*}
Note that 
\begin{equation*}
\frac{1}{|e_1|}\int_{\overline{A_3D}}\textbf{n}_h\cdot\overrightarrow{D\mathbf{x}}ds=\frac{1}{2}|e_1|^{-1}|A_3D|\textbf{n}_h\cdot\overrightarrow{DA_3}.
\end{equation*}
Therefore, it follows from the above identities and the fact $|M_2Q||e_1|=|T|$ that 
\begin{equation}\label{pro_011}
\nabla\lambda_{1,T}\cdot\mathbf{n}_h\frac{1}{|e_1|}\int_{\overline{A_3D}} \mathbf{n}_h\cdot\overrightarrow{D\mathbf{x}}=\frac{1}{2}|T|^{-1}(\textbf{n}_h\cdot\overrightarrow{DA_3})(\overrightarrow{A_3D}\cdot\mathbf{t}_h).
\end{equation}
Analogously, we have
\begin{equation}\label{pro_012}
\nabla\lambda_{2,T}\cdot\mathbf{n}_h\frac{1}{|e_2|}\int_{\overline{A_3E}} \mathbf{n}_h\cdot\overrightarrow{D\mathbf{x}}=\frac{1}{2}|T|^{-1}(\textbf{n}_h\cdot\overrightarrow{DA_3})(\overrightarrow{EA_3}\cdot\mathbf{t}_h).
\end{equation}
Substituting (\ref{pro_011}) and (\ref{pro_012}) into (\ref{pro_010}) yields 
\begin{equation}
\nabla\pi_{h,T}^{CR}w\cdot \mathbf{n}_h=\frac{1}{2}|T|^{-1}(\textbf{n}_h\cdot\overrightarrow{DA_3})(\overrightarrow{ED}\cdot\mathbf{t}_h)=\frac{1}{2}|T|^{-1}(\textbf{n}_h\cdot\overrightarrow{DA_3})|ED|=\frac{|T_h^+|}{|T|}\in [0,1],
\end{equation}
which completes the proof.
\end{proof}

Now we are ready to prove the unisolvence of IFE shape functions with respect to the degrees of freedom on arbitrary triangles.
\begin{lemma}\label{lem_IFEbasis}
For an arbitrary interface triangle $T\in\mathcal{T}_h^\Gamma$, the pair of functions $(\mathbf{v},q)\in \mathbf{V}M_h^{IFE}(T)$ is uniquely determined by $N_{i,T}(\mathbf{v},q),~i=1,...,7$. Furthermore, we have the following explicit formula
\begin{equation}\label{explicit_formula}
\begin{aligned}
&(\mathbf{v},q)=(\mathbf{v}^{J_0}, q^{J_0}) +(c_2\mathbf{v}^{J_2}, c_1q^{J_1})
\end{aligned}
\end{equation}
with 
\begin{equation}\label{ex_c1c2}
\begin{aligned}
&c_1=\sigma(\mu^--\mu^+,\mathbf{v}^{J_0},0)\mathbf{n}_h\cdot\mathbf{n}_h, ~~~c_2=\frac{\sigma(\mu^-/\mu^+-1,\mathbf{v}^{J_0},0)\mathbf{n}_h\cdot\mathbf{t}_h}{1+(\mu^-/\mu^+-1)\nabla\pi_{h,T}^{CR}w\cdot \mathbf{n}_h},\\
&\mathbf{v}^{J_0}=\sum_{i=1}^{6}N_{i,T}(\mathbf{v},q)\boldsymbol{\phi}_{i,T},~~q^{J_0}=N_{7,T}(\mathbf{v},q),
\end{aligned}
\end{equation}
where 
$\mathbf{v}^{J_2}$, $q^{J_1}$, $w$ and $\boldsymbol{\phi}_{i,T}$ are defined in (\ref{cons_j1j2}), (\ref{def_zw}) and (\ref{crp0_basis}), and $\pi_{h,T}^{CR}$ is the standard $CR$ interpolation defined in (\ref{def_opera}).
\end{lemma}
\begin{proof}
From Lemma~\ref{lem_01}, we have
\begin{equation}\label{pro_basi_fenmu}
1+(\mu^-/\mu^+-1)\nabla\pi_{h,T}^{CR}w\cdot \mathbf{n}_h\geq\left\{
\begin{aligned}
&1\qquad&&\mbox{ if }\mu^-/\mu^+\geq 1,\\
&\mu^-/\mu^+\qquad&&\mbox{ if }0<\mu^-/\mu^+< 1.
\end{aligned}\right.
\end{equation}
Hence, the equation (\ref{eq_c1c2}) has a unique solution $(c_1,c_2)^T$ as shown in (\ref{ex_c1c2}). The proof is now completed by substituting  (\ref{cons_j1j2}) into (\ref{base_decouple}).
\end{proof}

\begin{remark}\label{IFEbasiseqbasis}
If $\mu^+=\mu^-$, then $c_1=c_2=0$. Thus, the IFE space $\mathbf{V}M_h^{IFE}$ becomes the standard $CR$-$P_0$ finite element space $(\mathbf{V}_h, M_h)$. If $|T_h^+|\rightarrow 0$ or $|T_h^-|\rightarrow 0$, from (\ref{cons_j1j2}) and  (\ref{def_zw})  we have $\mathbf{v}^{J_2}\rightarrow\mathbf{0}$ and $q^{J_1}\rightarrow0$.  In addition, using (\ref{pro_basi_fenmu}) and (\ref{ex_c1c2}) we have $|c_1|\leq C|\nabla \mathbf{v}^{J_0}|$ and $|c_1|\leq C|\nabla \mathbf{v}^{J_0}|$. Then it holds $(\mathbf{v},q)\rightarrow (\mathbf{v}^{J_0}, q^{J_0})\in (\mathbf{V}_h(T), M_h(T))$. Therefore, the IFE space $\mathbf{V}M_h^{IFE}(T)$ tends to the standard $CR$-$P_0$ finite element space $(\mathbf{V}_h(T), M_h(T))$  as $|T_h^+|\rightarrow 0$ or $|T_h^-|\rightarrow 0$.  This nice feature of the IFE space is desirable for moving interface problems. 
\end{remark}

\begin{remark}\label{basiseq0}
If $N_{i,T}(\mathbf{v},q)=0$, $i=1,...,7$, then $(\mathbf{v}^{J_0}, q^{J_0})=(\mathbf{0},0)$. From (\ref{ex_c1c2}), we also have $c_1=c_2=0$. Hence, we conclude $(\mathbf{v},q)=(\mathbf{0},0)$ when $(\mathbf{v},q)\in \mathbf{V}M_h^{IFE}(T)$ and $N_{i,T}(\mathbf{v},q)=0$, $i=1,...,7$.
\end{remark}

\begin{remark}\label{basiseq0_affect}
From (\ref{def_opera})-(\ref{operaPI_1}), we know $\mathbf{v}^{J_0}=\boldsymbol{\pi}_{h,T}^{CR}\mathbf{v}$ and $q^{J_0}=\pi_{h,T}^0q$.  Hence, the IFE interpolations of $(\mathbf{v}, q)\in (H^1(\Omega)^2,L^2(\Omega))$ on an interface element $T\in\mathcal{T}_h^\Gamma$ are
\begin{equation*}
(\Pi_{\mathbf{v},q}^{IFE}\mathbf{v})|_{T}=\boldsymbol{\pi}_{h,T}^{CR}\mathbf{v}+c_2\mathbf{v}^{J_2}\mbox{ and }(\Pi_{\mathbf{v},q}^{IFE}q)|_{T}=\pi_{h,T}^{0}q+c_1q^{J_1}
\end{equation*}
with $c_1$ and $c_2$ defined in (\ref{ex_c1c2}) that are independent of the pressure $q$. From the above identities, we find that $\Pi_{\mathbf{v},q}^{IFE}\mathbf{v} $  depends only on the velocity $\mathbf{v}$, not on the pressure $q$. However, $\Pi_{\mathbf{v},q}^{IFE}q$ depends both on $\mathbf{v}$ and $q$. 
\end{remark}

\subsection{Estimates of IFE basis functions}

For each interface element $T\in\mathcal{T}_h^\Gamma$, similar to (\ref{crp0_basis}), we define IFE basis functions by
\begin{equation}\label{def_IFEbasis}
(\boldsymbol{\phi}_{i,T}^{IFE}, \varphi_{i,T}^{IFE})\in \mathbf{V}M_h^{IFE}(T), \quad N_{j,T}(\boldsymbol{\phi}_{i,T}^{IFE}, \varphi_{i,T}^{IFE})=\delta_{ij},\quad  \forall i,j=1,...7.
\end{equation}
Using Lemma~\ref{lem_IFEbasis}, we can write these IFE basis functions $(\boldsymbol{\phi}_{i,T}^{IFE}, \varphi_{i,T}^{IFE})$  explicitly as
\begin{equation}\label{exp_IFE_basis}
\begin{aligned}
&\boldsymbol{\phi}_{i,T}^{IFE}=\boldsymbol{\phi}_{i,T}+\frac{\sigma(\mu^-/\mu^+-1,\boldsymbol{\phi}_{i,T},0)\mathbf{n}_h\cdot\mathbf{t}_h}{1+(\mu^-/\mu^+-1)\nabla\pi_{h,T}^{CR}w\cdot \mathbf{n}_h}(w-\pi_{h,T}^{CR}w)\mathbf{t}_h,~i=1,...,6,\\
&\varphi_{i,T}^{IFE}=\sigma(\mu^--\mu^+,\boldsymbol{\phi}_{i,T},0)\mathbf{n}_h\cdot\mathbf{n}_h(z-\pi_{h,T}^0z),~i=1,...,6,\\
&\boldsymbol{\phi}_{7,T}^{IFE}=\mathbf{0},~~~\varphi_{7,T}^{IFE}=1,
\end{aligned}
\end{equation}
where  $\boldsymbol{\phi}_{i,T}$, $i=1,...,6$ are the standard CR basis functions for the velocity (see (\ref{crvec_basis})), and $w$ and $z$ are known functions defined in (\ref{def_zw}). 
Also we have $(\boldsymbol{\phi}_{7,T}^{IFE}, \varphi_{7,T}^{IFE})=(\boldsymbol{\phi}_{7,T}, \varphi_{7,T})$ from (\ref{crvec_basis}).
We emphasize   that these explicit formulas for IFE basis functions  are very useful in the implementation.

From (\ref{pro_basi_fenmu}), we  highlight  that  the denominator in the IFE basis functions  (\ref{exp_IFE_basis}) does not tend to zero  even if $|T_h^+|\rightarrow 0$ or $|T_h^-|\rightarrow 0$. 


\begin{lemma}\label{lem_basisest}
There exists a positive constant $C$ depending only on $\mu^\pm$ and the shape regularity parameter $\varrho$ such that, for $m=0,1$,
\begin{equation}\label{est_IFEbsis}
\begin{aligned}
&|(\boldsymbol{\phi}_{i,T}^{IFE})^\pm|_{W_\infty^m(T)}\leq Ch_T^{-m},~~~\|(\varphi_{i,T}^{IFE})^\pm\|_{L^\infty(T)}\leq Ch_T^{-1},~~i=1,...,6,\\
&|(\boldsymbol{\phi}_{7,T}^{IFE})^\pm|_{W_\infty^m(T)}=0,~~\|(\varphi_{7,T}^{IFE})^\pm\|_{L^\infty(T)}=1.
\end{aligned}
\end{equation}
\end{lemma}
\begin{proof}
It suffices to estimate the terms on the right-hand side of (\ref{exp_IFE_basis}).
First we have the following estimates about the standard CR basis functions  
\begin{equation*}
|\lambda_{i,T}|_{W_\infty^m(T)}\leq Ch_T^{-m} ~~\mbox{ and } ~~|\boldsymbol{\phi}_{i,T}|_{W_\infty^m(T)}\leq Ch_T^{-m},\qquad m=0,1.
\end{equation*}
Using the $\mathbf{n}_h$-$\mathbf{t}_h$ coordinate system, we then have 
\begin{equation*}
\begin{aligned}
|\sigma(\mu^-/\mu^+-1,\boldsymbol{\phi}_{i,T},0)\mathbf{n}_h\cdot\mathbf{t}_h|&=|(\mu^-/\mu^+-1)(\nabla(\boldsymbol{\phi}_{i,T}\cdot\mathbf{n}_h)\cdot\mathbf{t}_h+\nabla(\boldsymbol{\phi}_{i,T}\cdot\mathbf{t}_h)\cdot\mathbf{n}_h)|\leq Ch_T^{-1},\\
|\sigma(\mu^--\mu^+,\boldsymbol{\phi}_{i,T},0)\mathbf{n}_h\cdot\mathbf{n}_h|&=|2(\mu^--\mu^+)\nabla(\boldsymbol{\phi}_{i,T}\cdot\mathbf{n}_h)\cdot\mathbf{n}_h|\leq Ch_T^{-1}.
\end{aligned}
\end{equation*}
By the definitions of $w$ and $z$ in (\ref{def_zw}), we also have 
\begin{equation}\label{est_piw}
\begin{aligned}
&|w^+|_{W_\infty^m(T)}\leq Ch_T^{1-m}, ~~|w^-|_{W_\infty^m(T)}=0, \\
&|\pi_{h,T}^{CR}w|_{W_\infty^m(T)}=\left|\sum_{i=1}^3 \lambda_{i,T}\frac{1}{|e_i|}\int_{e_i} w\right|_{W_\infty^m(T)}\leq Ch_T\sum_{i=1}^3|\lambda_{i,T}|_{W_\infty^m(T)}\leq Ch^{1-m}_T,\\
&|z^+|=1,~~ |z^-|=0, ~~|\pi_{h,T}^0z|=|T|^{-1}\left|\int_{T}z\right|\leq \|z\|_{L^\infty(T)}\leq 1.
\end{aligned}
\end{equation}
Finally, the desired  estimates (\ref{est_IFEbsis}) are obtained by substituting (\ref{pro_basi_fenmu}) and the above estimates into (\ref{exp_IFE_basis}).
\end{proof}

\subsection{Approximation capabilities of the IFE space}
For clarity, we first describe the main idea of the proof of approximation capabilities of the IFE space. Our goal is to estimate the following error on each interface element $T\in\mathcal{T}_h^\Gamma$,
\begin{equation*}
\begin{aligned}
\|(\mathbf{v}_E^\pm, q_E^\pm)- \left(\Pi_{h,T}^{IFE}(\mathbf{v}, q)\right)^\pm\|_{T},
\end{aligned}
\end{equation*}
where $\|\cdot\|_T$ is a specific norm, $\mathbf{v}_E^\pm$ and $q_E^\pm$ are extensions of $\mathbf{v}^\pm$ and $q^\pm$ as shown in Lemma~\ref{lem_ext}, and  the notation of superscripts $+$ or $-$ is described at the end of Section~\ref{preliminary}.  Obviously, the function can be split as 
\begin{equation}\label{decomp_newadd}
\begin{aligned}
(\mathbf{v}_E^\pm, q_E^\pm)- \left(\Pi_{h,T}^{IFE}(\mathbf{v}, q)\right)^\pm
=\underbrace{(\mathbf{v}_E^\pm, q_E^\pm)-\Pi_{h,T}(\mathbf{v}_E^\pm, q_E^\pm)}_{({\rm I})}+\underbrace{\left(\Pi_{h,T}(\mathbf{v}_E^\pm, q_E^\pm)- \left(\Pi_{h,T}^{IFE}(\mathbf{v}, q)\right)^\pm \right)}_{({\rm II})}.
\end{aligned}
\end{equation}
The estimate of the first term $({\rm I})$ is standard and the main difficulty is to estimate the second term $({\rm II})$. 
Noticing that functions in the term $({\rm II})$ are piecewise polynomials on the interface element $T\in\mathcal{T}_h^\Gamma$, our idea is to decompose the term $({\rm II})$ by proper degrees of freedom as shown in Lemma~\ref{lema_fenjie}. Then we estimate every terms in the decomposition  to get the desired results (see Theorem~\ref{lem_interIFE_up}). The degrees of freedom for determining the term $({\rm II})$ include $N_{j,T}$, $j=1,...,7$, and others related to the interface jumps  (\ref{jp_cond_IFE1})-(\ref{jp_cond_IFE3}), which inspire us to define the following novel auxiliary functions.

On each interface element $T\in\mathcal{T}_h^\Gamma$, we define auxiliary functions $( \boldsymbol{\Psi}_{i,T}, \psi_{i,T})$, $i=1,...,7$ with  $ \boldsymbol{\Psi}_{i,T}|_{T_h^\pm}= \boldsymbol{\Psi}_{i,T}^{\pm},~\psi_{i,T}|_{T_h^\pm}=\psi_{i,T}^{\pm}$ such that 
\begin{equation}\label{aux1}
\begin{aligned}
(\boldsymbol{\Psi}_{i,T}^{\pm}, \psi_{i,T}^{\pm})\in (\mathbf{V}_h(T), M_h(T)),~N_{j,T}( \boldsymbol{\Psi}_{i,T}, \psi_{i,T})=0,~j=1,...,7,
\end{aligned}
\end{equation}
and
\begin{equation}\label{aux2}
\begin{aligned}
&[\![\sigma(\mu^\pm, \boldsymbol{\Psi}_{1,T}^{\pm},\psi_{1,T}^{\pm})\mathbf{n}_h]\!]=\mathbf{0},~&&[\![ \boldsymbol{\Psi}_{1,T}^{\pm}]\!](\mathbf{x}_{T})=\mathbf{n}_h,~&&[\![\nabla \boldsymbol{\Psi}_{1,T}^{\pm}\mathbf{t}_h]\!]=\mathbf{0},~&&[\![\nabla\cdot \boldsymbol{\Psi}_{1,T}^{\pm}]\!]=0,\\
&[\![\sigma(\mu^\pm, \boldsymbol{\Psi}_{2,T}^{\pm},\psi_{2,T}^{\pm})\mathbf{n}_h]\!]=\mathbf{0},~&&[\![ \boldsymbol{\Psi}_{2,T}^{\pm}]\!](\mathbf{x}_{T})=\mathbf{t}_h,~&&[\![\nabla \boldsymbol{\Psi}_{2,T}^{\pm}\mathbf{t}_h]\!]=\mathbf{0},~&&[\![\nabla\cdot \boldsymbol{\Psi}_{2,T}^{\pm}]\!]=0,\\
&[\![\sigma(\mu^\pm, \boldsymbol{\Psi}_{3,T}^{\pm},\psi_{3,T}^{\pm})\mathbf{n}_h]\!]=\mathbf{n}_h,~&&[\![ \boldsymbol{\Psi}_{3,T}^{\pm}]\!](\mathbf{x}_{T})=\mathbf{0},~&&[\![\nabla \boldsymbol{\Psi}_{3,T}^{\pm}\mathbf{t}_h]\!]=\mathbf{0},~&&[\![\nabla\cdot \boldsymbol{\Psi}_{3,T}^{\pm}]\!]=0,\\
&[\![\sigma(\mu^\pm, \boldsymbol{\Psi}_{4,T}^{\pm},\psi_{4,T}^{\pm})\mathbf{n}_h]\!]=\mathbf{t}_h,~&&[\![ \boldsymbol{\Psi}_{4,T}^{\pm}]\!](\mathbf{x}_{T})=\mathbf{0},~&&[\![\nabla \boldsymbol{\Psi}_{4,T}^{\pm}\mathbf{t}_h]\!]=\mathbf{0},~&&[\![\nabla\cdot \boldsymbol{\Psi}_{4,T}^{\pm}]\!]=0,\\
&[\![\sigma(\mu^\pm, \boldsymbol{\Psi}_{5,T}^{\pm},\psi_{5,T}^{\pm})\mathbf{n}_h]\!]=\mathbf{0},~&&[\![ \boldsymbol{\Psi}_{5,T}^{\pm}]\!](\mathbf{x}_{T})=\mathbf{0},~&&[\![\nabla \boldsymbol{\Psi}_{5,T}^{\pm}\mathbf{t}_h]\!]=\mathbf{n}_h,~&&[\![\nabla\cdot \boldsymbol{\Psi}_{5,T}^{\pm}]\!]=0,\\
&[\![\sigma(\mu^\pm, \boldsymbol{\Psi}_{6,T}^{\pm},\psi_{6,T}^{\pm})\mathbf{n}_h]\!]=\mathbf{0},~&&[\![ \boldsymbol{\Psi}_{6,T}^{\pm}]\!](\mathbf{x}_{T})=\mathbf{0},~&&[\![\nabla \boldsymbol{\Psi}_{6,T}^{\pm}\mathbf{t}_h]\!]=\mathbf{t}_h,~&&[\![\nabla\cdot \boldsymbol{\Psi}_{6,T}^{\pm}]\!]=0,\\
&[\![\sigma(\mu^\pm, \boldsymbol{\Psi}_{7,T}^{\pm},\psi_{7,T}^{\pm})\mathbf{n}_h]\!]=\mathbf{0},~&&[\![ \boldsymbol{\Psi}_{7,T}^{\pm}]\!](\mathbf{x}_{T})=\mathbf{0},~&&[\![\nabla \boldsymbol{\Psi}_{7,T}^{\pm}\mathbf{t}_h]\!]=\mathbf{0},~&&[\![\nabla\cdot \boldsymbol{\Psi}_{7,T}^{\pm}]\!]=1,
\end{aligned}
\end{equation}
where $\mathbf{x}_T$ is the same as that in (\ref{jp_cond_IFE2}).
\begin{lemma}\label{lem_auxest}
On each interface element $T\in\mathcal{T}_h^\Gamma$, these auxiliary functions $( \boldsymbol{\Psi}_{i,T}, \psi_{i,T})$, $i=1,...,7$  defined in (\ref{aux1})-(\ref{aux2}) exist uniquely and satisfy,  for $m=0,1$,
\begin{equation}\label{aux_esti}
|\boldsymbol{\Psi}_{i,T}^{\pm}|_{W_\infty^m(T)} \left\{
\begin{aligned}
&\leq Ch_T^{-m}~~ &&\mbox{ if } ~ i=1,2,\\
&=0 &&\mbox{ if } ~ i=3,\\
&\leq Ch_T^{1-m} &&\mbox{ if } ~  i=4,...,7,
\end{aligned}\right. ~~
\|\psi_{i,T}^{\pm}\|_{L^\infty(T)} \left\{
\begin{aligned}
&\leq Ch_T^{-1}~~ &&\mbox{ if } ~ i=1,2,\\
&\leq C &&\mbox{ if } ~  i=3,...,7,
\end{aligned}\right.
\end{equation}
where the constant $C$ depends only on $\mu^\pm$ and the shape regularity parameter $\varrho$.
\end{lemma}
\begin{proof}
The justification of the existence and uniqueness  is that the coefficient matrix is the same as that for determining the IFE shape functions in the space $\mathbf{V}M_h^{IFE}(T)$  if we write a 14-by-14 linear system of equations for the fourteen parameters (see Remark~\ref{remark_14by14}).  To derive the estimates (\ref{aux_esti}), we need explicit expressions of these auxiliary functions. First, we define $(\mathbf{v}^-_i, q^-_i)=(\mathbf{0}, 0)$ and  $(\mathbf{v}^+_i, q^+_i)\in (\mathbf{V}_h(T), M_h(T))$, $i=1,...,7$ such that
\begin{equation}\label{def_vq_aux}
\begin{aligned}
&\mathbf{v}^+_1(\mathbf{x}_T)=\mathbf{n}_h, \frac{\partial (\mathbf{v}^+_1\cdot\mathbf{n}_h)}{\partial \mathbf{n}_h}=0, \frac{\partial (\mathbf{v}^+_1\cdot\mathbf{n}_h)}{\partial \mathbf{t}_h}=0, \frac{\partial (\mathbf{v}^+_1\cdot\mathbf{t}_h)}{\partial \mathbf{n}_h}=0, \frac{\partial (\mathbf{v}^+_1\cdot\mathbf{t}_h)}{\partial \mathbf{t}_h}=0, q^+_1=0,\\
&\mathbf{v}^+_2(\mathbf{x}_T)=\mathbf{t}_h, \frac{\partial (\mathbf{v}^+_2\cdot\mathbf{n}_h)}{\partial \mathbf{n}_h}=0, \frac{\partial (\mathbf{v}^+_2\cdot\mathbf{n}_h)}{\partial \mathbf{t}_h}=0, \frac{\partial (\mathbf{v}^+_2\cdot\mathbf{t}_h)}{\partial \mathbf{n}_h}=0, \frac{\partial (\mathbf{v}^+_2\cdot\mathbf{t}_h)}{\partial \mathbf{t}_h}=0, q^+_2=0,\\
&\mathbf{v}^+_3(\mathbf{x}_T)=\mathbf{0}, \frac{\partial (\mathbf{v}^+_3\cdot\mathbf{n}_h)}{\partial \mathbf{n}_h}=0, \frac{\partial (\mathbf{v}^+_3\cdot\mathbf{n}_h)}{\partial \mathbf{t}_h}=0, \frac{\partial (\mathbf{v}^+_3\cdot\mathbf{t}_h)}{\partial \mathbf{n}_h}=0, \frac{\partial (\mathbf{v}^+_3\cdot\mathbf{t}_h)}{\partial \mathbf{t}_h}=0, q^+_3=-1,\\
&\mathbf{v}^+_4(\mathbf{x}_T)=\mathbf{0}, \frac{\partial (\mathbf{v}^+_4\cdot\mathbf{n}_h)}{\partial \mathbf{n}_h}=0, \frac{\partial (\mathbf{v}^+_4\cdot\mathbf{n}_h)}{\partial \mathbf{t}_h}=0, \frac{\partial (\mathbf{v}^+_4\cdot\mathbf{t}_h)}{\partial \mathbf{n}_h}=\frac{1}{\mu^+}, \frac{\partial (\mathbf{v}^+_4\cdot\mathbf{t}_h)}{\partial \mathbf{t}_h}=0, q^+_4=0,\\
&\mathbf{v}^+_5(\mathbf{x}_T)=\mathbf{0}, \frac{\partial (\mathbf{v}^+_5\cdot\mathbf{n}_h)}{\partial \mathbf{n}_h}=0, \frac{\partial (\mathbf{v}^+_5\cdot\mathbf{n}_h)}{\partial \mathbf{t}_h}=1, \frac{\partial (\mathbf{v}^+_5\cdot\mathbf{t}_h)}{\partial \mathbf{n}_h}=-1, \frac{\partial (\mathbf{v}^+_5\cdot\mathbf{t}_h)}{\partial \mathbf{t}_h}=0, q^+_5=0,\\
&\mathbf{v}^+_6(\mathbf{x}_T)=\mathbf{0}, \frac{\partial (\mathbf{v}^+_6\cdot\mathbf{n}_h)}{\partial \mathbf{n}_h}=-1, \frac{\partial (\mathbf{v}^+_6\cdot\mathbf{n}_h)}{\partial \mathbf{t}_h}=0, \frac{\partial (\mathbf{v}^+_6\cdot\mathbf{t}_h)}{\partial \mathbf{n}_h}=0, \frac{\partial (\mathbf{v}^+_6\cdot\mathbf{t}_h)}{\partial \mathbf{t}_h}=1, q^+_6=-2\mu^+,\\
&\mathbf{v}^+_7(\mathbf{x}_T)=\mathbf{0}, \frac{\partial (\mathbf{v}^+_7\cdot\mathbf{n}_h)}{\partial \mathbf{n}_h}=1, \frac{\partial (\mathbf{v}^+_7\cdot\mathbf{n}_h)}{\partial \mathbf{t}_h}=0, \frac{\partial (\mathbf{v}^+_7\cdot\mathbf{t}_h)}{\partial \mathbf{n}_h}=0, \frac{\partial (\mathbf{v}^+_7\cdot\mathbf{t}_h)}{\partial \mathbf{t}_h}=0, q^+_7=2\mu^+.
\end{aligned}
\end{equation}
By using the $\mathbf{n}_h$-$\mathbf{t}_h$ coordinate system, it is easy to verify that  the above defined functions exist uniquely and satisfy the jump conditions (\ref{aux2}). If we define  $(\mathbf{v}_i, q_i)|_{T_h^\pm}=(\mathbf{v}^\pm_i, q^\pm_i)$,
then  the auxiliary functions satisfying (\ref{aux1})-(\ref{aux2}) can be obtained  by
\begin{equation}\label{psi_decomp_N}
\begin{aligned}
( \boldsymbol{\Psi}_{i,T}, \psi_{i,T})&=(\mathbf{v}_i, q_i)-\Pi_{h,T}^{IFE}(\mathbf{v}_i, q_i)=(\mathbf{v}_i, q_i)-\sum_{j=1}^7N_{j,T}(\mathbf{v}_i, q_i)(\boldsymbol{\phi}_{j,T}^{IFE}, \varphi_{j,T}^{IFE}).
\end{aligned}
\end{equation}
Now we estimate the terms on the right-hand side of the above identity. From (\ref{def_vq_aux}), we have
\begin{equation}\label{vpest}
\begin{aligned}
&|\mathbf{v}_i^+|_{W^m_\infty(T)}\leq Ch_T^{-m},~i=1,2,~|\mathbf{v}_3^+|_{W^m_\infty(T)}=0,~|\mathbf{v}_i^+|_{W^m_\infty(T)}\leq Ch_T^{1-m},~i=4,...,7,~m=1,2,\\
&\|q_i^+\|_{L^\infty(T)}\leq C,~i=3,6,7,~~\|q_i^+\|_{L^\infty(T)}=0,~i=1,2,4,5,
\end{aligned}
\end{equation}
which together with (\ref{loc_dof}) leads to
\begin{equation}\label{Nvpest}
\begin{aligned}
&|N_{j,T}(\mathbf{v}_i, q_i)|\leq C,~~j=1,...,6,~~&&|N_{7,T}(\mathbf{v}_i, q_i)|=0,~~i=1,2,\\
&|N_{j,T}(\mathbf{v}_3, q_3)|=0,~~j=1,...,6,~~&&|N_{7,T}(\mathbf{v}_3, q_3)|\leq C,\\
&|N_{j,T}(\mathbf{v}_i, q_i)|\leq Ch_T,~~j=1,...,6,~~&&|N_{7,T}(\mathbf{v}_i, q_i)|=0,~~i=4,5,\\
&|N_{j,T}(\mathbf{v}_i, q_i)|\leq Ch_T,~~j=1,...,6,~~&&|N_{7,T}(\mathbf{v}_i, q_i)|\leq C,~~i=6,7.\\
\end{aligned}
\end{equation}
Combining (\ref{psi_decomp_N})-(\ref{Nvpest}) and (\ref{est_IFEbsis}), we get the desired estimates (\ref{aux_esti}).
\end{proof}

The following lemma  presents a decomposition of the term $({\rm II})$ in (\ref{decomp_newadd}) by the auxiliary functions and IFE basis functions.
\begin{lemma}\label{lema_fenjie}
For any $(\mathbf{v},q)\in  (\mathbf{V}, M)$, let $(\mathbf{v}_E^\pm, q_E^\pm)$ be extensions of $(\mathbf{v}^\pm, q^\pm)$ as defined in Lemma~\ref{lem_ext}, and let $v_{E,1}^s$ and $v_{E,2}^s$ be two components of $\mathbf{v}_E^s$, i.e., $\mathbf{v}_E^s=(v_{E,1}^s,v_{E,2}^s)^T$, $s=\pm$.
For any $T\in\mathcal{T}_h^\Gamma$, let $e_i$, $i=1,2,3$ be its edges and we set $e_i^\pm=e_i\cap\Omega^\pm$.  Then it holds that
\begin{equation*}
\begin{aligned}
\Pi_{h,T}(\mathbf{v}_E^\pm, q_E^\pm)-\left(\Pi_{h,T}^{IFE}(\mathbf{v}, q)\right)^\pm=\sum_{i=1}^7(\boldsymbol{\phi}_{i,T}^{IFE}, \varphi_{i,T}^{IFE})^\pm\alpha_i +\sum_{i=1}^7(\boldsymbol{\Psi}_{i,T}, \psi_{i,T})^{\pm}\beta_i, 
\end{aligned}
\end{equation*}
where
\begin{equation}\label{alpha}
\begin{aligned}
&\alpha_i=\frac{1}{|e_i|}\sum_{s=\pm}\int_{e_i^s}(\pi_{h,T}^{CR}v_{E,1}^s-v_{E,1}^s),~\alpha_{3+i}=\frac{1}{|e_i|}\sum_{s=\pm}\int_{e_i^s}(\pi_{h,T}^{CR}v_{E,2}^s-v_{E,2}^s),~i=1,2,3,\\
&\alpha_7=\frac{1}{|T|}\left(\sum_{s=\pm}\int_{T_h^s}(\pi_{h,T}^0q_E^s-q_E^s)\textcolor{red}{+\int_{T^\triangle\cap T^-}[\![q_E^\pm]\!]-\int_{T^\triangle\cap T^+}[\![q_E^\pm]\!]}\right)
\end{aligned}
\end{equation}
and
\begin{equation}\label{beta}
\begin{aligned}
&\beta_1=[\![\boldsymbol{\pi}^{CR}_{h,T} \mathbf{v}_{E}^\pm]\!](\mathbf{x}_T)\cdot\mathbf{n}_h,~&&\beta_2=[\![\boldsymbol{\pi}^{CR}_{h,T} \mathbf{v}_{E}^\pm]\!](\mathbf{x}_T)\cdot\mathbf{t}_h,\\
&\beta_3=[\![\sigma(\mu^\pm, \boldsymbol{\pi}^{CR}_{h,T}\mathbf{v}_{E}^{\pm},\pi_{h,T}^0q_{E}^{\pm})\mathbf{n}_h]\!]\cdot\mathbf{n}_h,~&&\beta_4=[\![\sigma(\mu^\pm, \boldsymbol{\pi}^{CR}_{h,T}\mathbf{v}_{E}^{\pm},\pi_{h,T}^0q_{E}^{\pm})\mathbf{n}_h]\!]\cdot\mathbf{t}_h,\\
&\beta_5=[\![\nabla (\boldsymbol{\pi}^{CR}_{h,T}\mathbf{v}_{E}^{\pm})\mathbf{t}_h]\!]\cdot\mathbf{n}_h,~&&\beta_6=[\![\nabla (\boldsymbol{\pi}^{CR}_{h,T}\mathbf{v}_{E}^{\pm})\mathbf{t}_h]\!]\cdot\mathbf{t}_h, \qquad\beta_7=[\![\nabla\cdot (\boldsymbol{\pi}^{CR}_{h,T}\mathbf{v}_{E}^{\pm})]\!].
\end{aligned}
\end{equation}
\end{lemma}
\begin{proof}
For simplicity of notation, we define a pair of functions $(\boldsymbol{\Xi}_h,  \xi_h)$ such that 
\begin{equation}\label{xi}
(\boldsymbol{\Xi}_h,  \xi_h)|_{T_h^\pm}=(\boldsymbol{\Xi}_h,  \xi_h)^\pm ~~\mbox{ with }~~(\boldsymbol{\Xi}_h,  \xi_h)^\pm=\Pi_{h,T}(\mathbf{v}_E^\pm, q_E^\pm)-\left(\Pi_{h,T}^{IFE}(\mathbf{v}, q)\right)^\pm. 
\end{equation}
Define  another pair of functions $(\widehat{\boldsymbol{\Xi}}_h, \widehat{\xi}_h)$ by
\begin{equation}\label{xihat}
(\widehat{\boldsymbol{\Xi}}_h, \widehat{\xi}_h)=\sum_{i=1}^7(\boldsymbol{\phi}_{i,T}^{IFE}, \varphi_{i,T}^{IFE})\alpha_i +\sum_{i=1}^7(\boldsymbol{\Psi}_{i,T}, \psi_{i,T})\beta_i
\end{equation}
with
\begin{equation}\label{pro_alpha}
\alpha_i=N_{i,T}(\boldsymbol{\Xi}_h,  \xi_h), ~i=1,...,7
\end{equation}
and
\begin{equation}\label{pro_beta}
\begin{aligned}
&\beta_1=[\![ \boldsymbol{\Xi}_{h}^\pm]\!](\mathbf{x}_T)\cdot\mathbf{n}_h,~&&\beta_2=[\![\boldsymbol{\Xi}_{h}^\pm]\!](\mathbf{x}_T)\cdot\mathbf{t}_h,\\
&\beta_3=[\![\sigma(\mu^\pm, \boldsymbol{\Xi}_{h}^{\pm}, \xi_{h}^{\pm})\mathbf{n}_h]\!]\cdot\mathbf{n}_h,~&&\beta_4=[\![\sigma(\mu^\pm,  \boldsymbol{\Xi}_{h}^{\pm}, \xi_{h}^{\pm})\mathbf{n}_h]\!]\cdot\mathbf{t}_h,\\
&\beta_5=[\![\nabla \boldsymbol{\Xi}_{h}^{\pm}\mathbf{t}_h]\!]\cdot\mathbf{n}_h,~&&\beta_6=[\![\nabla  \boldsymbol{\Xi}_{h}^{\pm}\mathbf{t}_h]\!]\cdot\mathbf{t}_h, \qquad\beta_7=[\![\nabla\cdot  \boldsymbol{\Xi}_{h}^{\pm}]\!].
\end{aligned}
\end{equation}

Next, we prove $(\boldsymbol{\Xi}_h,  \xi_h)=(\widehat{\boldsymbol{\Xi}}_h, \widehat{\xi}_h)$. Using the fact that the IFE basis function $(\boldsymbol{\phi}_{i,T}^{IFE}, \varphi_{i,T}^{IFE})$ and the constructed function $(\boldsymbol{\Psi}_{i,T}, \psi_{i,T})$ satisfy the interface jump conditions (\ref{jp_cond_IFE1})-(\ref{jp_cond_IFE3}) and  (\ref{aux2}) respectively,  we have from (\ref{xihat}) and (\ref{pro_beta}) that   
\begin{equation*}
\begin{aligned}
&[\![(\widehat{\boldsymbol{\Xi}}_h-\boldsymbol{\Xi}_{h})^\pm]\!](\mathbf{x}_T)\cdot\mathbf{n}_h=0,~[\![(\widehat{\boldsymbol{\Xi}}_h-\boldsymbol{\Xi}_{h})^\pm]\!](\mathbf{x}_T)\cdot\mathbf{t}_h=0,\\
&[\![\sigma(\mu^\pm, (\widehat{\boldsymbol{\Xi}}_h-\boldsymbol{\Xi}_{h})^{\pm}, (\widehat{\xi}_h-\xi_{h})^{\pm})\mathbf{n}_h]\!]\cdot\mathbf{n}_h=0,~[\![\sigma(\mu^\pm,  (\widehat{\boldsymbol{\Xi}}_h-\boldsymbol{\Xi}_{h})^{\pm}, (\widehat{\xi}_h-\xi_{h})^{\pm})\mathbf{n}_h]\!]\cdot\mathbf{t}_h=0,\\
&[\![\nabla (\widehat{\boldsymbol{\Xi}}_h-\boldsymbol{\Xi}_{h})^{\pm}\mathbf{t}_h]\!]\cdot\mathbf{n}_h=0,~[\![\nabla  (\widehat{\boldsymbol{\Xi}}_h-\boldsymbol{\Xi}_{h})^{\pm}\mathbf{t}_h]\!]\cdot\mathbf{t}_h=0, ~[\![\nabla\cdot  (\widehat{\boldsymbol{\Xi}}_h-\boldsymbol{\Xi}_{h})^{\pm}]\!]=0,
\end{aligned}
\end{equation*}
which implies 
\begin{equation*}
(\widehat{\boldsymbol{\Xi}}_h-\boldsymbol{\Xi}_h, \widehat{\xi}_h-\xi_h)\in\mathbf{V}M_h^{IFE}(T).
\end{equation*}
Similarly, from (\ref{def_IFEbasis}), (\ref{aux1}) and (\ref{xihat})-(\ref{pro_alpha}), we also have 
\begin{equation*}
N_{i,T}(\widehat{\boldsymbol{\Xi}}_h-\boldsymbol{\Xi}_h, \widehat{\xi}_h-\xi_h)=0,~~i=1,...,7.
\end{equation*}
In view of Remark~\ref{basiseq0} we therefore conclude  $(\boldsymbol{\Xi}_h-\widehat{\boldsymbol{\Xi}}_h,  \xi_h-\widehat{\xi}_h)=(\mathbf{0}, 0)$, i.e.,  $(\boldsymbol{\Xi}_h,  \xi_h)=(\widehat{\boldsymbol{\Xi}}_h, \widehat{\xi}_h)$.

Now it remains to calculate the constants $\alpha_i$ and $\beta_i$ in (\ref{pro_alpha})-(\ref{pro_beta}).
If we define a broken interpolation operator $\Pi_{h,T}^{BK}$ such that 
\begin{equation}\label{bk}
\left(\Pi_{h,T}^{BK}(\mathbf{v}, q)\right)|_{T_h^\pm}=\Pi_{h,T}(\mathbf{v}_E^\pm, q_E^\pm),
\end{equation}
then $(\boldsymbol{\Xi}_h,  \xi_h)$ defined in (\ref{xi}) can be written as 
\begin{equation*}
(\boldsymbol{\Xi}_h,  \xi_h)=\Pi_{h,T}^{BK}(\mathbf{v}, q)-\Pi_{h,T}^{IFE}(\mathbf{v}, q).
\end{equation*}
By (\ref{local_PIife}),  we can calculate $\alpha_i$ in (\ref{pro_alpha}) as
\begin{equation*}
\alpha_i=N_{i,T}(\Pi_{h,T}^{BK}(\mathbf{v}, q))-N_{i,T}(\Pi_{h,T}^{IFE}(\mathbf{v}, q))=N_{i,T}(\Pi_{h,T}^{BK}(\mathbf{v}, q))-N_{i,T}(\mathbf{v}, q),\quad i=1,...,7,
\end{equation*}
which together with (\ref{bk}), (\ref{operaPI_1}) and (\ref{loc_dof}) leads to (\ref{alpha}).  The results in (\ref{beta}) for $\beta_i$, $i=1,...,7$ are obtained by substituting (\ref{xi}) into (\ref{pro_beta}) and using the fact that $\Pi_{h,T}^{IFE}(\mathbf{v}, q)$ satisfies the interface jump conditions (\ref{jp_cond_IFE1})-(\ref{jp_cond_IFE3}). This completes the proof of the lemma.
\end{proof}

\begin{theorem}\label{lem_interIFE_up}
For any $(\mathbf{v},q)\in \widetilde{\boldsymbol{H}_2H_1}$, there exists a positive constant $C$ independent of $h_\Gamma$ and the interface location relative to the mesh such that
\begin{align}
&\sum_{T\in\mathcal{T}_h^\Gamma}|\mathbf{v}_{E}^\pm-(\Pi_{\mathbf{v},q}^{IFE}\mathbf{v})^\pm |^2_{H^m(T)}\leq Ch_\Gamma^{4-2m}(\|\mathbf{v}\|^2_{H^2(\Omega^+\cup\Omega^-)}+\|q\|^2_{H^1(\Omega^+\cup\Omega^-)}),~~m=0,1,\label{interIFE_u}\\
&\sum_{T\in\mathcal{T}_h^\Gamma}\|q_{E}^\pm- (\Pi_{\mathbf{v},q}^{IFE}q)^\pm\|^2_{L^2(T)}\leq Ch_\Gamma^{2}(\|\mathbf{v}\|^2_{H^2(\Omega^+\cup\Omega^-)}+\|q\|^2_{H^1(\Omega^+\cup\Omega^-)}).\label{interIFE_p}
\end{align}
\end{theorem}
\begin{proof}
On each interface element $T\in\mathcal{T}_h^\Gamma$, by  the triangle inequality, we have
\begin{equation}\label{pro_main1}
\begin{aligned}
&|\mathbf{v}_E^\pm-(\Pi_{\mathbf{v},q}^{IFE}\mathbf{v})^\pm|_{H^m(T)}\leq |\mathbf{v}_E^\pm-\boldsymbol{\pi}^{CR}_{h,T} \mathbf{v}_E^\pm|_{H^m(T)}+|\boldsymbol{\pi}^{CR}_{h,T} \mathbf{v}_E^\pm -(\Pi_{\mathbf{v},q}^{IFE}\mathbf{v})^\pm|_{H^m(T)},\\
&\|q_E^\pm-(\Pi_{\mathbf{v},q}^{IFE}q)^\pm\|_{L^2(T)}\leq \|q_E^\pm-\pi_{h,T}^{0}q_E^\pm\|_{L^2(T)}+\|\pi_{h,T}^{0} q_E^\pm -(\Pi_{\mathbf{v},q}^{IFE}q)^\pm\|_{L^2(T)}.
\end{aligned}
\end{equation}
The estimates of the first terms are standard,
\begin{equation}\label{pro_main2}
\begin{aligned}
&|\mathbf{v}_E^\pm-\boldsymbol{\pi}^{CR}_{h,T} \mathbf{v}_E^\pm|^2_{H^m(T)} \leq Ch_T^{4-2m}| \mathbf{v}_E^\pm|^2_{H^2(T)},~~m=0,1,\\
&\|q_E^\pm-\pi_{h,T}^{0}q_E^\pm\|^2_{L^2(T)}\leq Ch_T^2|q_E^\pm|^2_{H^1(T)}.
\end{aligned}
\end{equation}
For the second term on the right-hand side of (\ref{pro_main1}),  we use (\ref{operaPI_1}), (\ref{Pi_IFE_2}) and Lemmas~\ref{lema_fenjie}, \ref{lem_basisest} and \ref{lem_auxest} to get
\begin{equation}\label{pro_main3}
\begin{aligned}
&|\boldsymbol{\pi}^{CR}_{h,T} \mathbf{v}_E^\pm -(\Pi_{\mathbf{v},q}^{IFE}\mathbf{v})^\pm|^2_{H^m(T)}\leq C\left(\sum_{i=1}^7\alpha_i^2|(\boldsymbol{\phi}_{i,T}^{IFE})^\pm|^2_{H^m(T)} +\sum_{i=1}^7 \beta_i^2|\boldsymbol{\Psi}_{i,T}^{\pm}|^2_{H^m(T)}\right)\\
&\qquad\qquad\qquad\qquad\leq C\sum_{i=1}^6 \alpha_i^2 h_T^{2-2m}+C\sum_{i=1}^2\beta_i^2 h_T^{2-2m}+C\sum_{i=4}^7\beta_i^2 h_T^{4-2m}, \\
&\|\pi_{h,T}^{0} q_E^\pm -(\Pi_{\mathbf{v},q}^{IFE}q)^\pm\|^2_{L^2(T)}\leq C\left(\sum_{i=1}^7\alpha_i^2 \|(\varphi_{i,T}^{IFE})^\pm\|^2_{L^2(T)} +\sum_{i=1}^7\beta_i^2\|\psi_{i,T}^{\pm}\|^2_{L^2(T)}\right)\\
&\qquad\qquad\qquad\qquad\leq C\sum_{i=1}^6 \alpha_i^2 +C\alpha_7^2h_T^2+C\sum_{i=1}^2\beta_i^2 +C\sum_{i=3}^7\beta_i^2 h_T^{2}, 
\end{aligned}
\end{equation}
where the constants $\alpha_i$ and $\beta_i$ are defined in (\ref{alpha}) and (\ref{beta}). Next, we estimate these constants one by one.  By the Cauchy-Schwarz inequality, we have
\begin{equation*}
\begin{aligned}
&\alpha_i^2=\frac{1}{|e_i|^2}\left|\sum_{s=\pm}\int_{e_i^s}(\pi_{h,T}^{CR}v_{E,1}^s-v_{E,1}^s)\right|^2\leq C|e_i|^{-1} \sum_{s=\pm} \|\pi_{h,T}^{CR}v_{E,1}^s-v_{E,1}^s\|^2_{L^2(e_i)}, ~ i=1,2,3,\\
&\alpha_{i+3}^2\leq C|e_i|^{-1} \sum_{i=\pm} \|\pi_{h,T}^{CR}v_{E,2}^s-v_{E,2}^s\|^2_{L^2(e_i)}, ~ i=1,2,3,\\
&\alpha_7^2\leq Ch_T^{-2}\sum_{s=\pm}\|\pi_{h,T}^0q_E^s-q_E^s\|^2_{L^2(T)}\textcolor{red}{+Ch_T^{-1}\sum_{s=\pm}\|q_E^s\|^2_{L^2(T^\triangle)}}.
\end{aligned}
\end{equation*}
By the standard trace inequality,  the standard  interpolation error estimates, \textcolor{red}{Lemma 2.5 in \cite{2021ji_nonconform} and the interface trace inequality (see \cite{hansbo2002unfitted,2012wuAn,2016High})}, it follows that 
\begin{equation}\label{est_alpha16}
\begin{aligned}
\alpha_i^2\leq Ch_T^2\sum_{s=\pm} |\mathbf{v}_E^s|^2_{H^2(T)},~i=1,...,6,\qquad \alpha_7^2&\leq C\sum_{s=\pm}\textcolor{red}{\|q_E^s\|^2_{H^1(T)}}.
\end{aligned}
\end{equation}
Since $(\mathbf{v},q)\in \widetilde{\boldsymbol{H}_2H_1}$,  the value $\mathbf{v}(\mathbf{x}_T)$ is well-defined and the identity  $[\![ \mathbf{v}_E^\pm]\!](\mathbf{x}_T)=0$ holds on the point $\mathbf{x}_T\in \Gamma_{h,T}\cap \Gamma_{T}$. Therefore, the constants $\beta_1$ and $\beta_2$ in (\ref{beta}) can be estimated as
\begin{equation}\label{est_beta12}
\begin{aligned}
\beta_i^2&\leq \left|[\![\boldsymbol{\pi}^{CR}_{h,T} \mathbf{v}_{E}^\pm]\!](\mathbf{x}_T)\right|^2=\left|[\![\boldsymbol{\pi}^{CR}_{h,T} \mathbf{v}_{E}^\pm-\mathbf{v}_{E}^\pm]\!](\mathbf{x}_T)\right|^2\leq \left\|[\![\boldsymbol{\pi}^{CR}_{h,T} \mathbf{v}_{E}^\pm-\mathbf{v}_{E}^\pm]\!]\right\|_{L^\infty(T)}^2 \\
&\leq C\sum_{s=\pm}\left\| \boldsymbol{\pi}^{CR}_{h,T} \mathbf{v}_{E}^s-\mathbf{v}_{E}^s \right\|_{L^\infty(T)}^2  \leq Ch_T^2\sum_{s=\pm} |\mathbf{v}_E^s|^2_{H^2(T)},~~i=1,2.
\end{aligned}
\end{equation}
To estimate $\beta_3$ and $\beta_4$, we use the following notations for simplicity
\begin{equation}\label{nota_new}
\sigma_{\boldsymbol{\pi}}^\pm:=\sigma(\mu^\pm, \boldsymbol{\pi}^{CR}_{h,T}\mathbf{v}_{E}^{\pm},\pi_{h,T}^0q_{E}^{\pm}),~~\sigma^\pm:=\sigma(\mu^\pm,\mathbf{v}_{E}^{\pm}, q_{E}^{\pm}).
\end{equation}
Noticing  $\sigma_{\boldsymbol{\pi}}^\pm\mathbf{n}_h$ are constant vectors, we derive  from (\ref{beta}) that
\begin{equation*}
\begin{aligned}
\beta_4^2&=[\![\mathbf{t}_h^T\sigma_{\boldsymbol{\pi}}^\pm\mathbf{n}_h]\!]^2\leq Ch_T^{-2} \|[\![\mathbf{t}_h^T\sigma_{\boldsymbol{\pi}}^\pm\mathbf{n}_h]\!]\|^2_{L^2(T)}\\
&=Ch_T^{-2} \|[\![\mathbf{t}_h^T(\sigma_{\boldsymbol{\pi}}^\pm-\sigma^\pm)\mathbf{n}_h+(\mathbf{t}_h^T-\mathbf{t}^T)\sigma^\pm\mathbf{n}+\mathbf{t}_h^T\sigma^\pm(\mathbf{n}_h-\mathbf{n})+\mathbf{t}^T\sigma^\pm\mathbf{n}]\!]\|^2_{L^2(T)}\\
&\leq Ch_T^{-2}\sum_{s=\pm}\left( \|\sigma_{\boldsymbol{\pi}}^s-\sigma^s\|^2_{L^2(T)} + \|\mathbf{t}_h-\mathbf{t}\|^2_{L^\infty(T)}\|\sigma^s\|^2_{L^2(T)}\right.\\
&\qquad\qquad\left.+\|\mathbf{n}_h-\mathbf{n}\|^2_{L^\infty(T)}\|\sigma^s\|^2_{L^2(T)} \right)+Ch_T^{-2}\| [\![\sigma^\pm\mathbf{n}]\!]\|^2_{L^2(T)},
\end{aligned}
\end{equation*}
where $\|\sigma^s\| ^2_{L^2(T)}=\int_T |\sigma^s|^2$  with $|\sigma^s|=\sqrt{\sigma^s:\sigma^s}$ being the Frobenius norm for matrix. 
It follows from (\ref{error_nt}), (\ref{nota_new}),  and the standard interpolation error estimates that
\begin{equation}\label{est_beta4}
\beta_4^2\leq C\sum_{s=\pm}\sum_{i=1}^2\left( |\mathbf{v}_E^s|^2_{H^i(T)}+|q_E^s|^2_{H^{i-1}(T)}\right)+Ch_T^{-2}\| [\![\sigma(\mu^\pm,\mathbf{v}_{E}^{\pm}, q_{E}^{\pm})\mathbf{n}]\!]\|^2_{L^2(T)}.
\end{equation}
Analogously, we can estimate $\beta_i$, $i=3,5,6,7$ as
\begin{equation}\label{est_beta3567}
\begin{aligned}
&\beta_3^2\leq C\sum_{s=\pm}\sum_{i=1}^2\left( |\mathbf{v}_E^s|^2_{H^i(T)}+|q_E^s|^2_{H^{i-1}(T)}\right)+Ch_T^{-2}\| [\![\sigma(\mu^\pm,\mathbf{v}_{E}^{\pm}, q_{E}^{\pm})\mathbf{n}]\!]\|^2_{L^2(T)},\\
&\beta_i^2\leq C\sum_{s=\pm}\left( |\mathbf{v}_E^s|^2_{H^2(T)}+ |\mathbf{v}_E^s|^2_{H^1(T)}\right)+Ch_T^{-2} \|[\![\nabla \mathbf{v}_{E}^{\pm}\mathbf{t}]\!]\|^2_{L^2(T)},\quad i=5,6,\\
&\beta_7^2\leq C\sum_{s=\pm} |\mathbf{v}_E^s|^2_{H^2(T)}+Ch_T^{-2} \|[\![\nabla\cdot \mathbf{v}_{E}^{\pm}]\!]\|^2_{L^2(T)}.
\end{aligned}
\end{equation}
Substituting (\ref{est_alpha16})-(\ref{est_beta12}), (\ref{est_beta4})-(\ref{est_beta3567}) into (\ref{pro_main3}) and combining (\ref{pro_main1})-(\ref{pro_main2}), we have
\begin{equation}\label{pro_main_ls}
\begin{aligned}
&|\mathbf{v}_E^\pm -(\Pi_{\mathbf{v},q}^{IFE}\mathbf{v})^\pm|^2_{H^m(T)}\leq Ch_T^{4-2m}\sum_{s=\pm}(\|\mathbf{v}_E^s\|^2_{H^2(T)}+\|q_E^s\|^2_{H^1(T)})+Ch_T^{2-2m}\mathcal{J}(T),\\
&\|q_E^\pm -(\Pi_{\mathbf{v},q}^{IFE}q)^\pm\|^2_{L^2(T)}\leq Ch_T^{2}\sum_{s=\pm}(\|\mathbf{v}_E^s\|^2_{H^2(T)}+\|q_E^s\|^2_{H^1(T)})+C\mathcal{J}(T),
\end{aligned}
\end{equation}
where 
$$
\mathcal{J}(T):=\|[\![\nabla\cdot \mathbf{v}_{E}^{\pm}]\!]\|^2_{L^2(T)}+\| [\![\sigma(\mu^\pm,\mathbf{v}_{E}^{\pm}, q_{E}^{\pm})\mathbf{n}]\!]\|^2_{L^2(T)}+ \|[\![\nabla \mathbf{v}_{E}^{\pm}\mathbf{t}]\!]\|^2_{L^2(T)}.
$$
%
%
%
Since $(\mathbf{v},q)\in \widetilde{\boldsymbol{H}_2H_1}$, we know from the definition (\ref{def_h2h1}) that 
\begin{equation*}
\begin{aligned}
&[\![\sigma(\mu^\pm,\mathbf{v}_{E}^{\pm}, q_{E}^{\pm})\mathbf{n}]\!]|_\Gamma=\mathbf{0},~~\|[\![\nabla\cdot \mathbf{v}_{E}^{\pm}]\!]|_\Gamma=0,~~ [\![ \mathbf{v}_E^\pm ]\!]|_\Gamma=\mathbf{0} ~(\mbox{implying } [\![\nabla \mathbf{v}_E^\pm \textbf{t}]\!]|_\Gamma=\mathbf{0}). 
\end{aligned} 
\end{equation*}
Noticing that  $T\subset U(\Gamma,h_\Gamma)~\forall T\in\mathcal{T}_h^\Gamma$ from Assumption~\ref{assumption_delta}, by Lemma~\ref{strip} and the above identities, it holds 
\begin{equation}\label{newadded}
\begin{aligned}
\sum_{T\in\mathcal{T}_h^\Gamma}\mathcal{J}(T)&\leq\|[\![\nabla\cdot \mathbf{v}_{E}^{\pm}]\!]\|^2_{L^2(U(\Gamma,h_\Gamma))}+\| [\![\sigma(\mu^\pm,\mathbf{v}_{E}^{\pm}, q_{E}^{\pm})\mathbf{n}]\!]\|^2_{L^2(U(\Gamma,h_\Gamma))}+\|[\![\nabla \mathbf{v}_{E}^{\pm}\mathbf{t}]\!]\|^2_{L^2(U(\Gamma,h_\Gamma))}\\
&\leq Ch_\Gamma^2\sum_{s=\pm}\left(\|\mathbf{v}_E^s\|^2_{H^2(\Omega)}+\|q_E^s\|^2_{H^1(\Omega)} \right),
\end{aligned}
\end{equation}
where we note that the constant $C$ also depends on the curvature of $\Gamma$.  Summing up (\ref{pro_main_ls}) over all interface elements and  using (\ref{newadded}) and Lemma~\ref{lem_ext}, we obtain the desired  estimates (\ref{interIFE_u}) and (\ref{interIFE_p}).
\end{proof}

Now we are ready to prove the optimal approximation capabilities of the IFE space $\mathbf{V}M_h^{IFE}$,
where the error resulting from the mismatch of $\Gamma$ and $\Gamma_h$ is considered rigorously.
\begin{theorem}\label{lem_interIFE_up2}
For any $(\mathbf{v},q)\in \widetilde{\boldsymbol{H}_2H_1}$, there exists a positive constant $C$ independent of $h$ and the interface location relative to the mesh such that
\begin{align}
\sum_{T\in\mathcal{T}_h}|\mathbf{v}-\Pi_{\mathbf{v},q}^{IFE}\mathbf{v} |^2_{H^m(T)}&\leq Ch^{4-2m}(\|\mathbf{v}\|^2_{H^2(\Omega^+\cup\Omega^-)}+\|q\|^2_{H^1(\Omega^+\cup\Omega^-)}),~~m=0,1,\label{ultimate_ve}\\
\|q- \Pi_{\mathbf{v},q}^{IFE} q\|^2_{L^2(\Omega)}&\leq Ch^{2}(\|\mathbf{v}\|^2_{H^2(\Omega^+\cup\Omega^-)}+\|q\|^2_{H^1(\Omega^+\cup\Omega^-)}).\label{ultimate_pr}
\end{align}

\end{theorem}
\begin{proof}
It suffices to consider the interface elements.
On each interface element $T\in\mathcal{T}_h^\Gamma$,  in view of the relations $T=T^+\cup T^-$ and $T^s=(T^s\cap T_h^+)\cup (T^s\cap T_h^-)$, $s=\pm$, we have
\begin{equation}\label{pro_int3a}
\begin{aligned}
|\mathbf{v}-\Pi_{\mathbf{v},q}^{IFE}\mathbf{v}|^2_{H^m(T)}&= \sum_{s=\pm}|\mathbf{v}^s-(\Pi_{\mathbf{v},q}^{IFE}\mathbf{v})^s|^2_{H^m(T^s\cap T_h^s)} \\
&+|\mathbf{v}^--(\Pi_{\mathbf{v},q}^{IFE}\mathbf{v})^+|^2_{H^m(T^-\cap T_h^+)}+|\mathbf{v}^+-(\Pi_{\mathbf{v},q}^{IFE}\mathbf{v})^-|^2_{H^m(T^+ \cap T_h^-)}.
\end{aligned}
\end{equation}
By the triangle inequality, we further obtain
\begin{equation}\label{pro_int4}
\begin{aligned}
&|\mathbf{v}^--(\Pi_{\mathbf{v},q}^{IFE}\mathbf{v})^+|^2_{H^m(T^-\cap T_h^+)}\leq 2|\mathbf{v}^--\mathbf{v}_E^+|^2_{H^m(T^-\cap T_h^+)}+2|\mathbf{v}_E^+-(\Pi_{\mathbf{v},q}^{IFE}\mathbf{v})^+|^2_{H^m(T^-\cap T_h^+)},\\
&|\mathbf{v}^+-(\Pi_{\mathbf{v},q}^{IFE}\mathbf{v})^-|^2_{H^m(T^+\cap T_h^-)}\leq 2|\mathbf{v}^+-\mathbf{v}_E^-|^2_{H^m(T^+\cap T_h^-)}+ 2|\mathbf{v}_E^-- (\Pi_{\mathbf{v},q}^{IFE}\mathbf{v})^- |^2_{H^m(T^+\cap T_h^-)}.
\end{aligned}
\end{equation}
Substituting (\ref{pro_int4}) into (\ref{pro_int3a}) and using the definition (\ref{def_triangleT}), we  conclude, for all $T\in\mathcal{T}_h^\Gamma$,
\begin{equation*}
|\mathbf{v}-\Pi_{\mathbf{v},q}^{IFE}\mathbf{v}|^2_{H^m(T)} \leq  C\sum_{s=\pm}|\mathbf{v}^s-(\Pi_{\mathbf{v},q}^{IFE}\mathbf{v})^s|^2_{H^m(T)}+C|[\![\mathbf{v}_E^\pm]\!]|^2_{H^m(T^\triangle)},~m=0,1.
\end{equation*}
Analogously, for all $T\in\mathcal{T}_h^\Gamma$, it holds
\begin{equation*}
\|q-\Pi_{\mathbf{v},q}^{IFE}q\|^2_{L^2(T)} \leq  C\sum_{s=\pm}\|q^s-(\Pi_{\mathbf{v},q}^{IFE}q)^s\|^2_{L^2(T)}+C\|[\![q_E^\pm]\!]\|^2_{L^2(T^\triangle)}.
\end{equation*}
Summing up and using Theorem~\ref{lem_interIFE_up} and the relation (\ref{triT_rela}), we arrive at
\begin{equation}\label{pro_zhazhi_js}
\begin{aligned}
&\sum_{T\in\mathcal{T}_h}|\mathbf{v}-\Pi_{\mathbf{v},q}^{IFE}\mathbf{v}|^2_{H^m(T)} \leq Ch^{4-2m}(\|\mathbf{v}\|^2_{H^2(\Omega^+\cup\Omega^-)}+\|q\|^2_{H^1(\Omega^+\cup\Omega^-)})+C|[\![\mathbf{v}_E^\pm]\!]|^2_{H^m(U(\Gamma,Ch_\Gamma^2))},\\
&\|q- \Pi_{\mathbf{v},q}^{IFE} q\|^2_{L^2(\Omega)}\leq Ch^{2}(\|\mathbf{v}\|^2_{H^2(\Omega^+\cup\Omega^-)}+\|q\|^2_{H^1(\Omega^+\cup\Omega^-)})+C\|[q_E^\pm]\|^2_{L^2(U(\Gamma,Ch_\Gamma^2))}.
\end{aligned}
\end{equation}
On the other hand, Lemma~\ref{strip} provides the following estimates
\begin{equation*}
\begin{aligned}
&\|[\![\mathbf{v}_E^\pm]\!]\|^2_{L^2(U(\Gamma,Ch_\Gamma^2))}\leq Ch_\Gamma^4\left|[\![\mathbf{v}_E^\pm]\!] \right|^2_{H^1(U(\Gamma,Ch_\Gamma^2))}\leq Ch_\Gamma^4\sum_{s=\pm}| \mathbf{v}_E^s |^2_{H^1(\Omega)},\\
&\| [\![\nabla\mathbf{v}_E^\pm]\!]\|^2_{L^2(U(\Gamma,Ch_\Gamma^2))}\leq Ch_\Gamma^2\left\|[\![\nabla\mathbf{v}_E^\pm]\!] \right\|^2_{H^1(\Omega)}\leq Ch_\Gamma^2\sum_{s=\pm}\| \mathbf{v}_E^s \|^2_{H^2(\Omega)},\\
&\|[q_E^\pm]\|^2_{L^2(U(\Gamma,Ch_\Gamma^2))}\leq Ch_\Gamma^2\|[\![q_E^\pm]\!]\|^2_{H^1(\Omega)} \leq Ch_\Gamma^2\sum_{s=\pm}\|q_E^s\|^2_{H^1(\Omega)},
\end{aligned}
\end{equation*}
where the fact $[\![\mathbf{v}_E^\pm]\!]|_\Gamma=\mathbf{0}$ is used for proving the first inequality.
Substituting the above inequalities into (\ref{pro_zhazhi_js}) and using  Lemma~\ref{lem_ext}, 
we complete the proof of the theorem.
\end{proof}

\begin{remark}
As shown in Remark~\ref{basiseq0_affect},  the function $\Pi_{\mathbf{v},q}^{IFE}\mathbf{v} $  depends only on the velocity $\mathbf{v}$, not on the pressure $q$. Accordingly, we can remove the term $\|q\|^2_{H^1(\Omega^+\cup\Omega^-)}$ on the right-hand sides of the estimates (\ref{interIFE_u}) and (\ref{ultimate_ve}) in Theorems~\ref{lem_interIFE_up} and \ref{lem_interIFE_up2}. Indeed, given $(\mathbf{v},q)\in \widetilde{\boldsymbol{H}_2H_1}$, we can construct a new function $\tilde{q}$ such that 
\begin{equation}\label{cons_qprime}
(\mathbf{v},\tilde{q})\in \widetilde{\boldsymbol{H}_2H_1}~\mbox{ and }~\|\tilde{q} \|_{H^1(\Omega^+\cup\Omega^-)}\leq C\|\mathbf{v}\|_{H^2(\Omega^+\cup\Omega^-)},
\end{equation}
which enables us to  remove the term $\|q\|_{H^1(\Omega^+\cup\Omega^-)}$ in (\ref{ultimate_ve}) (similarly, in (\ref{interIFE_u})) as
\begin{equation*}
\begin{aligned}
\sum_{T\in\mathcal{T}_h}&|\mathbf{v}-\Pi_{\mathbf{v},q}^{IFE}\mathbf{v} |^2_{H^m(T)}=\sum_{T\in\mathcal{T}_h}|\mathbf{v}-\Pi_{\mathbf{v},\tilde{q}}^{IFE}\mathbf{v} |^2_{H^m(T)}\\
&\leq Ch^{4-2m}(\|\mathbf{v}\|^2_{H^2(\Omega^+\cup\Omega^-)}+\|\tilde{q}\|^2_{H^1(\Omega^+\cup\Omega^-)})\leq Ch^{4-2m}\|\mathbf{v}\|^2_{H^2(\Omega^+\cup\Omega^-)}.
\end{aligned}
\end{equation*}
The function $\tilde{q}$ is constructed as follows. Define $\tilde{q}|_{\Omega^\pm}:=\tilde{q}^\pm$ with $\tilde{q}^\pm$ satisfying 
\begin{equation*}
\tilde{q}^+=0 ~~\mbox{ and }~~ \Delta \tilde{q}^-=0 \mbox{ in } \Omega^-,\quad \tilde{q}^-|_{\Gamma}=-[\sigma(\mu,\mathbf{v},0)\mathbf{n}]_\Gamma\cdot\mathbf{n}.
\end{equation*}
It is easy to verify that  the condition (\ref{cons_qprime}) is satisfied.
\end{remark}

\section{Analysis of the IFE method}\label{sec_error}
For all $(\mathbf{v}, q) \in (\mathbf{V}, M)+\mathbf{V}M_{h,0}^{IFE}$,  we introduce the following mesh dependent norms
\begin{equation}\label{def_norms}
\begin{aligned}
&\|\mathbf{v} \|^2_{1,h}:=\sum_{T\in\mathcal{T}_h} | \mathbf{v}|^2_{H^1(T)},~~\interleave \mathbf{v} \interleave^2_{1,h}:=\sum_{T\in\mathcal{T}_h} \|\sqrt{2\mu_h} \boldsymbol{\epsilon}(\mathbf{v})\|^2_{L^2(T)}+\sum_{e\in\mathcal{E}_h}\frac{1}{|e|}\|[\mathbf{v}]_e \|_{L^2(e)}^2,\\
&\interleave \mathbf{v} \interleave^2_{*,h}:=\interleave \mathbf{v} \interleave^2_{1,h}+\sum_{e\in\mathcal{E}_h^\Gamma}|e|\|\{ 2\mu_h\boldsymbol{\epsilon}(\mathbf{v})\mathbf{n}_e\}_e \|_{L^2(e)}^2 +\sum_{e\in\mathcal{E}_h^\Gamma}\frac{\eta+1}{|e|}\|[\mathbf{v}]_e \|_{L^2(e)}^2,\\
&\|q\|^2_{*,pre}:=\|q\|_{L^2(\Omega)}^2+\sum_{e\in\mathcal{E}_h^\Gamma}|e|\|\{q\}_e\|^2_{L^2(e)},\\
&\|(\mathbf{v}, q)\|^2:= \|\mathbf{v} \|^2_{1,h}+\|q\|_{L^2(\Omega)}^2+ J_h(q,q),~~\|(\mathbf{v}, q)\|^2_*:= \interleave \mathbf{v} \interleave^2_{*,h}+\|q\|^2_{*,pre}+ J_h(q,q).
\end{aligned}
\end{equation}
As $\mathbf{v}|_{T}\in H^1(T)^2$ for all $T\in\mathcal{T}_h$ from (\ref{jp_cond_IFE2}), $\int_e[\mathbf{v}]_e=\mathbf{0}$ for all $e\in\mathcal{E}_h$ and $\int_{\partial\Omega} \mathbf{v}=\mathbf{0}$, the Poincar\'e-Friedrichs inequality for piecewise $H^1$ functions (see \cite{brenner2003poincare}) and the Korn inequality  for piecewise $H^1$ vector functions (see \cite{brenner2004korn})  imply 
\begin{equation}\label{korn_ineq}
\|\mathbf{v}\|^2_{L^2(\Omega)}\leq C\sum_{T\in\mathcal{T}_h}|\mathbf{v}|^2_{H^1(T)},~~ \sum_{T\in\mathcal{T}_h}|\mathbf{v}|^2_{H^1(T)}\leq C\sum_{T\in\mathcal{T}_h}\|\boldsymbol{\epsilon}(\mathbf{v})\|^2_{L^2(T)}+C\sum_{e\in\mathcal{E}_h}|e|^{-1}\|[\mathbf{v}]\|_{L^2(e)}^2.
\end{equation}
Hence, $\|\cdot\| $ and $\| \cdot \|_{*}$ are indeed norms for the space $(\mathbf{V}, M)+\mathbf{V}M_{h,0}^{IFE}$.
\subsection{Boundedness and coercivity}
It follows from the Cauchy-Schwarz inequality that the bilinear forms $a_h(\cdot,\cdot)$ and $b_h(\cdot,\cdot)$ are bounded, {\em i.e.}, 
\begin{equation}\label{bd}
a_h(\mathbf{u},\mathbf{v})\leq \interleave \mathbf{u} \interleave_{*,h}\interleave \mathbf{v} \interleave_{*,h}\qquad\mbox{ and }\qquad b_h(\mathbf{v}, q)\leq C_b\|\mathbf{v}\|_{1,h}\|q\|_{*,pre},
\end{equation}
where $C_b$ is a constant independent of $h$ and the interface location relative to the mesh.
Furthermore, by the definitions (\ref{def_Ah}) and (\ref{def_norms}) we have the following lemma.
\begin{lemma}\label{lem_bd_Ah}
For all $(\mathbf{u}, p)$ and $ (\mathbf{v}, q)$ belonging to  $(\mathbf{V}, M)+\mathbf{V}M_{h,0}^{IFE}$, it holds 
\begin{equation}\label{bd_Ah}
A_h(\mathbf{u},p;\mathbf{v},q)\leq C_{A}\|(\mathbf{u},p)\|_*\|(\mathbf{v},q)\|_*,
\end{equation}
where $C_{A}$ is a positive constant independent of $h$ and the interface location relative to the mesh.
\end{lemma}

To prove the coercivity of the bilinear form $a_h(\cdot,\cdot)$,  we need a trace inequality for IFE functions.
For all  $(\mathbf{v}_h, q_h) \in \mathbf{V}M^{IFE}_h(T)$ on an interface element $T\in\mathcal{T}_h^\Gamma$,  since $\mathbf{v}_h\in (H^1(T))^2$, we have the  standard trace inequality:
 $\| \mathbf{v}_h\|_{L^2(\partial T)}\leq C(h_T^{-1/2}\|\mathbf{v}_h\|_{L^2(T)}+h_T^{1/2} \|\nabla \mathbf{v}_h\|_{L^2(T)}).$
However, the standard  trace inequality cannot be applied to $\nabla \mathbf{v}_h$ directly  since the function  $\mathbf{v}_h$ no longer belongs to $(H^2(T))^2$.  We establish the trace inequality for IFE functions in the following lemma.
\begin{lemma}\label{trac_IFE}
For any interface element $T\in\mathcal{T}_h^\Gamma$,  there exists a positive constant $C$ independent of $h_T$ and the interface location relative to the mesh such that 
\begin{equation}\label{trac_IFE_inequality}
\|\nabla \mathbf{v}_h\|_{L^2(\partial T)}\leq Ch_T^{-1/2}\|\nabla \mathbf{v}_h\|_{L^2(T)}\quad \forall (\mathbf{v}_h, q_h) \in \mathbf{V}M^{IFE}_h(T).
\end{equation}
\end{lemma}
\begin{proof}
From Lemma~\ref{lem_IFEbasis} and the definition~(\ref{cons_j1j2}), we know   $\mathbf{v}_h=\boldsymbol{\pi}_{h,T}^{CR}\mathbf{v}_h+c_2(w-\pi_{h,T}^{CR}w)\mathbf{t}_h$ with $w$ and $c_2$ defined in (\ref{def_zw}) and (\ref{ex_c1c2}), respectively. Using the facts  $\boldsymbol{\pi}_{h,T}^{CR}\mathbf{v}_h \in P_1(T)^2$, $\pi_{h,T}^{CR}w\in P_1(T)$, $|\nabla w^+|=1$,  and  (\ref{est_piw}),  we have
\begin{equation*}
\begin{aligned}
\|\nabla \mathbf{v}_h\|_{L^2(\partial T)}&\leq \|\nabla \boldsymbol{\pi}_{h,T}^{CR}\mathbf{v}_h \|_{L^2(\partial T)} +|c_2|\left(\|\nabla \pi_{h,T}^{CR}w \|_{L^2(\partial T)}+ \|\nabla w^+\|_{L^2(\partial T)}\right)\\
&\leq Ch_T^{-1/2}\|\nabla \boldsymbol{\pi}_{h,T}^{CR} \mathbf{v}_h\|_{L^2(T)}+C|c_2|h_T^{1/2}.
\end{aligned}
\end{equation*}
From  (\ref{ex_c1c2}) and (\ref{pro_basi_fenmu}), the constant $|c_2|$ can be estimated as 
\begin{equation*}
|c_2|=\left|\frac{\sigma(\mu^-/\mu^+-1, \boldsymbol{\pi}_{h,T}^{CR} \mathbf{v}_h, 0)\mathbf{n}_h\cdot\mathbf{t}_h}{1+(\mu^-/\mu^+-1)\nabla\pi_{h,T}^{CR}w\cdot \mathbf{n}_h}\right|\leq C|\nabla \boldsymbol{\pi}_{h,T}^{CR} \mathbf{v}_h|.
\end{equation*}
Combining the above inequalities, we obtain 
\begin{equation}\label{pro_stabipih1}
\|\nabla \mathbf{v}_h\|_{L^2(\partial T)}\leq Ch_T^{-1/2}\|\nabla \boldsymbol{\pi}_{h,T}^{CR} \mathbf{v}_h\|_{L^2(T)}.
\end{equation}

Let $e_i$, $i=1,2,3$ be edges of $T$ and $\mathbf{v}_h=(v_{h1}, v_{h2})^T$, then $\boldsymbol{\pi}_{h,T}^{CR} \mathbf{v}_h=(\pi_{h,T}^{CR} v_{h1}, \pi_{h,T}^{CR} v_{h2})^T$. By choosing a constant $c_T=|T|^{-1}\int_T\pi_{h,T}^{CR} v_{h1}$ we can derive 
\begin{equation*}
\begin{aligned}
\|\nabla \pi_{h,T}^{CR} &v_{h1}\|_{L^2(T)}=\|\nabla \pi_{h,T}^{CR} (v_{h1}-c_T)\|_{L^2(T)}\leq \sum_{i=1}^3 \frac{1}{|e_i|}\left|\int_{e_i} (v_{h1}-c_T)\right||\lambda_i|_{H^1(T)}\\
&\leq C\sum_{i=1}^3 h_T^{-1/2}\|v_{h1}-c_T\|_{L^2(e_i)}\leq C\left(h_T^{-1}\| v_{h1}-c_T\|_{L^2(T)}+|v_{h1}|_{H^1(T)}\right)\leq C|v_{h1}|_{H^1(T)},
\end{aligned}
\end{equation*}
which together with a similar estimate for $\pi_{h,T}^{CR} v_{h2}$ implies 
\begin{equation}\label{pihvleqv}
\|\nabla \boldsymbol{\pi}_{h,T}^{CR} \mathbf{v}_h\|_{L^2(T)}\leq C\|\nabla\mathbf{v}_h\|_{L^2(T)}.
\end{equation} 
Substituting (\ref{pihvleqv}) into (\ref{pro_stabipih1}) we complete  the proof of the lemma.
\end{proof}

 The coercivity of $a_h(\cdot,\cdot)$ is shown in the following lemma.
\begin{lemma}\label{lem_cor}
There exists a positive constant $C_{a}$ independent of $h$ and the interface location relative to the mesh such that 
\begin{equation}\label{inequ_cor}
a_h(\mathbf{v}_h,\mathbf{v}_h)\geq C_{a}\|\mathbf{v}_h\|^2_{1,h}\qquad \forall (\mathbf{v}_h,q_h)\in\mathbf{V}M_{h,0}^{IFE}
\end{equation}
is true for  $\delta=-1$ with an arbitrary $\eta\geq 0$ and is true for $\delta=1$ with a  sufficiently large $\eta$. 
\end{lemma}
\begin{proof}
From (\ref{trac_IFE_inequality}), the Cauchy-Schwarz inequality and the relation $|\boldsymbol{\epsilon}(\mathbf{v})|\leq C|\nabla \mathbf{v}|$, we obtain 
\begin{equation*}
\begin{aligned}
\sum_{e\in\mathcal{E}_h^\Gamma}\int_e \{ 2\mu_h\boldsymbol{\epsilon}(\mathbf{v}_h)\mathbf{n}_e\}_e\cdot[\mathbf{v}_h]_e&\leq \left(C\sum_{e\in\mathcal{E}_h^\Gamma}|e|\|\{\nabla\mathbf{v}\}_e\|_{L^2(e)}^2\right)^{1/2}\left(\sum_{e\in\mathcal{E}_h^\Gamma}|e|^{-1}\|[\mathbf{v}]_e\|_{L^2(e)}^2\right)^{1/2}\\
&\leq \left(C_{1}\sum_{T\in\mathcal{T}_h^\Gamma}|\mathbf{v}|_{H^1(T)}^2\right)^{1/2}\left(\sum_{e\in\mathcal{E}_h^\Gamma}|e|^{-1}\|[\mathbf{v}]_e\|_{L^2(e)}^2\right)^{1/2}\\
&\leq \frac{\varepsilon C_{1}}{2}\sum_{T\in\mathcal{T}_h^\Gamma}|\mathbf{v}|_{H^1(T)}^2+\frac{1}{2\varepsilon}\sum_{e\in\mathcal{E}_h^\Gamma}|e|^{-1}\|[\mathbf{v}]_e\|_{L^2(e)}^2,
\end{aligned}
\end{equation*}
where the positive constant $C_{1}$ is independent of $h$ and the interface location relative to the mesh. 
By the second inequality in (\ref{korn_ineq}), there is another constant $C_2$ independent of $h$ and the interface location relative to the mesh such that 
\begin{equation}\label{pro_coer_delta}
\sum_{T\in\mathcal{T}_h}\int_T2\mu_h \boldsymbol{\epsilon}(\mathbf{v}_h):\boldsymbol{\epsilon}(\mathbf{v}_h)+\sum_{e\in\mathcal{E}_h}\frac{1}{|e|}\int_e[\mathbf{v}_h]_e\cdot[\mathbf{v}_h]_e\geq C_2\sum_{T\in\mathcal{T}_h}|\mathbf{v}|^2_{H^1(T)}.
\end{equation}
It then follows from (\ref{def_Ah})  that, for $\delta=1$,
\begin{equation*}
\begin{aligned}
a_h(\mathbf{v}_h,\mathbf{v}_h)\geq(C_2-\varepsilon C_1)\sum_{T\in\mathcal{T}_h}|\mathbf{v}|^2_{H^1(T)}+(\eta-\varepsilon^{-1}) \sum_{e\in\mathcal{E}_h^\Gamma}|e|^{-1}\|[\mathbf{v}]_e\|_{L^2(e)}^2,
\end{aligned}
\end{equation*}
which implies the coercivity (\ref{inequ_cor} ) with $C_a=2^{-1}C_2$ when choosing $\varepsilon=C_2(2C_1)^{-1}$ and $\eta\geq\varepsilon^{-1}$. And for $\delta=-1$, the result (\ref{inequ_cor}) is a direct consequence of $(\ref{pro_coer_delta})$.
\end{proof}

\subsection{Norm-equivalence for IFE functions}
In this subsection we show that the two norms $\|\cdot\| $ and $\| \cdot \|_{*}$ are equivalent for the coupled IFE functions. First we need the following result about the coupled velocity and pressure.
\begin{lemma}\label{lem_jumpe}
For all $e\in \mathcal{E}_h$, let $\mathcal{T}_h^e$ be the set of all elements in $\mathcal{T}_h$ having $e$ as an edge, then there exists a positive  constant $C$ independent of $h$ and the interface location relative to the mesh such that, for all $(\mathbf{v}_h, q_h)\in \mathbf{V}M_h^{IFE}$, 
\begin{align}
&|e|^{-1}\left\| [\mathbf{v}_h]_{e}\right \|^2_{L^2(e)}\leq C\sum_{T\in\mathcal{T}_h^e}| \mathbf{v}_h|^2_{H^1(T)}\qquad \forall e\in\mathcal{E}_h,\label{jumpe_t}\\
&|e|\|\{q_h\}_e\|^2_{L^2(e)}\leq C\sum_{T\in\mathcal{T}_h^e}\left(|\mathbf{v}_h|^2_{H^1(T)}+\|q_h\|^2_{L^2(T)}\right)\qquad\forall e\in\mathcal{E}_h^{\Gamma}.\label{jumpe_tpre}
\end{align}
\end{lemma}
\begin{proof}
If $e\in\mathcal{E}_h^{non}$ and $\mathcal{T}_h^e\subset \mathcal{T}_h^{non}$, the proof of (\ref{jumpe_t}) is standard. 
For other cases, noticing that $\mathbf{v}_h|_T\in H^1(T)^2$ from (\ref{jp_cond_IFE2}), we can prove (\ref{jumpe_t}) analogously; see Lemma~4.2 in  \cite{2021ji_nonconform}.
%

Next, we prove (\ref{jumpe_tpre}). For an interface element $T\in\mathcal{T}_h^\Gamma$, from Lemma~\ref{lem_IFEbasis},  the pressure can be written as 
\begin{equation*}
q_h=\pi_{h,T}^{0}q_h+c_1(z-\pi_{h,T}^0z) \quad\mbox{ with }\quad c_1=\sigma(\mu^--\mu^+,\boldsymbol{\pi}_{h,T}^{CR}\mathbf{v}_h, 0)\mathbf{n}_h\cdot\mathbf{n}_h,
\end{equation*}
where $z$ is defined in (\ref{def_zw}). Let $e$ be an edge of $T$. Using (\ref{est_piw}) we have
\begin{equation*}
\begin{aligned}
|e|\|q_h\|^2_{L^2(e)}\leq C \|\pi_{h,T}^{0}q_h\|^2_{L^2(T)}+C |\boldsymbol{\pi}_h^{CR}\mathbf{v}_h|^2_{H^1(T)} \leq C\|q_h\|^2_{L^2(T)}+C|\mathbf{v}_h|^2_{H^1(T)},
\end{aligned}
\end{equation*}
which implies the estimate (\ref{jumpe_tpre}).
\end{proof}

We now prove the norm-equivalence in the following lemma.
\begin{lemma}\label{lem_norm_eq }
There exists positive constants $c_0$ and $C_0$ independent of $h$ and the interface location relative to the mesh such that, for all $(\mathbf{v}_h,q_h)\in\mathbf{V}M_{h,0}^{IFE}$, 
\begin{equation}\label{norm_eq1}
c_0\|\mathbf{v}_h\|_{1,h}\leq \interleave \mathbf{v}_h \interleave_{*,h}\leq C_{0}\|\mathbf{v}_h\|_{1,h}
\end{equation}
and correspondingly, 
\begin{equation}\label{norm_all_eq}
c_0 \|(\mathbf{v}_h,q_h)\|\leq \|(\mathbf{v}_h,q_h)\|_*\leq C_0\|(\mathbf{v}_h,q_h)\|.
\end{equation}
\end{lemma}
\begin{proof}
The result (\ref{norm_eq1}) is obtained by using (\ref{def_norms}), (\ref{trac_IFE_inequality}),  (\ref{jumpe_t}) and the relation $|\boldsymbol{\epsilon}(\mathbf{v})|\leq C|\nabla \mathbf{v}|$. Combining (\ref{norm_eq1}), (\ref{jumpe_tpre}) and the definitions in (\ref{def_norms}), we proved  (\ref{norm_all_eq}).
\end{proof}

\subsection{The inf-sup stability}
In order to prove the stability, we first need to bound  the jump of the pressure by the coupled velocity.
\begin{lemma}
For any $T\in\mathcal{T}_h^\Gamma$, there exists a positive constant $C$  independent of $h_T$ and the interface location relative to the mesh such that
\begin{equation}\label{qhjmpt}
h_T\|[\![q_h^\pm]\!]\|^2_{L^2(\Gamma_{h,T})}\leq  C|\mathbf{v}_h|^2_{H^1(T)}\qquad \forall (\mathbf{v}_h, q_h)\in\mathbf{V}M_{h}^{IFE}(T).
\end{equation}
\end{lemma}
\begin{proof}
Using  Lemma~\ref{lem_IFEbasis} and  the facts that $[\![q^{J_1,\pm}]\!]=-1$ and $q^{J_0}\in P_0(T)$,  we have  
\begin{equation*}
[\![q_h^\pm]\!](\mathbf{x})= -\sigma(\mu^--\mu^+,\boldsymbol{\pi}_{h,T}^{CR}\mathbf{v}_h,0)\mathbf{n}_h\cdot\mathbf{n}_h\qquad \forall \mathbf{x}\in T.
\end{equation*}
We then obtain
\begin{equation*}
h_T\|[\![q_h^\pm]\!]\|^2_{L^2(\Gamma_{h,T})}\leq Ch_T|\Gamma_{h,T}||\nabla \boldsymbol{\pi}_{h,T}^{CR}\mathbf{v}_h|^2\leq C\|\nabla \boldsymbol{\pi}_{h,T}^{CR}\mathbf{v}_h\|^2_{L^2(T)}\leq C\|\nabla\mathbf{v}_h\|^2_{L^2(T)},
\end{equation*}
where we have used (\ref{pihvleqv}) in the last inequality. This completes the proof.
\end{proof}

We also need the stability of the IFE interpolation and some interpolation error estimates under the $H^1$-regularity. 
\begin{lemma}
For any $\mathbf{v}\in (H^1(T))^2$, there exists a positive constant $C$ independent of $h$ and the interface location relative to the mesh such that
\begin{align}
&|\Pi_{\mathbf{v},q}^{IFE}\mathbf{v}|_{H^1(T)}\leq C|\mathbf{v}|_{H^1(T)}\qquad \forall T\in\mathcal{T}_h^\Gamma,\label{sta_Pi}\\
&\|\mathbf{v}-\Pi_{\mathbf{v},q}^{IFE}\mathbf{v}\|_{L^2(T)}\leq Ch_T|\mathbf{v}|_{H^1(T)},~~|\mathbf{v}-\Pi_{\mathbf{v},q}^{IFE}\mathbf{v}|_{H^1(T)}\leq C|\mathbf{v}|_{H^1(T)}\qquad \forall T\in\mathcal{T}_h^\Gamma.\label{int_h1_IFE}
\end{align}
\end{lemma}
\begin{proof}
On an interface element $T\in\mathcal{T}_h^\Gamma$, it follows from  Lemma~\ref{lem_IFEbasis} and Remark~\ref{basiseq0_affect} that  
\begin{equation*}
\Pi_{\mathbf{v},q}^{IFE}\mathbf{v}=\boldsymbol{\pi}_{h,T}^{CR}\mathbf{v}+c_2(w-\pi_{h,T}^{CR}w)\mathbf{t}_h~~\mbox{with}~~c_2=\frac{\sigma(\mu^-/\mu^+-1,\boldsymbol{\pi}_{h,T}^{CR}\mathbf{v},0)\mathbf{n}_h\cdot\mathbf{t}_h}{1+(\mu^-/\mu^+-1)\nabla\pi_{h,T}^{CR}w\cdot \mathbf{n}_h},
\end{equation*}
where  $w$ is defined in (\ref{def_zw}). 
Similar to the proof of Lemma~\ref{trac_IFE},  we have
\begin{equation*}
|c_2|\leq C|\nabla\boldsymbol{\pi}_{h,T}^{CR}\mathbf{v}|, ~|\nabla w^+|=1,~w^-=0,~ |\pi_{h,T}^{CR}w|_{W_\infty^m(T)}\leq Ch^{1-m}_T,~ |\boldsymbol{\pi}_{h,T}^{CR}\mathbf{v}|_{H^1(T)}\leq |\mathbf{v}|_{H^1(T)}.
\end{equation*}
The result (\ref{sta_Pi}) then is obtained from
\begin{equation*}
\begin{aligned}
|\Pi_{\mathbf{v},q}^{IFE}\mathbf{v}|_{H^1(T)}&\leq |\boldsymbol{\pi}_{h,T}^{CR}\mathbf{v}|_{H^1(T)}+|c_2|\left(|w|_{H^1(T)}+|\pi_{h,T}^{CR}w|_{H^1(T)}\right)\\
&\leq |\boldsymbol{\pi}_{h,T}^{CR}\mathbf{v}|_{H^1(T)}+Ch_T|\boldsymbol{\pi}_{h,T}^{CR}\mathbf{v}|\leq C|\boldsymbol{\pi}_{h,T}^{CR}\mathbf{v}|_{H^1(T)}\leq C|\mathbf{v}|_{H^1(T)}.
\end{aligned}
\end{equation*}
From the definition (\ref{def_zw}), it is easy to verify  $\|w\|_{L^2(T)}\leq Ch_T^2$. Therefore,
\begin{equation*}
\begin{aligned}
\|\mathbf{v}-&\Pi_{\mathbf{v},q}^{IFE}\mathbf{v}\|_{L^2(T)}\leq \|\mathbf{v}-\boldsymbol{\pi}_{h,T}^{CR}\mathbf{v}\|_{L^2(T)}+|c_2|\left(\|w\|_{L^2(T)}+\|\pi_{h,T}^{CR}w\|_{L^2(T)}\right)\\
&\leq Ch_T|\mathbf{v}|_{H^1(T)}+Ch_T^2|\boldsymbol{\pi}_{h,T}^{CR}\mathbf{v}|\leq Ch_T|\mathbf{v}|_{H^1(T)}+Ch_T|\boldsymbol{\pi}_{h,T}^{CR}\mathbf{v}|_{H^1(T)}\leq  Ch_T|\mathbf{v}|_{H^1(T)},
\end{aligned}
\end{equation*}
which proves the first estimate of (\ref{int_h1_IFE}). The second estimate of (\ref{int_h1_IFE}) can be easily obtained by (\ref{sta_Pi}) and the triangle inequality.
\end{proof}

With these preparations, we now prove the inf-sup stability of the proposed IFE method.
\begin{lemma}
There exist a positive constant $C_{3}$ independent of $h$ and the interface location relative to the mesh such that, for all $(\mathbf{v}_h, q_h)\in\mathbf{V}M_{h,0}^{IFE}$,
\begin{equation}\label{infsup1}
C_3\|q_h\|_{L^2(\Omega)}\leq \sup_{(\widetilde{\mathbf{v}}_h, \widetilde{q}_h)\in\mathbf{V}M_{h,0}^{IFE}}\frac{b_h(\widetilde{\mathbf{v}}_h, q_h) }{\| \widetilde{\mathbf{v}}_h \|_{1,h}} + \left(\sum_{T\in\mathcal{T}_h^\Gamma}|\mathbf{v}_h|^2_{H^1(T)}\right)^{\frac{1}{2}}+J_h^{\frac{1}{2}}(q_h,q_h).
\end{equation}
\end{lemma}
\begin{proof}
Let $(\mathbf{v}_h, q_h)\in\mathbf{V}M_{h,0}^{IFE}$. Since $q_h$ also belongs to the space $M$,  there is a function $\widetilde{\mathbf{v}}\in \mathbf{V} $ satisfying 
\begin{equation}\label{pro_inf_ad1}
\nabla \cdot \widetilde{\mathbf{v}}=q_h~~\mbox{ and }~~ \|\widetilde{\mathbf{v}}\|_{H^1(\Omega)}\leq C\|q_h\|_{L^2(\Omega)}
\end{equation}
with a constant $C$ only depends on $\Omega$ (see Lemma 11.2.3 in \cite{brenner2008mathematical}). Applying the integration by parts, we find 
\begin{equation}\label{pro_qhinf1}
\begin{aligned}
\|q_h\|_{L^2(\Omega)}^2&=\int_{\Omega}q_h\nabla\cdot \widetilde{\mathbf{v}}=\sum_{e\in\mathcal{E}_h}\int_e[q_h]_e\widetilde{\mathbf{v}}\cdot\mathbf{n}_e-\sum_{T\in\mathcal{T}_h^\Gamma}\int_{\Gamma_{h,T}}[\![q_h^\pm]\!]\widetilde{\mathbf{v}}\cdot\mathbf{n}_h.
\end{aligned}
\end{equation}
Since the IFE interpolation function $\Pi_{\widetilde{\mathbf{v}},\widetilde{q}}^{IFE}\widetilde{\mathbf{v}}$ is continuous on the whole element $T$ and  independent of  the pressure $\widetilde{q}$ (see Remark~\ref{basiseq0_affect}),  we apply the integration by parts again to get
\begin{equation}\label{pro_bh_eq1}
\begin{aligned}
b_h&(\Pi_{\widetilde{\mathbf{v}},\widetilde{q}}^{IFE}\widetilde{\mathbf{v}},q_h)=-\sum_{T\in\mathcal{T}_h}\int_T q_h\nabla\cdot\Pi_{\widetilde{\mathbf{v}},\widetilde{q}}^{IFE}\widetilde{\mathbf{v}} +\sum_{e\in\mathcal{E}_h^\Gamma}\int_e\{q_h\}_e[\Pi_{\widetilde{\mathbf{v}},\widetilde{q}}^{IFE}\widetilde{\mathbf{v}}]_e\cdot\mathbf{n}_e\\
&=-\sum_{e\in\mathcal{E}_h}\int_e\left([q_h]_e\{\Pi_{\widetilde{\mathbf{v}},\widetilde{q}}^{IFE}\widetilde{\mathbf{v}}\}_e\cdot\mathbf{n}_e+\{q_h\}_e[\Pi_{\widetilde{\mathbf{v}},\widetilde{q}}^{IFE}\widetilde{\mathbf{v}}]_e\cdot\mathbf{n}_e\right)+\sum_{T\in\mathcal{T}_h^\Gamma}\int_{\Gamma_{h,T}}[\![q_h^\pm]\!]\Pi_{\widetilde{\mathbf{v}},\widetilde{q}}^{IFE}\widetilde{\mathbf{v}}\cdot\mathbf{n}_h\\
&~~~~~~~~~~~~~+\sum_{e\in\mathcal{E}_h^\Gamma}\int_e\{q_h\}_e[\Pi_{\widetilde{\mathbf{v}},\widetilde{q}}^{IFE}\widetilde{\mathbf{v}}]_e\cdot\mathbf{n}_e\\
&=-\sum_{e\in\mathcal{E}_h}\int_e[q_h]_e\{\Pi_{\widetilde{\mathbf{v}},\widetilde{q}}^{IFE}\widetilde{\mathbf{v}}\}_e\cdot\mathbf{n}_e+\sum_{T\in\mathcal{T}_h^\Gamma}\int_{\Gamma_{h,T}}[\![q_h^\pm]\!]\Pi_{\widetilde{\mathbf{v}},\widetilde{q}}^{IFE}\widetilde{\mathbf{v}}\cdot\mathbf{n}_h,
\end{aligned}
\end{equation}
where we have used the facts that   $\int_{e}[\Pi_{\widetilde{\mathbf{v}},\widetilde{q}}^{IFE}\widetilde{\mathbf{v}}]_e=\mathbf{0}$ for all $e\in\mathcal{E}_h$ and $\{q_h\}_e$ is a constant for all $e\in\mathcal{E}_h^{non}$.
Combining (\ref{pro_qhinf1})-(\ref{pro_bh_eq1}) and using the facts that $[q_h]_e$ is a constant for all $e\in\mathcal{E}_h^{non}$ and $\int_e (\widetilde{\mathbf{v}}-\Pi_{\widetilde{\mathbf{v}},\widetilde{q}}^{IFE}\widetilde{\mathbf{v}})|_T=\mathbf{0}$ for all $e\in\mathcal{E}_h^{non}$ with $e\subset \partial T$, we further have
\begin{equation}\label{pro_i0}
\begin{aligned}
&\|q_h\|_{L^2(\Omega)}^2=-b_h(\Pi_{\widetilde{\mathbf{v}},\widetilde{q}}^{IFE}\widetilde{\mathbf{v}},q_h) + \left(b_h(\Pi_{\widetilde{\mathbf{v}},\widetilde{q}}^{IFE}\widetilde{\mathbf{v}},q_h)+\int_{\Omega}q_h\nabla\cdot \widetilde{\mathbf{v}}\right)\\
&=-b_h(\Pi_{\widetilde{\mathbf{v}},\widetilde{q}}^{IFE}\widetilde{\mathbf{v}},q_h)+\sum_{e\in\mathcal{E}_h^{\Gamma}}\int_e[q_h]_e\{\widetilde{\mathbf{v}}-\Pi_{\widetilde{\mathbf{v}},\widetilde{q}}^{IFE}\widetilde{\mathbf{v}}\}_e\cdot\mathbf{n}_e-\sum_{T\in\mathcal{T}_h^\Gamma}\int_{\Gamma_{h,T}}[\![q_h^\pm]\!](\widetilde{\mathbf{v}}-\Pi_{\widetilde{\mathbf{v}},\widetilde{q}}^{IFE}\widetilde{\mathbf{v}})\cdot\mathbf{n}_h\\
&:=\MakeUppercase{\romannumeral1}_1+\MakeUppercase{\romannumeral1}_2+\MakeUppercase{\romannumeral1}_3.
\end{aligned}
\end{equation}
It follows from  (\ref{sta_Pi}) and (\ref{pro_inf_ad1}) that 
\begin{equation}\label{pro_i1}
\begin{aligned}
|\MakeUppercase{\romannumeral1}_1|&=\frac{|b_h(\Pi_{\widetilde{\mathbf{v}},\widetilde{q}}^{IFE}\widetilde{\mathbf{v}},q_h)|}{\|\Pi_{\widetilde{\mathbf{v}},\widetilde{q}}^{IFE}\widetilde{\mathbf{v}}\|_{1,h}}\|\Pi_{\widetilde{\mathbf{v}},\widetilde{q}}^{IFE}\widetilde{\mathbf{v}}\|_{1,h}\leq \left( \sup_{(\widetilde{\mathbf{v}}_h, \widetilde{q}_h)\in\mathbf{V}M_{h,0}^{IFE}}\frac{b_h(\widetilde{\mathbf{v}}_h, q_h) }{\| \widetilde{\mathbf{v}}_h \|_{1,h}}\right)C|\widetilde{\mathbf{v}}|_{H^1(\Omega)}\\
&\leq C \left( \sup_{(\widetilde{\mathbf{v}}_h, \widetilde{q}_h)\in\mathbf{V}M_{h,0}^{IFE}}\frac{b_h(\widetilde{\mathbf{v}}_h, q_h) }{\| \widetilde{\mathbf{v}}_h \|_{1,h}}\right) \|q_h\|_{L^2(\Omega)}.
\end{aligned}
\end{equation}
Since $\left(\Pi_{\widetilde{\mathbf{v}},\widetilde{q}}^{IFE}\widetilde{\mathbf{v}}\right)|_{T}\in H^1(T)^2$ for all $T\in\mathcal{T}_h^\Gamma$, we use the standard trace inequality and the interpolation estimates (\ref{int_h1_IFE}) to get 
\begin{equation}\label{pro_i2}
\begin{aligned}
|\MakeUppercase{\romannumeral1}_2|&\leq \left(\sum_{e\in\mathcal{E}_h^\Gamma}|e|\|[q_h]_e\|^2_{L^2(e)}\right)^{\frac{1}{2}}\left(\sum_{e\in\mathcal{E}_h^\Gamma}|e|^{-1}\left\|\{\widetilde{\mathbf{v}}-\Pi_{\widetilde{\mathbf{v}},\widetilde{q}}^{IFE}\widetilde{\mathbf{v}}\}_e \right\|^2_{L^2(e)}\right)^{\frac{1}{2}}\\
&\leq CJ^\frac{1}{2}_h(q_h,q_h)\left(\sum_{T\in\mathcal{T}_h^\Gamma}h_T^{-2}\|\widetilde{\mathbf{v}}-\Pi_{\widetilde{\mathbf{v}},\widetilde{q}}^{IFE}\widetilde{\mathbf{v}}\|^2_{L^2(T)}+|\widetilde{\mathbf{v}}-\Pi_{\widetilde{\mathbf{v}},\widetilde{q}}^{IFE}\widetilde{\mathbf{v}}|^2_{H^1(T)}\right)^{\frac{1}{2}}\\
&\leq CJ^\frac{1}{2}_h(q_h,q_h)|\widetilde{\mathbf{v}}|_{H^1(\Omega)}\leq CJ^{\frac{1}{2}}(q_h,q_h)\|q_h\|_{L^2(\Omega)}.
\end{aligned}
\end{equation}
Similarly, by (\ref{qhjmpt}) and the following interface trace inequality (see \cite{hansbo2002unfitted,2012wuAn,2016High})
$$\|v\|^2_{L^2(\Gamma_{h,T})}\leq C(h_T^{-1}\|v\|^2_{L^2(T)}+h_T|v|^2_{H^1(T)})\qquad \forall v\in H^1(T),$$
we can bound the third term  by 
\begin{equation}\label{pro_i3}
\begin{aligned}
|\MakeUppercase{\romannumeral1}_3|&\leq \left(\sum_{T\in\mathcal{T}_h^\Gamma}h_T\|[\![q_h]\!]\|^2_{L^2(\Gamma_{h,T})}\right)^{\frac{1}{2}}\left(\sum_{T\in\mathcal{T}_h^\Gamma}h_T^{-1}\left\|\widetilde{\mathbf{v}}-\Pi_{\widetilde{\mathbf{v}},\widetilde{q}}^{IFE}\widetilde{\mathbf{v}} \right\|^2_{L^2(\Gamma_{h,T})}\right)^{\frac{1}{2}}\\
&\leq C\left(\sum_{T\in\mathcal{T}_h^\Gamma}|\mathbf{v}_h|^2_{H^1(T)}\right)^{\frac{1}{2}}\left(\sum_{T\in\mathcal{T}_h^\Gamma}h_T^{-2}\|\widetilde{\mathbf{v}}-\Pi_{\widetilde{\mathbf{v}},\widetilde{q}}^{IFE}\widetilde{\mathbf{v}}\|^2_{L^2(T)}+|\widetilde{\mathbf{v}}-\Pi_{\widetilde{\mathbf{v}},\widetilde{q}}^{IFE}\widetilde{\mathbf{v}}|^2_{H^1(T)}\right)^{\frac{1}{2}}\\
&\leq C\left(\sum_{T\in\mathcal{T}_h^\Gamma}|\mathbf{v}_h|^2_{H^1(T)}\right)^{\frac{1}{2}}|\widetilde{\mathbf{v}}|_{H^1(\Omega)} \leq    C\left(\sum_{T\in\mathcal{T}_h^\Gamma}|\mathbf{v}_h|^2_{H^1(T)}\right)^{\frac{1}{2}}\|q_h\|_{L^2(\Omega)}.
\end{aligned}
\end{equation}
Substituting (\ref{pro_i1})-(\ref{pro_i3}) into (\ref{pro_i0}) we conclude the proof.
\end{proof}

\begin{theorem}\label{lem_inf_sup_ulti}
There exists a positive constant $C_{s}$ independent of $h$ and the interface location relative to the mesh such that
\begin{equation}\label{inf_sup_ulti}
C_{s}\|(\mathbf{v}_h, q_h)\|\leq \sup_{(\mathbf{w}_h, r_h)\in \mathbf{V}M_{h,0}^{IFE}}\frac{A_h(\mathbf{v}_h, q_h; \mathbf{w}_h, r_h)}{\|(\mathbf{w}_h, r_h)\|}\qquad \forall (\mathbf{v}_h, q_h)\in \mathbf{V}M_{h,0}^{IFE}.
\end{equation}
\end{theorem}
\begin{proof}
Let $(\mathbf{v}_h, q_h)\in \mathbf{V}M_{h,0}^{IFE}$.  Since $\mathbf{V}M_{h,0}^{IFE}$ is a finite-dimensional space, we assume that the supremum in (\ref{infsup1}) is achieved at $(\widetilde{\mathbf{v}}_h^*, \widetilde{q}_h^*)\in\mathbf{V}M_{h,0}^{IFE}$, i.e., 
\begin{equation}\label{pro_k}
\sup_{(\widetilde{\mathbf{v}}_h, \widetilde{q}_h)\in\mathbf{V}M_{h,0}^{IFE}}\frac{b_h(\widetilde{\mathbf{v}}_h, q_h) }{\interleave \widetilde{\mathbf{v}}_h \interleave_{1,h}} =\frac{b_h(\widetilde{\mathbf{v}}_h^*, q_h) }{\interleave \widetilde{\mathbf{v}}_h^* \interleave_{1,h}}=\frac{b_h(k\widetilde{\mathbf{v}}_h^*, q_h) }{\|q_h\|_{L^2(\Omega)}}\quad\mbox{ with }\quad k=\frac{\|q_h\|_{L^2(\Omega)}}{\interleave \widetilde{\mathbf{v}}_h^* \interleave_{1,h}}.
\end{equation}
Here the function $\widetilde{q}_h^*$ is not unique and will be specified latter.
Therefore, (\ref{infsup1}) becomes 
\begin{equation}\label{infsup2}
C_3\|q_h\|^2_{L^2(\Omega)}\leq b_h(k\widetilde{\mathbf{v}}_h^*, q_h)  + \left(\sum_{T\in\mathcal{T}_h^\Gamma}|\mathbf{v}_h|^2_{H^1(T)}\right)^{1/2}\|q_h\|_{L^2(\Omega)}+J_h^{\frac{1}{2}}(q_h,q_h)\|q_h\|_{L^2(\Omega)}.
\end{equation}

Before continuing, we discuss some properties of the coupled functions $\widetilde{\mathbf{v}}_h^*$ and $\widetilde{q}_h^*$. From Lemma~\ref{lem_IFEbasis} we know that $N_{7,T}(\widetilde{\mathbf{v}}_h^*,\widetilde{q}_h^*)$ does not affect  the function $\widetilde{\mathbf{v}}_h^*$. Thus, we let $N_{7,T}(\widetilde{\mathbf{v}}_h^*,\widetilde{q}_h^*)=0$ for all $T\in\mathcal{T}_h$. Obviously, $\widetilde{q}_h^*|_{T}=0$ for all $T\in\mathcal{T}_h^{non}$. On an interface element $T\in\mathcal{T}_h^{\Gamma}$, it follows from (\ref{explicit_formula})-(\ref{ex_c1c2}) that 
\begin{equation*}
\widetilde{q}_h^*|_{T}=
\left(\sigma(\mu^--\mu^+,\boldsymbol{\pi}_{h,T}^{CR}\widetilde{\mathbf{v}}_h^*,0)\mathbf{n}_h\cdot\mathbf{n}_h\right)q^{J_1}
\end{equation*}
with $q^{J_1}$ defined in (\ref{cons_j1j2}).
Let $e$ be an edge of $T$.  
Using the above identity  and (\ref{pihvleqv}) we can derive
\begin{equation*}
|e|\|\widetilde{q}_h^*\|^2_{L^2(e)}+\|\widetilde{q}_h^*\|^2_{L^2(T)}\leq Ch_T^2 |\nabla \boldsymbol{\pi}_{h,T}^{CR}\widetilde{\mathbf{v}}_h^*|^2 \leq C \|\nabla \boldsymbol{\pi}_{h,T}^{CR}\widetilde{\mathbf{v}}_h^*\|^2_{L^2(T)}\leq C \|\nabla \widetilde{\mathbf{v}}_h^*\|^2_{L^2(T)}.
\end{equation*}
Thus, there exists a constant $C_*$ independent of $h$ and the interface location relative to the mesh such that
\begin{equation}\label{cstar}
\|\widetilde{q}_h^*\|_{*,pre}\leq C_*\| \widetilde{\mathbf{v}}_h^*\|_{1,h}~~ \mbox{ and }~~ J^{\frac{1}{2}}_h(\widetilde{q}_h^*,\widetilde{q}_h^*)\leq C_*\| \widetilde{\mathbf{v}}_h^*\|_{1,h},
\end{equation}
which mean that $\widetilde{q}_h^*$ can be controlled by $\widetilde{\mathbf{v}}_h^*$ in a proper norm.

Now we estimate the first term on the right-hand side of (\ref{infsup2}). From (\ref{def_Ah}), (\ref{bd}), (\ref{norm_eq1}), (\ref{cstar}) and (\ref{cstar}), we have
\begin{equation*}
\begin{aligned}
&b_h(k\widetilde{\mathbf{v}}_h^*, q_h)= A_h(\mathbf{v}_h, q_h; k\widetilde{\mathbf{v}}_h^*, k\widetilde{q}_h^*)-a_h(\mathbf{v}_h, k\widetilde{\mathbf{v}}_h^*)+b_h(\mathbf{v}_h, k\widetilde{q}_h^*)- J_h(q_h, k\widetilde{q}_h^*)\\
&\leq A_h(\mathbf{v}_h, q_h; k\widetilde{\mathbf{v}}_h^*, k\widetilde{q}_h^*) + \interleave \mathbf{v}_h \interleave_{*,h}\interleave k\widetilde{\mathbf{v}}_h^* \interleave_{*,h}+C_b\|\mathbf{v}_h\|_{1,h}\|k\widetilde{q}_h^*\|_{*,pre}+J_h^{\frac{1}{2}}(q_h,q_h)J_h^{\frac{1}{2}}(k\widetilde{q}_h^*,k\widetilde{q}_h^*)\\
&\leq A_h(\mathbf{v}_h, q_h; k\widetilde{\mathbf{v}}_h^*, k\widetilde{q}_h^*) + C_0^{2}\| \mathbf{v}_h \|_{1,h} \|k\widetilde{\mathbf{v}}_h^* \|_{1,h}+C_*\left(C_b\|\mathbf{v}_h\|_{1,h}+ J_h^{\frac{1}{2}}(q_h,q_h) \right) \|k\widetilde{\mathbf{v}}_h^* \|_{1,h}.
\end{aligned}
\end{equation*}
Substituting the above inequality into (\ref{infsup2}), and using the arithmetic-geometric mean inequality: $ab\leq \frac{2}{C_3}a^2+\frac{C_3}{8}b^2$ and the fact  $\|k\widetilde{\mathbf{v}}_h^* \|_{1,h}=\|q_h\|_{L^2(\Omega)}$ from (\ref{pro_k}), we further have 
\begin{equation*}
\begin{aligned}
C_3\|q_h\|^2_{L^2(\Omega)}&\leq A_h(\mathbf{v}_h, q_h; k\widetilde{\mathbf{v}}_h^*, k\widetilde{q}_h^*) +  \frac{2C_0^{4}}{C_3}\|\mathbf{v}_h \|_{1,h}^2 + \frac{C_3}{8}\|q_h\|_{L^2(\Omega)}^2\\
&\qquad+\frac{2C_*^2C_b^2}{C_3}\|\mathbf{v}_h \|_{1,h}^2+ \frac{C_3}{8}\|q_h\|_{L^2(\Omega)}^2+ \frac{2C_*^2}{C_3}J_h(q_h, q_h)+\frac{C_3}{8}\|q_h\|_{L^2(\Omega)}^2\\
&\qquad+\frac{2}{C_3}\|\mathbf{v}_h \|_{1,h}^2+ \frac{C_3}{8}\|q_h\|_{L^2(\Omega)}^2+ \frac{2}{C_3}J_h(q_h, q_h)+\frac{C_3}{8}\|q_h\|_{L^2(\Omega)}^2,
\end{aligned}
\end{equation*}
which leads to
\begin{equation}\label{pro_qh2sss}
\begin{aligned}
\frac{3C_3}{8}\|q_h\|^2_{L^2(\Omega)}&\leq A_h(\mathbf{v}_h, q_h; k\widetilde{\mathbf{v}}_h^*, k\widetilde{q}_h^*) +  \frac{2C_0^{4}+2C_*^2C_b^2+2}{C_3}\|\mathbf{v}_h \|_{1,h}^2+ \frac{2C_*^2+2}{C_3}J_h(q_h, q_h).
\end{aligned}
\end{equation}
On the other hand, by Lemma~\ref{lem_cor} and the definition (\ref{def_Ah}), we  know
\begin{equation*}
J_h(q_h,q_h)+C_a\|\mathbf{v}_h\|^2_{1,h}\leq A_h(\mathbf{v}_h, q_h; \mathbf{v}_h, q_h).
\end{equation*}
Combining this with (\ref{pro_qh2sss})  we get 
\begin{equation}\label{pro_C2}
\begin{aligned}
C_4 \|(\mathbf{v}_h, q_h )\|^2&= C_4\left(J_h(q_h,q_h)+\|\mathbf{v}_h\|^2_{1,h} + \|q\|^2_{L^2(\Omega)}\right) \\
&\leq A_h(\mathbf{v}_h, q_h; \mathbf{v}_h + \theta k\widetilde{\mathbf{v}}_h^*, q_h + \theta k\widetilde{q}_h^*)
\end{aligned}
\end{equation}
with
\begin{equation*}
\theta=\min\left(\frac{C_3C_a}{2(2C_0^{4}+2C_*^2C_b^2+2)}, \frac{C_3}{2(2C_*^2+2)}  \right)\quad \mbox{and} \quad C_4=\min\left( \frac{3C_3\theta}{8}, \frac{1}{2}, \frac{C_a}{2} \right).
\end{equation*}
Since $(\mathbf{v}_h, q_h )\in \mathbf{V}M_{h,0}^{IFE}$ and $(k\widetilde{\mathbf{v}}_h^*, k\widetilde{q}_h^*)\in\mathbf{V}M_{h,0}^{IFE}$, it holds 
\begin{equation}\label{cost_in}
(\mathbf{v}_h + \theta k\widetilde{\mathbf{v}}_h^*, q_h + \theta k\widetilde{q}_h^*)=(\mathbf{v}_h, q_h )+\theta (k\widetilde{\mathbf{v}}_h^*, k\widetilde{q}_h^*)\in\mathbf{V}M_{h,0}^{IFE}.
\end{equation}
By (\ref{cstar}) and the fact $\|k\widetilde{\mathbf{v}}_h^* \|_{1,h}=\|q_h\|_{L^2(\Omega)}$ from (\ref{pro_k}), we see
\begin{equation*}
\|k\widetilde{\mathbf{v}}_h^* \|_{1,h}+\|k\widetilde{q}_h^*\|_{L^2(\Omega)}+ J^{\frac{1}{2}}_h(k\widetilde{q}_h^*,k\widetilde{q}_h^*)\leq (2C_*+1)\|k\widetilde{\mathbf{v}}_h^* \|_{1,h}=(2C_*+1)\|q_h\|_{L^2(\Omega)},
\end{equation*}
which leads to
\begin{equation*}
\|(k\widetilde{\mathbf{v}}_h^*, k\widetilde{q}_h^*)\|\leq \sqrt{3}(2C_*+1)\|(\mathbf{v}_h, q_h )\|.
\end{equation*}
Therefore, we have
\begin{equation}\label{adfad}
\|(\mathbf{v}_h + \theta k\widetilde{\mathbf{v}}_h^*, q_h + \theta k\widetilde{q}_h^*)\|\leq \|(\mathbf{v}_h, q_h )\|+\theta\|(k\widetilde{\mathbf{v}}_h^*, k\widetilde{q}_h^*)\|\leq \left(1+\sqrt{3}(2C_*+1)\theta \right)\|(\mathbf{v}_h, q_h )\|.
\end{equation}
Combining (\ref{pro_C2})-(\ref{adfad})  yields the desired result (\ref{inf_sup_ulti})  with $$C_s=\left(1+\sqrt{3}(2C_*+1)\theta\right)^{-1}C_4>0$$ which is independent of $h$ and the interface location relative to the mesh. 
\end{proof}

As a consequence of Theorem~\ref{lem_inf_sup_ulti}, the discrete problem (\ref{IFE_method}) is well-posed; see \cite{Brezzi} for example.

\subsection{A priori error estimates}
We first derive an optimal estimate for the IFE interpolation error in terms of the norm $\|\cdot\|_{*}$.
\begin{lemma}\label{ener_app}
Suppose $(\mathbf{v}, q)\in \widetilde{\boldsymbol{H}_2H_1}$, then there exists a constant $C$ independent of $h$ and the interface location relative to the mesh such that
\begin{equation*}
\|(\mathbf{v}, q)-\Pi_h^{IFE}(\mathbf{v},q)\|_{*} \leq Ch(\|\mathbf{v}\|_{H^2(\Omega^+\cup\Omega^-)}+\|q\|_{H^1(\Omega^+\cup\Omega^-)}).
\end{equation*}
\end{lemma}
\begin{proof}
It suffices to consider the interface edges.
Let $e^\pm=e\cap\Omega^\pm$.  The following inequality holds 
\begin{equation*}
\begin{aligned}
\|\{ 2\mu_h\boldsymbol{\epsilon}(\mathbf{v}-\Pi_{\mathbf{v},q}^{IFE}\mathbf{v})\mathbf{n}_e\}_e \|_{L^2(e)}^2 &= \sum_{s=\pm}\|\{ 2\mu_h\boldsymbol{\epsilon}(\mathbf{v}_E^s-(\Pi_{\mathbf{v},q}^{IFE}\mathbf{v})^s)\mathbf{n}_e\}_e \|_{L^2(e^s)}^2\\
&\leq \sum_{s=\pm}\|\{ 2\mu_h\boldsymbol{\epsilon}(\mathbf{v}_E^s-(\Pi_{\mathbf{v},q}^{IFE}\mathbf{v})^s)\mathbf{n}_e\}_e \|_{L^2(e)}^2,
\end{aligned}
\end{equation*}
which together with the standard trace inequality yields 
\begin{equation*}
\begin{aligned}
&\sum_{e\in\mathcal{E}_h^\Gamma}|e|\|\{ 2\mu_h\boldsymbol{\epsilon}(\mathbf{v}-\Pi_{\mathbf{v},q}^{IFE}\mathbf{v})\mathbf{n}_e\}_e \|_{L^2(e)}^2 \leq C\sum_{T\in\mathcal{T}_h^\Gamma}\sum_{s=\pm}\left(|\mathbf{v}_E^s-(\Pi_{\mathbf{v},q}^{IFE}\mathbf{v})^s|^2_{H^1(T)}+h_T^2|\mathbf{v}_E^s|^2_{H^2(T)}\right).
\end{aligned}
\end{equation*}
Analogously, we have
\begin{equation*}
\begin{aligned}
&\sum_{e\in\mathcal{E}_h^\Gamma}\frac{1}{|e|}\|[\mathbf{v}-\Pi_{\mathbf{v},q}^{IFE}\mathbf{v}]_e \|_{L^2(e)}^2\leq Ch_\Gamma^{2}\sum_{s=\pm}(\|\mathbf{v}_E^s\|^2_{H^2(\Omega)}+\|q_E^s\|^2_{H^1(\Omega)})+\sum_{T\in\mathcal{T}_h^\Gamma}\mathcal{J}(T),\\
&\sum_{e\in\mathcal{E}_h^\Gamma}|e|\|\{q- \Pi_{\mathbf{v},q}^{IFE} q\}_e\|^2_{L^2(e)}\leq C\sum_{T\in\mathcal{T}_h^\Gamma}\sum_{s=\pm}\left(\|q^s- (\Pi_{\mathbf{v},q}^{IFE} q)^s\|^2_{L^2(T)}+h_T^2|q^s|_{H^1(T)}^2\right),\\
&J_h(q- \Pi_{\mathbf{v},q}^{IFE} q, q- \Pi_{\mathbf{v},q}^{IFE} q)\leq C\sum_{T\in\mathcal{T}_h^\Gamma}\sum_{s=\pm}\left(\|q^s- (\Pi_{\mathbf{v},q}^{IFE} q)^s\|^2_{L^2(T)}+h_T^2|q^s|_{H^1(T)}^2\right),
\end{aligned}
\end{equation*}
where we have used (\ref{pro_main_ls}) in the first inequality.
Combining the above estimates, the inequality (\ref{newadded}) and Theorems~\ref{lem_interIFE_up} and \ref{lem_interIFE_up2} we complete the proof.
 \end{proof}
 
The following lemma concerns the consistent errors.
\begin{lemma}\label{lem_consis}
Let $(\mathbf{u}, p)$ and $(\mathbf{u}_h, p_h)$ be the solutions of the problems (\ref{weakform}) and (\ref{IFE_method}), respectively. Suppose $(\mathbf{u}, p)\in \widetilde{\boldsymbol{H}_2H_1}\cap (\mathbf{V}, M)$.
 Then,  there exists a constant $C$ independent of $h$ and the interface location relative to the mesh such that, for all  $(\mathbf{w}_h, r_h)\in \mathbf{V}M_h^{IFE}$,
\begin{equation*}
|A_h(\mathbf{u}-\mathbf{u}_h, p-p_h; \mathbf{w}_h,r_h)|\leq Ch\left(\|\mathbf{u}\|_{H^2(\Omega^+\cup\Omega^-)}+\|p\|_{H^1(\Omega^+\cup\Omega^-)}\right)\|(\mathbf{w}_h, r_h)\|.
\end{equation*}
\end{lemma}
 \begin{proof}
 Let $(\mathbf{w}_h, r_h)\in \mathbf{V}M_h^{IFE}$ be arbitrary and $\mathbf{n}_{\partial T}$ be the unit outward normal to $\partial T$.
Multiplying (\ref{originalpb1}) by $\mathbf{w}_h$ and applying integration by parts, we obtain  
\begin{equation*}
\int_\Omega \mathbf{f}\cdot\mathbf{w}_h=\sum_{T\in\mathcal{T}_h}\left(\int_T2(\mu\boldsymbol{\epsilon}(\mathbf{u})-p\mathbb{I}):\nabla \mathbf{w}_h-\int_{\partial T}\left(2\mu\boldsymbol{\epsilon}(\mathbf{u})-p\mathbb{I}\right)\mathbf{n}_{\partial T}\cdot\mathbf{w}_h \right),
\end{equation*}
where the integral on the interface $\Gamma$ is canceled due to the interface condition (\ref{jp_cond1}) and the fact that $\mathbf{w}_h|_{T}\in C^0(T)^2$ for all interface elements $T\in\mathcal{T}_h^\Gamma$. Since $(\mathbf{u}, p)\in \widetilde{\boldsymbol{H}_2H_1}$, we have $[2\mu\boldsymbol{\epsilon}(\mathbf{u})-p\mathbb{I}]_e\cdot\mathbf{n}_e=\mathbf{0}$ for all $e\in\mathcal{E}_h$, and 
\begin{equation}\label{pro_consis1}
\begin{aligned}
\int_\Omega \mathbf{f}\cdot\mathbf{w}_h&=\sum_{T\in\mathcal{T}_h}\int_T 2\mu\boldsymbol{\epsilon}(\mathbf{u}):\boldsymbol{\epsilon}(\mathbf{w}_h)-\sum_{T\in\mathcal{T}_h}\int_Tp\nabla\cdot\mathbf{w}_h \\
&\qquad+\sum_{e\in\mathcal{E}_h}\int_e \{p\}_e[\mathbf{w}_h\cdot\mathbf{n}_e]_e-\sum_{e\in\mathcal{E}_h}\int_e\{2\mu\boldsymbol{\epsilon}(\mathbf{u})\mathbf{n}_e\}_e\cdot[\mathbf{w}_h]_e.
\end{aligned}
\end{equation}
Subtracting (\ref{IFE_method})  from  (\ref{pro_consis1}) we further obtain 
\begin{equation*}
\begin{aligned}
&A_h(\mathbf{u}-\mathbf{u}_h, p-p_h; \mathbf{w}_h, r_h)=-\sum_{T\in\mathcal{T}_h^\Gamma}\int_{T^\triangle} 2(\mu-\mu_h)\boldsymbol{\epsilon}(\mathbf{u}):\boldsymbol{\epsilon}(\mathbf{w}_h) \\
&~~~-\sum_{e\in\mathcal{E}_h^{non}}\int_e p[\mathbf{w}_h\cdot\mathbf{n}_e]_e+\sum_{e\in\mathcal{E}_h^{non}}\int_e 2\mu\boldsymbol{\epsilon}(\mathbf{u})\mathbf{n}_e\cdot[\mathbf{w}_h]_e:=\MakeUppercase{\romannumeral2}_1+\MakeUppercase{\romannumeral2}_2+\MakeUppercase{\romannumeral2}_3,
\end{aligned}
\end{equation*}
where we have used the facts that $\int_{\Omega}r_h\nabla\cdot \mathbf{u}=0$ from (\ref{originalpb2}), $\mu|_e=\mu_h|_e$ for all $e\in\mathcal{E}_h$, and  $[p]_e=[\mathbf{u}]_e=0$ for all $e\in\mathcal{E}_h$ since $(\mathbf{u}, p)\in \widetilde{\boldsymbol{H}_2H_1}$.

We use (\ref{triT_rela}) and  Lemmas~\ref{strip} and \ref{lem_ext} to bound the first term  below
\begin{equation*}
\begin{aligned}
|\MakeUppercase{\romannumeral2}_1|&\leq \|2(\mu-\mu_h)\boldsymbol{\epsilon}(\mathbf{u})\|_{L^2(U(\Gamma,Ch_\Gamma^2))}\|\boldsymbol{\epsilon}(\mathbf{w}_h)\|_{L^2(U(\Gamma,Ch^2))}\leq C|\mathbf{u}|_{H^1(U(\Gamma,Ch_\Gamma^2))}\|\mathbf{w}_h\|_{1,h}\\
&\leq C\sum_{s=\pm}|\mathbf{u}_E^s|_{H^1(U(\Gamma,Ch_\Gamma^2))}\|\mathbf{w}_h\|_{1,h}\leq Ch_\Gamma\sum_{s=\pm}\|\mathbf{u}_E^s\|_{H^2(\Omega)}\|\mathbf{w}_h\|_{1,h}\\
&\leq Ch_\Gamma\sum_{s=\pm}\|\mathbf{u}^s\|_{H^2(\Omega^s)}\|\mathbf{w}_h\|_{1,h}\leq Ch\|\mathbf{u}\|_{H^2(\Omega^+\cup\Omega^-)} \|\mathbf{w}_h\|_{1,h}.
\end{aligned}
\end{equation*}
Let $\mathcal{T}_h^e$ be the set of all elements in $\mathcal{T}_h$ having $e$ as an edge. If $\mathcal{T}_h^e\cap \mathcal{T}_h^{non}\not=\emptyset$,  let $T_e\in \mathcal{T}_h^e\cap \mathcal{T}_h^{non}$.   Then, we have the standard result for the nonconforming finite elements (see, e.g., \cite{brenner2008mathematical})
\begin{equation*}
\left|\int_e p[\mathbf{w}_h\cdot\mathbf{n}_e]_e\right|\leq \inf_{c_e\in\mathbb{R}}\| p-c_e\|_{L^2(e)}\|[\mathbf{w}_h]_e\|_{L^2(e)}\leq C|p|_{H^1(T_e)}|e|^{1/2}\|[\mathbf{w}_h]_e\|_{L^2(e)}.
\end{equation*}
If $\mathcal{T}_h^e\cap \mathcal{T}_h^{non}=\emptyset$ (i.e., $\mathcal{T}_h^e\subset \mathcal{T}_h^{\Gamma}$), we have, for all $T\in \mathcal{T}_h^e$, 
\begin{equation*}
\left|\int_e p[\mathbf{w}_h\cdot\mathbf{n}_e]_e\right|\leq \sum_{s=\pm}\left|\int_e p_E^s[\mathbf{w}_h\cdot\mathbf{n}_e]_e\right| \leq C\sum_{s=\pm}|p_E^s|_{H^1(T)}|e|^{1/2}\|[\mathbf{w}_h]_e\|_{L^2(e)}.
\end{equation*}
Combining the above estimates with Lemmas~\ref{lem_ext} and \ref{lem_jumpe} we further get
\begin{equation*}
|\MakeUppercase{\romannumeral2}_2|\leq C\left(\sum_{s=\pm} |p_E^s|^2_{H^1(\Omega)}\right)^{1/2}\left(\sum_{e\in\mathcal{E}_h^{non}} |e|\|[\mathbf{w}_h]_e\|^2_{L^2(e)}\right)^{1/2}\leq Ch\|p\|_{H^1(\Omega^+\cup\Omega^-)} \|\mathbf{w}_h\|_{1,h}.
\end{equation*}
Analogously, we have the following estimate for the third term 
\begin{equation*}
|\MakeUppercase{\romannumeral2}_3| \leq Ch\|\mathbf{u}\|_{H^2(\Omega^+\cup\Omega^-)} \|\mathbf{w}_h\|_{1,h}.
\end{equation*}
This concludes the proof.
 \end{proof}

We now provide the error estimate for the proposed IFE method in the following theorem.
\begin{theorem}
Let $(\mathbf{u}, p)$ and $(\mathbf{u}_h, p_h)$ be the solutions of the problems (\ref{weakform}) and (\ref{IFE_method}), respectively. Suppose $(\mathbf{u}, p)\in \widetilde{\boldsymbol{H}_2H_1}\cap (\mathbf{V}, M)$, then the following error estimate holds
\begin{equation}\label{ulti_esti}
\| (\mathbf{u}, p)-(\mathbf{u}_h, p_h) \|_*\leq Ch\left(\|\mathbf{u}\|_{H^2(\Omega^+\cup\Omega^-)}+\|p\|_{H^1(\Omega^+\cup\Omega^-)}\right),
\end{equation}
where  the constant $C$ is independent of $h$ and the interface location relative to the mesh.
\end{theorem}
\begin{proof}
Using (\ref{norm_all_eq}) for the equivalence of two norms, the inf-sup stability (\ref{inf_sup_ulti})  and the continuity (\ref{bd_Ah}) of the bilinear form $A_h(\cdot,\cdot)$, we have, for all $(\mathbf{v}_h, q_h)\in \mathbf{V}M_h^{IFE}$,
\begin{equation*}
\begin{aligned}
\| (\mathbf{u}_h&, p_h)-(\mathbf{v}_h, q_h) \|_*\leq C_0\| (\mathbf{u}_h, p_h)-(\mathbf{v}_h, q_h) \| \\
&\leq C_0C_s^{-1} \sup_{(\mathbf{w}_h, r_h)\in \mathbf{V}M_{h,0}^{IFE}}\frac{A_h(\mathbf{u}_h-\mathbf{v}_h, p_h-q_h; \mathbf{w}_h, r_h)}{\|(\mathbf{w}_h, r_h)\|}\\
&= C_0C_s^{-1} \sup_{(\mathbf{w}_h, r_h)\in \mathbf{V}M_{h,0}^{IFE}}\frac{A_h(\mathbf{u}-\mathbf{v}_h, p-q_h; \mathbf{w}_h, r_h)+A_h(\mathbf{u}_h-\mathbf{u}, p_h-p; \mathbf{w}_h, r_h)}{\|(\mathbf{w}_h, r_h)\|}\\
&\leq C_AC_0^2C_s^{-1}\|(\mathbf{u}-\mathbf{v}_h, p-q_h)\|_*+C_0C_s^{-1} \sup_{(\mathbf{w}_h, r_h)\in \mathbf{V}M_{h,0}^{IFE}}\frac{A_h(\mathbf{u}_h-\mathbf{u}, p_h-p; \mathbf{w}_h, r_h)}{\|(\mathbf{w}_h, r_h)\|}.
\end{aligned}
\end{equation*}
It follows from Lemma~\ref{lem_consis} and the triangle inequality that, for all $(\mathbf{v}_h, q_h)\in \mathbf{V}M_h^{IFE}$,
\begin{equation*}
\begin{aligned}
\| (\mathbf{u}, p)-(\mathbf{u}_h, p_h) \|_*&\leq \| (\mathbf{u}, p)-(\mathbf{v}_h, q_h) \|_*+\| (\mathbf{u}_h, p_h)-(\mathbf{v}_h, q_h) \|_*\\
&\leq C\| (\mathbf{u}, p)-(\mathbf{v}_h, q_h) \|_*+Ch(\|\mathbf{u}\|_{H^2(\Omega^+\cup\Omega^-)}+\|p\|_{H^1(\Omega^+\cup\Omega^-)}).
\end{aligned}
\end{equation*}
Finally, the estimate (\ref{ulti_esti}) is obtained by choosing $(\mathbf{v}_h, q_h)=\Pi_h^{IFE}(\mathbf{u}, p)$ and Lemma~\ref{ener_app}.
\end{proof}

\section{Numerical experiments}\label{sec_num}
In this section, we present some numerical experiments to validate the theoretical analysis. Consider $\Omega=(-1,1)\times(-1,1)$ as the computational domain and use uniform triangulations constructed as follows. We first partition the domain into $N\times N$ congruent rectangles, and then obtain the triangulation by cutting the rectangles along one of diagonals in the same direction. 
The interface is $\Gamma=\{(x_1,x_2)^T\in \mathbb{R}^2: x_1^2+x_2^2=r_0^2\}$ with $r_0=0.5$ and  the exact solution $(\mathbf{u}, p)$ is given  for all $\mathbf{x}=(x_1, x_2)^T$ by
\begin{equation*}
\mathbf{u}(\mathbf{x})=\left\{
\begin{aligned}
&\frac{r_0^2-|\mathbf{x}|^2}{\mu^-}\left(
\begin{array}{c}
-x_2\\
x_1
\end{array}
\right)~~~\mbox{ if }~ |\mathbf{x}|<r_0,\\
&\frac{r_0^2-|\mathbf{x}|^2}{\mu^+}\left(
\begin{array}{c}
-x_2\\
x_1
\end{array}
\right)~~~\mbox{ if }~  |\mathbf{x}|\geq r_0,\\
\end{aligned}\right.
~~\mbox{and }~~p(\mathbf{x})=x_2^2-x_1^2.
\end{equation*}
 The right-hand side  $\mathbf{f}$ and the non-homogeneous Dirichlet boundary condition $\mathbf{u}|_{\partial \Omega}$ are determined  from  the exact solution.

We set  $\delta=-1$ and $\eta=0$ and use a standard approach from the finite element  framework to deal with the non-homogeneous Dirichlet boundary condition. The resulting systems of equations are solved by  a robust sparse direct solver from the MKL PARDISO package \cite{pardiso}. Note that the explicit formulas (\ref{exp_IFE_basis}) have been used to compute the IFE basis functions.  We denote the errors by $\|e_\mathbf{u}\|_{L^2}:=\|\mathbf{u}-\mathbf{u}_h\|_{L^2(\Omega)}$, $|e_\mathbf{u}|_{H^1}:=\|\mathbf{u}-\mathbf{u}_h\|_{1,h} $ and $\|e_p\|_{L^2}:=\|p-p_h\|_{L^2(\Omega)}$ and compute them experimentally on a sequence of uniform triangulations. We test the example with the viscosity coefficient  ranging from  small  to  large jumps: $\mu^+=5$, $\mu^-=1$; $\mu^+=1$, $\mu^-=5$; $\mu^+=1000$, $\mu^-=1$; $\mu^+=1$, $\mu^-=1000$. The errors and rates of convergence are listed in Tables~\ref{ex_biao1}-\ref{ex_biao4}. All data indicate that  the IFE method achieves the optimal convergence rates, which in turn confirms our theoretical analysis.

\begin{table}[H]
\caption{Errors of the IFE method  for the example with $\mu^+=5$, $\mu^-=1$.\label{ex_biao1}}
\begin{center}
{\small
\begin{tabular}{|c|c c|c c|c c|}
  \hline
       $N$  &  $\|e_\mathbf{u}\|_{L^2}$   &  rate   &  $|e_\mathbf{u}|_{H^1}$  &  rate  &  $\|e_p\|_{L^2}$  &  rate  \\ \hline
         8    &  1.001E-02    &              &  2.020E-01   &             &  2.476E-01   &      \\ \hline
        16   &   2.688E-03    &  1.90    &  1.065E-01    &  0.92   &   1.297E-01    &  0.93 \\ \hline
        32   &   6.821E-04    &  1.98    &  5.422E-02    &  0.97   &   6.154E-02    & 1.08 \\ \hline
        64   &   1.667E-04    &  2.03    &  2.722E-02    &  0.99   &   2.971E-02    &  1.05 \\ \hline
       128  &    4.216E-05   &   1.98   &   1.364E-02   &   1.00   &   1.459E-02   &   1.03 \\ \hline
       256   &   1.054E-05   &   2.00   &   6.826E-03   &   1.00   &   7.250E-03   &   1.01 \\ \hline
       512   &   2.642E-06   &   2.00   &   3.414E-03   &   1.00    &  3.614E-03   &   1.00 \\ \hline
\end{tabular}
}
\end{center}
\end{table}

\begin{table}[H]
\caption{Errors of the IFE method  for the example with $\mu^+=1$, $\mu^-=5$.\label{ex_biao2}}
\begin{center}
{\small
\begin{tabular}{|c|c c|c c|c c|}
  \hline
       $N$  &  $\|e_\mathbf{u}\|_{L^2}$   &  rate   &  $|e_\mathbf{u}|_{H^1}$  &  rate  &  $\|e_p\|_{L^2}$  &  rate  \\ \hline
        8    &  2.497E-02      &            &  6.643E-01     &            &  2.241E-01     &        \\ \hline
        16   &   6.419E-03    &  1.96    &  3.329E-01    &  1.00   &   1.172E-01   &   0.93 \\ \hline
        32   &   1.605E-03    &  2.00    &  1.667E-01    &  1.00    &  5.427E-02   &   1.11 \\ \hline
        64   &   3.997E-04    &  2.01    &  8.335E-02    &  1.00    &  2.653E-02   &   1.03 \\ \hline
       128  &    9.972E-05   &   2.00   &   4.169E-02    &  1.00   &   1.330E-02  &    1.00 \\ \hline
       256   &   2.490E-05   &   2.00   &   2.084E-02   &   1.00   &   6.631E-03  &    1.00 \\ \hline
       512   &   6.221E-06  &    2.00   &   1.042E-02    & 1.00    &  3.310E-03   &   1.00 \\ \hline
\end{tabular}
}
\end{center}
\end{table}

\begin{table}[H]
\caption{Errors of the IFE method  for the example with $\mu^+=1000$, $\mu^-=1$.\label{ex_biao3}}
\begin{center}
{\small
\begin{tabular}{|c|c c|c c|c c|}
  \hline
       $N$  &  $\|e_\mathbf{u}\|_{L^2}$   &  rate   &  $|e_\mathbf{u}|_{H^1}$  &  rate  &  $\|e_p\|_{L^2}$  &  rate  \\ \hline
         8     &  9.349E-03    &      &   1.228E-01     &   &     3.835E-01    &    \\ \hline
        16     &  2.906E-03     &  1.69     &  6.905E-02      & 0.83      & 3.490E-01      & 0.14 \\ \hline
        32     &  8.687E-04     &  1.74      & 3.752E-02      & 0.88      & 1.759E-01      & 0.99 \\ \hline
        64     &  1.971E-04     &  2.14      & 1.976E-02      & 0.92      & 9.581E-02      & 0.88 \\ \hline
       128     &  5.417E-05     &  1.86      & 1.100E-02     &  0.85      & 5.046E-02     &  0.93 \\ \hline
       256     &  1.402E-05     &  1.95      & 5.827E-03     &  0.92      & 1.979E-02     &  1.35 \\ \hline
       512     &  3.539E-06    &   1.99      & 2.981E-03    &   0.97      & 7.686E-03     &  1.36 \\ \hline
\end{tabular}
}
\end{center}
\end{table}

\begin{table}[H]
\caption{Errors of the IFE method  for the example with  $\mu^+=1$, $\mu^-=1000$.\label{ex_biao4}}
\begin{center}
{\small
\begin{tabular}{|c|c c|c c|c c|}
  \hline
       $N$  &  $\|e_\mathbf{u}\|_{L^2}$   &  rate   &  $|e_\mathbf{u}|_{H^1}$  &  rate  &  $\|e_p\|_{L^2}$  &  rate  \\ \hline
          8   &   2.517E-02   &            &    6.636E-01   &              &  2.275E-01   &          \\ \hline
        16    &  6.444E-03    &  1.97   &   3.329E-01     & 1.00    &  1.426E-01    &  0.67\\ \hline
        32    &  1.618E-03    &  1.99   &   1.667E-01    &  1.00    &  9.357E-02    &  0.61\\ \hline
        64    &  4.049E-04    &  2.00   &   8.336E-02    &  1.00    &  6.253E-02    &  0.58\\ \hline
       128   &   1.010E-04   &   2.00   &   4.169E-02   &   1.00   &   2.371E-02   &   1.40\\ \hline
       256   &   2.518E-05   &   2.00   &   2.084E-02   &   1.00    &  1.014E-02   &   1.23\\ \hline
       512   &   6.263E-06  &    2.01    &  1.042E-02   &   1.00    &  4.677E-03   &   1.12\\ \hline
\end{tabular}
}
\end{center}
\end{table}

 \section{Conclusions}\label{sec_con}
 In this paper we have developed and analyzed an IFE method for Stokes interface problems with discontinuous viscosity coefficients. The IFE space is constructed by modifying the traditional $CR$-$P_0$  finite element space.  We have shown the unisolvence of IFE basis functions and the optimal approximation capabilities of IFE space.  The stability and the optimal error estimates have been derived rigorously. This paper presents the first theoretical analysis for IFE methods for Stokes interface problems.
In the future we intend to study the Stokes interface problems with non-homogeneous jump conditions and construct IFE spaces for three-dimensional Stokes interface problems.

\section*{Acknowledgments}
H. Ji was partially supported by the National Natural Science Foundation of China (Grants Nos. 11701291, 12101327 and 11801281) and the Natural Science Foundation of Jiangsu Province (Grant No. BK20200848); F. Wang was partially supported by the National Natural Science Foundation of China (Grant Nos. 12071227 and 11871281) and the Natural Science Foundation of the Jiangsu Higher Education Institutions of China (Grant No. 20KJA110001);  J. Chen was partially supported by the National Natural Science Foundation of China (Grant Nos. 11871281, 11731007 and 12071227); Z. Li was partially supported by  a Simons grant (No. 633724).

\section*{Declaration of competing interest}
The authors declare that they have no known competing financial interests or personal relationships that could have appeared to influence the work reported in this paper.
 
\bibliographystyle{plain}

\begin{thebibliography}{10}

\bibitem{pardiso}
{Intel oneAPI Math Kernel Library PARDISO Solver - Parallel Direct Sparse
  Solver Interface,
  https://software.intel.com/content/www/us/en/develop/documentation/onemkl-developer-reference-c/top/sparse-solver-routines/onemkl-pardiso-parallel-direct-sparse-solver-interface.html.}

\bibitem{adjerid2015immersed}
S.~Adjerid, N.~Chaabane, and T.~Lin.
\newblock An immersed discontinuous finite element method for {Stokes}
  interface problems.
\newblock {\em Comput. Methods Appl. Mech. Engrg.}, 293:170--190, 2015.

\bibitem{adjerid2019immersed}
S.~Adjerid, N.~Chaabane, T.~Lin, and P.~Yue.
\newblock An immersed discontinuous finite element method for the {S}tokes
  problem with a moving interface.
\newblock {\em J. Comput. Appl. Math.}, 362:540--559, 2019.

\bibitem{brenner2003poincare}
S.~C. Brenner.
\newblock Poincar{\'e}-{F}riedrichs inequalities for piecewise {$H^1$}
  functions.
\newblock {\em SIAM J. Numer. Anal.}, 41:306--324, 2003.

\bibitem{brenner2004korn}
S.~C. Brenner.
\newblock Korn's inequalities for piecewise {$H^1$} vector fields.
\newblock {\em Math. Comp.}, 73:1067--1087, 2004.

\bibitem{brenner2008mathematical}
S.~C. Brenner and L.~R. Scott.
\newblock {\em The mathematical theory of finite element methods}.
\newblock Texts in Applied Mathematics 15, Springer, Berlin, 2008.

\bibitem{Brezzi}
F.~Brezzi and M.~Fortin.
\newblock {\em Mixed and Hybrid Finite Element Methods}.
\newblock Springer, Berlin, 1991.

\bibitem{burman2015cutfem}
E.~Burman, S.~Claus, P.~Hansbo, M.~G. Larson, and A.~Massing.
\newblock {CutFEM}: discretizing geometry and partial differential equations.
\newblock {\em Internat. J. Nume. Methods Engrg.}, 104:472--501, 2015.

\bibitem{caceres2020new}
E.~C{\'a}ceres, J.~Guzm{\'a}n, and M.~Olshanskii.
\newblock New stability estimates for an unfitted finite element method for
  two-phase {Stokes} problem.
\newblock {\em SIAM J. Numer. Anal.}, 58:2165--2192, 2020.

\bibitem{cattaneo2015stabilized}
L.~Cattaneo, L.~Formaggia, G.~F. Iori, A.~Scotti, and P.~Zunino.
\newblock Stabilized extended finite elements for the approximation of saddle
  point problems with unfitted interfaces.
\newblock {\em Calcolo}, 52:123--152, 2015.

\bibitem{chen2021p2}
Y.~Chen and X.~Zhang.
\newblock A {$P_2-P_1$} partially penalized immersed finite element method for
  {Stokes} interface problems.
\newblock {\em Int. J. Numer. Anal. Model.}, 18:120--141, 2021.

\bibitem{crouzeix1973conforming}
M.~Crouzeix and P.-A. Raviart.
\newblock {Conforming and nonconforming finite element methods for solving the
  stationary Stokes equations I}.
\newblock {\em ESAIM: Mathematical Modelling and Numerical
  Analysis-Mod{\'e}lisation Math{\'e}matique et Analyse Num{\'e}rique},
  7(R3):33--75, 1973.

\bibitem{foote1984regularity}
R.~L. Foote.
\newblock Regularity of the distance function.
\newblock {\em Proc. Amer. Math. Soc.}, 92:153--155, 1984.

\bibitem{Fries2010}
T.‐P. Fries and T.~Belytschko.
\newblock The extended/generalized finite element method: an overview of the
  method and its applications.
\newblock {\em Internat. J. Numer. Methods Engrg.}, 84:253--304, 2010.

\bibitem{Gilbargbook}
D.~Gilbarg and N.~S. Trudinger.
\newblock {\em Elliptic partial differential equations of second order.
  Classics in Mathematics. Reprint of the 1998 edition}.
\newblock Springer-Verlag, Berlin, 2001.

\bibitem{GuoIMA2019}
R.~Guo and T.~Lin.
\newblock A group of immersed finite-element spaces for elliptic interface
  problems.
\newblock {\em IMA J. Numer. Anal.}, 39:482--511, 2019.

\bibitem{Guojcp2020}
R.~Guo and T.~Lin.
\newblock An immersed finite element method for elliptic interface problems in
  three dimensions.
\newblock {\em J. Comput. Phys.}, 414:109478, 2020.

\bibitem{guo2020solving}
R.~Guo, Y.~Lin, and J.~Zou.
\newblock {Solving two dimensional H(curl)-elliptic interface systems with
  optimal convergence on unfitted meshes}.
\newblock {\em arXiv:2011.11905}, 2020.

\bibitem{guzman2018inf}
J.~Guzm{\'a}n and M.~Olshanskii.
\newblock Inf-sup stability of geometrically unfitted {Stokes} finite elements.
\newblock {\em Math. Comp.}, 87:2091--2112, 2018.

\bibitem{hansbo2002unfitted}
A.~Hansbo and P.~Hansbo.
\newblock An unfitted finite element method, based on {N}itsche{'}s method, for
  elliptic interface problems.
\newblock {\em Comput. Methods Appl. Mech. Engrg.}, 191:5537--5552, 2002.

\bibitem{hansbo2014cut}
P.~Hansbo, M.~G. Larson, and S.~Zahedi.
\newblock A cut finite element method for a {Stokes} interface problem.
\newblock {\em Appl. Numer. Math.}, 85:90--114, 2014.

\bibitem{he2012convergence}
X.~He, T.~Lin, and Y.~Lin.
\newblock The convergence of the bilinear and linear immersed finite element
  solutions to interface problems.
\newblock {\em Numer. Methods Partial Differential Equations}, 28:312--330,
  2012.

\bibitem{xiaoxiao2019stabilized}
X.~He, F.~Song, and W.~Deng.
\newblock Stabilized nonconforming {Nitsche's} extended finite element method
  for {Stokes} interface problems.
\newblock {\em arXiv:1905.04844}, 2019.

\bibitem{ito2006interface}
K.~Ito and Z.~Li.
\newblock Interface conditions for {Stokes} equations with a discontinuous
  viscosity and surface sources.
\newblock {\em Appl. Math. Lett.}, 19:229--234, 2006.

\bibitem{ji2020IRT}
H.~Ji.
\newblock {An immersed Raviart-Thomas mixed finite element method for elliptic
  interface problems on unfitted meshes}.
\newblock {\em J. Sci. Comput.}, 66:91, 2022.

\bibitem{2021ji_nonconform}
H.~Ji, F.~Wang, J.~Chen, and Z.~Li.
\newblock Analysis of nonconforming {IFE} methods and a new scheme for elliptic
  interface problems.
\newblock {\em arXiv:2108.03179}, 2021.

\bibitem{2021ji_IFE}
H.~Ji, F.~Wang, J.~Chen, and Z.~Li.
\newblock A new parameter free partially penalized immersed finite element and
  the optimal convergence analysis.
\newblock {\em Numer. Math.}, 150:1035–1086, 2022.

\bibitem{jones2021class}
D.~Jones and X.~Zhang.
\newblock A class of nonconforming immersed finite element methods for {S}tokes
  interface problems.
\newblock {\em J. Comput. Appl. Math.}, 392:113493, 2021.

\bibitem{kirchhart2016analysis}
M.~Kirchhart, S.~Gross, and A.~Reusken.
\newblock Analysis of an {XFEM} discretization for {Stokes} interface problems.
\newblock {\em SIAM J. Sci. Comput.}, 38:A1019--A1043, 2016.

\bibitem{lehrenfeld2012nitsche}
C.~Lehrenfeld and A.~Reusken.
\newblock {Nitsche-XFEM} with streamline diffusion stabilization for a
  two-phase mass transport problem.
\newblock {\em SIAM J. Sci. Comput.}, 34:A2740--A2759, 2012.

\bibitem{Li2010Optimal}
J.~Li, J.~Markus, B.~Wohlmuth, and J.~Zou.
\newblock Optimal a priori estimates for higher order finite elements for
  elliptic interface problems.
\newblock {\em Appl. Numer. Math.}, 60:19--37, 2010.

\bibitem{li1998immersed}
Z.~Li.
\newblock The immersed interface method using a finite element formulation.
\newblock {\em Appl. Numer. Math.}, 27:253--267, 1998.

\bibitem{li2007augmented}
Z.~Li, K.~Ito, and M.-C. Lai.
\newblock An augmented approach for {Stokes} equations with a discontinuous
  viscosity and singular forces.
\newblock {\em Comput. {$\&$} Fluids}, 36:622--635, 2007.

\bibitem{li2004immersed}
Z.~Li, T.~Lin, Y.~Lin, and R.~Rogers.
\newblock An immersed finite element space and its approximation capability.
\newblock {\em Numer. Methods Partial Differential Equations}, 20:338--367,
  2004.

\bibitem{Li2003new}
Z.~Li, T.~Lin, and X.~Wu.
\newblock New {C}artesian grid methods for interface problems using the finite
  element formulation.
\newblock {\em Numer. Math.}, 96:61--98, 2003.

\bibitem{taolin2015siam}
T.~Lin, Y.~Lin, and X.~Zhang.
\newblock Partially penalized immersed finite element methods for elliptic
  interface problems.
\newblock {\em SIAM J. Numer. Anal.}, 53:1121--1144, 2015.

\bibitem{tan2009immersed}
Z.~Tan, D.~V. Le, K.~M. Lim, and B.~C. Khoo.
\newblock An immersed interface method for the incompressible {Navier-Stokes}
  equations with discontinuous viscosity across the interface.
\newblock {\em SIAM J. Sci. Comput.}, 31:1798--1819, 2009.

\bibitem{wang2019non}
N.~Wang and J.~Chen.
\newblock A nonconforming {Nitsche’s} extended finite element method for
  {Stokes} interface problems.
\newblock {\em J. Sci. Comput.}, 81:342–374, 2019.

\bibitem{wang2015new}
Q.~Wang and J.~Chen.
\newblock A new unfitted stabilized {Nitsche’s} finite element method for
  {Stokes} interface problems.
\newblock {\em Comput. Math. Appl.}, 70:820--834, 2015.

\bibitem{2012wuAn}
H.~Wu and Y.~Xiao.
\newblock An unfitted $hp$-interface penalty finite element method for elliptic
  interface problems.
\newblock {\em J. Comput. Math.}, 37:316--339, 2019.

\bibitem{2016High}
Y.~Xiao, J.~Xu, and F.~Wang.
\newblock High-order extended finite element methods for solving interface
  problems.
\newblock {\em Comput. Methods Appl. Mech. Engrg}, 364:112964, 2016.

\bibitem{ZHANG2020112889}
Q.~Zhang and I.~Babu{\v{s}}ka.
\newblock {A stable generalized finite element method (SGFEM) of degree two for
  interface problems}.
\newblock {\em Comput. Methods Appl. Mech. Engrg.}, 363:112889, 2020.

\end{thebibliography}

\end{document}